\documentclass[11pt]{IEEEtran}
\usepackage{dsfont}

\usepackage{hyperref}
\usepackage{mathrsfs}

\usepackage[usenames]{color}

\usepackage[color,final]{showkeys}

\usepackage{amsmath,amssymb,amsthm}
\usepackage{latexsym,amsfonts,amscd,amsxtra,amstext}
\usepackage{algorithm,algorithmic}
\usepackage{url,cite}
\usepackage{hyperref}

\usepackage{epsfig} 
\usepackage{threeparttable}
\usepackage{subfigure}

\allowdisplaybreaks

\usepackage{mathtools}
\usepackage[normalem]{ulem} 
\usepackage{enumerate}
\usepackage{amssymb}
\usepackage{booktabs}
\usepackage{multirow,color}
\usepackage{amsfonts,dsfont}
\usepackage{hyperref}
\usepackage{cite}
\usepackage{pifont}
\usepackage{cases}

\usepackage[color]{showkeys}
\definecolor{refkey}{gray}{0.8}
\definecolor{labelkey}{gray}{0.8}

\newcommand{\remove}[1]{{}}
\newcommand{\cut}[1]{}

\def\bE{\mathbb{E}}

\newcommand{\nn}{\nonumber}
\newcommand{\ba}{\left[ \begin{array}}
\newcommand{\ea}{\\ \end{array} \right]}
\newcommand{\qd}{\hfill{$\blacksquare$}}
\newcommand{\define}{\;\stackrel{\Delta}{=}\;}

\newcommand{\eq}[1]{\begin{align}#1\end{align}}
\def\tran{^{\mathsf{T}}}

\def\qed{\mbox{\rule[0pt]{1.3ex}{1.3ex}}}

\def\d{{\boldsymbol{d}}}

\def\u{{\boldsymbol{u}}}

\newcommand{\cp}{{\scriptstyle{\mathcal{P}}}}
\newcommand{\sw}{{\scriptstyle{\mathcal{W}}}}

\newcommand{\sx}{{\scriptstyle{\mathcal{X}}}}
\newcommand{\sy}{{\scriptstyle{\mathcal{Y}}}}

\newcommand{\sz}{{\scriptstyle{\mathcal{Z}}}}

\newcommand{\twd}{\widetilde{\scriptstyle{\mathcal{W}}}}
\newcommand{\tzd}{\widetilde{\scriptstyle{\mathcal{Z}}}}
\newcommand{\tcA}{\overline{{\mathcal{A}}}}
\newcommand{\tA}{\overline{A}}

\newcommand{\tyd}{\widetilde{\scriptstyle{\mathcal{Y}}}}
\newcommand{\nnb}{\nonumber \\}

\newcommand{\cA}{{\mathcal{A}}}
\newcommand{\cB}{{\mathcal{B}}}

\newcommand{\cD}{{\mathcal{D}}}
\newcommand{\cE}{{\mathcal{E}}}

\newcommand{\cH}{{\mathcal{H}}}
\newcommand{\cI}{{\mathcal{I}}}
\newcommand{\cJ}{{\mathcal{J}}}

\newcommand{\cL}{{\mathcal{L}}}
\newcommand{\cM}{{\mathcal{M}}}
\newcommand{\cN}{{\mathcal{N}}}

\newcommand{\cP}{{\mathcal{P}}}
\newcommand{\cQ}{{\mathcal{Q}}}
\newcommand{\cR}{{\mathcal{R}}}
\newcommand{\cS}{{\mathcal{S}}}
\newcommand{\cT}{{\mathcal{T}}}

\newcommand{\cV}{{\mathcal{V}}}

\newcommand{\cX}{{\mathcal{X}}}

\newcommand{\RR}{\mathbb{R}}

\newcommand{\diag}{{\mathrm{diag}}} 
\newcommand{\col}{{\mathrm{col}}} 
\newcommand{\grad}{{\nabla}}    

\DeclareMathOperator*{\argmin}{arg\,min}
\DeclareMathOperator*{\argmax}{arg\,max}

\newcommand{\bc}{\begin{center}}
\newcommand{\ec}{\end{center}}

\newcommand{\bdm}{\begin{displaymath}}
\newcommand{\edm}{\end{displaymath}}

\newcommand{\beq}{\begin{equation}}
\newcommand{\eeq}{\end{equation}}

\newcommand{\bfl}{\begin{flushleft}}
\newcommand{\efl}{\end{flushleft}}

\newcommand{\bt}{\begin{tabbing}}
\newcommand{\et}{\end{tabbing}}

\newcommand{\beqn}{\begin{align}}
\newcommand{\eeqn}{\end{align}}

\newcommand{\beqs}{\begin{align*}} 
\newcommand{\eeqs}{\end{align*}}  

\newtheorem{assumption}{Assumption}

\newcommand{\HRule}{\rule{\linewidth}{0.5mm}}

\newtheorem{lemma}{{Lemma}}
\newtheorem{theorem}{{Theorem}}

\def\tran{^{\mathsf{T}}}

\def\d{{\boldsymbol{d}}}

\def\u{{\boldsymbol{u}}}
\def\v{{\boldsymbol{v}}}

\def\bE{\mathbb{E}}

\def\real{{\mathbb{R}}}

\def\Zint{{\mathchoice{\setbox1=\hbox{\sf Z}\copy1\kern-.75\wd1\box1}
{\setbox1=\hbox{\sf Z}\copy1\kern-.75\wd1\box1}
{\setbox1=\hbox{\scriptsize\sf Z}\copy1\kern-.75\wd1\box1}
{\setbox1=\hbox{\scriptsize\sf Z}\copy1\kern-.75\wd1\box1}}}

\makeatletter
\def\hlinewd#1{%
  \noalign{\ifnum0=`}\fi\hrule \@height #1 \futurelet
   \reserved@a\@xhline}
\makeatother

\title{Exact Diffusion for Distributed Optimization and Learning --- Part II: Convergence Analysis}
\author{\IEEEauthorblockN{Kun Yuan, Bicheng Ying, Xiaochuan Zhao,  and Ali H. Sayed \vspace{-0.1cm}}
	
	\IEEEauthorblockA{
	}
	
	\thanks{\scriptsize{K. Yuan and B. Ying are with the Department of Electrical Engineering, University of California, Los Angeles, CA 90095 USA. Email:\{kunyuan, ybc, xiaochuanzhao\}@ucla.edu. X. Zhao is now with Goldman Sachs, NY. A. H. Sayed is with the School of Engineering, Ecole Polytechnique Federale de Lausanne (EPFL), Switzerland. Email: ali.sayed@epfl.ch. This  work was supported in part by NSF grants CCF-1524250 and ECCS-1407712. {A short conference version of the results from Parts I and II appear in the short conference publication \cite{yuan2017eusipco}.}}}
}

\begin{document}
%
\maketitle
\small
\begin{abstract}
Part I of this work \cite{yuan2017exact1} developed the exact diffusion algorithm to remove the bias that is characteristic of distributed solutions for deterministic optimization problems. The algorithm was shown to be applicable to a larger set of combination policies than earlier approaches in the literature. In particular, the combination matrices are not required to be doubly stochastic, which impose stringent conditions on the graph topology and communications protocol. In this Part II, we examine the convergence and stability properties of exact diffusion in some detail and establish its linear convergence rate. We also show that it has a wider stability range than the EXTRA consensus solution, meaning that it is stable for a wider range of step-sizes and can, therefore, attain faster convergence rates. Analytical examples and numerical simulations illustrate the theoretical findings. 
\end{abstract}
\begin{keywords}
distributed optimization, diffusion, consensus, exact convergence, left-stochastic matrix, doubly-stochastic matrix, balanced policy, Perron vector.
\end{keywords}

\section{Introduction and review of Part I\cite{yuan2017exact1}}
\setlength\abovedisplayskip{2.5pt}
\setlength\belowdisplayskip{2.5pt}
\setlength\abovedisplayshortskip{2.5pt}
\setlength\belowdisplayshortskip{2.5pt}
{
For ease of reference, we provide a brief review of the main construction from Part I \cite{yuan2017exact1}. We consider a collection of $N$ networked agents working cooperatively to solve an aggregate optimization problem of the form:
\eq{
	\label{prob-dist}
	w^\star = \argmin_{w\in \RR^M}\quad \cJ^\star(w)=\sum_{k=1}^{N} q_k J_k(w),
}
where the $\{q_k\}$ are positive weighting scalars, each $J_k(w)$ is convex and differentiable, and the aggregate cost $\cJ^\star(w)$ is strongly-convex. {When $q_1 \cdots = q_N$, problem \eqref{prob-dist} reduces to 
\eq{
	\label{prob-consensus}
	w^o = \argmin_{w\in \RR^M}\quad \cJ^o(w)=\sum_{k=1}^{N} J_k(w).
}
Problems of the type \eqref{prob-dist}--\eqref{prob-consensus} find applications in a wide range of areas including including wireless sensor networks \cite{braca2008running}, distributed adaptation
and estimation strategies \cite{sayed2014adaptation,sayed2014adaptive}, distributed statistical
learning\cite{chen2015dictionary,chouvardas2012sparsity} and clustering \cite{zhao2015distributed}.

{Various algorithms have been proposed to solve problem \eqref{prob-consensus} such as \cite{dimakis2010gossip,sardellitti2010fast,kar2011convergence,yuan2016convergence,mota2013d,tsianos2012push,xi2015linear,shi2014linear,mokhtari2015dqm,shi2015extra,nedich2016achieving,xu2015augmented,nedic2016geometrically}.  These algorithms either employ doubly-stochastic or right-stochastic combination matrices.} 
 In Part I\cite{yuan2017exact1}, we derived the exact diffusion strategy \eqref{adapt}--\eqref{combine}.
 \begin{table}[h]
 	\noindent \HRule\\
 	\noindent \textbf{\footnotesize Algorithm 1} {\footnotesize (Exact diffusion strategy for agent $k$)} \vspace{-2mm}\\ 
 	\HRule\\
 	\noindent \textbf{\vspace{0mm}{\scriptsize Setting:} 
 	}
 	{\scriptsize Let $\tA=(I_{N}+A)/2$, and $w_{k,\hspace{-0.2mm}-\hspace{-0.3mm}1}$ arbitrary. {\color{black}Set $\psi_{k,-1} = w_{k,-1}$.}}
 	\noindent \textbf{\vspace{0mm}{\scriptsize {\color{white}Setting:}} 
 	}
 	{\scriptsize {\color{black}Let $\mu_k = q_k \mu_o/p_k$}. }

 	\noindent 
 	\textbf{\vspace{-4mm} \hspace{-1.3mm}{\scriptsize Repeat for $i=0,1,2,\cdots$}}\\
 	{\footnotesize
 		\eq{
 			\psi_{k,i} &= w_{k,i-1} - \mu_k \grad J_k(w_{k,i-1}), \hspace{5mm} \mbox{\footnotesize (adaptation)} \label{adapt}\\
 			\phi_{k,i} &= \psi_{k,i} + w_{k,i-1}  - \psi_{k,i-1}, \hspace{8mm} \mbox{\footnotesize (correction)} \label{correct}\\
 			w_{k,i} &= \sum_{\ell\in \cN_k} \overline{a}_{\ell k} \phi_{\ell,i}. \hspace{2.32cm} \mbox{\footnotesize (combination)}  \label{combine}
 		}}
 		\HRule
 	\end{table}
%
%
%
 The matrix $A=[a_{\ell k}]$ in the table refers to the combination policy with $a_{\ell k}\geq 0$ denoting the weight that scales the data arriving from agent $\ell$ to agent $k$. The matrix $A$ is not required to be symmetric but is left-stochastic, i.e.,
\eq{
A\tran \mathds{1}_N = \mathds{1}_N
}
where $\mathds{1}_N$ refers to a column vector with all entries equal to one.  It is assumed that the network graph is strongly-connected, which translates into a primitive matrix $A$. This implies, in view of the Perron-Frobenius theorem \cite{sayed2014adaptation}, that there exists a Perron vector $p$  satisfying
\eq{
Ap=p,\;\;\;\mathds{1}_N\tran p=1,\;\;\;p \succ0.
}
Furthermore, it was argued in Eq.(11) of Part I\cite{yuan2017exact1} that given $q$ and $A$ (and hence $p$), one can always adjust $\{\mu_k\}_{k=1}^N$ and find a positive constant $\beta$ such that
\eq{\label{q-A}
	q = \beta \, \mbox{diag}\{\mu_1,\mu_2,\cdots,\mu_N\}p.\vspace{-4mm}
}
Let $P=\mbox{\rm diag}(p)$, the matrix $A$ is said to be {\em balanced} if 
\eq{\label{local-balance}
	PA\tran = AP.
}
We showed in Part I\cite{yuan2017exact1} that balanced left-stochastic matrices are common in practice and that condition \eqref{local-balance} endows $A$
%
%
with several useful properties that 
enabled the derivation of the above exact diffusion strategy, and which will be used again in this work to examine its convergence properties. 

The structure of the exact diffusion strategy listed in \eqref{adapt}--\eqref{combine} is very similar to the standard diffusion implementation \cite{sayed2014adaptive,sayed2014adaptation,chen2012diffusion,chen2015learning}, with the only difference being the addition of an extra correction step between the adaptation and combination steps. We can rewrite the recursions \eqref{adapt}--\eqref{combine} in an aggregate form by resorting to a block vector notation. First, we introduce the eigen-decomposition 
\eq{\label{syud}
	(P - AP)/{2} = U \Sigma U\tran,
}
where $\Sigma \in \RR^{N\times N}$ is a non-negative diagonal matrix and $U\in \RR^{N\times N}$ is an orthogonal matrix. Next, we select $V$ to be the symmetric square-root matrix defined as
\eq{\label{V-defi}
	V \define U \Sigma^{1/2} U\tran \in \RR^{N\times N},
}
and introduce the quantities: 
\eq{
\tA & = (A + I_N)/2, \quad \tcA  = \tA \otimes I_M,\\
\cP & = P \otimes I_M, \hspace{10mm} \cV = V \otimes I_M, \\
\sw_i &= \col\{w_{1,i}, \cdots, w_{N,i}\}, \\
\sy_i &= \col\{y_{1,i}, \cdots, y_{N,i}\}, \\
\cM &= \diag\{\mu_1 I_M, \cdots, \mu_N I_M\}, \\
\grad \cJ^o(\sw) &= \col\{\grad J_1(w_{1}), \cdots, \grad J_N(w_{N})\},\label{grad-J^o}\\
\grad \cJ^\star(\sw) &= \col\{q_1\grad J_1(w_{1}), \cdots, q_N\grad J_N(w_{N})\}. \label{J-star}
}

\noindent Using these variables, and was already explained in Part I\cite{yuan2017exact1}, the recursions \eqref{adapt}--\eqref{combine} can be rewritten in the following equivalent so-called primal-dual form:
\begin{equation}
\left\{
\begin{aligned}
\sw_i &= \tcA\tran \Big(\sw_{i\hspace{-0.3mm}-\hspace{-0.3mm}1}\hspace{-0.8mm}-\hspace{-0.8mm}\cM \grad \cJ^o(\sw_{i\hspace{-0.3mm}-\hspace{-0.3mm}1})\Big)\hspace{-0.8mm}-\hspace{-0.8mm}\cP^{-1}\cV \sy_{i-1} \label{zn-1}\\
\sy_i &= \sy_{i-1} + \cV \sw_i
\end{aligned}
\right.
\end{equation}
For the initialization, we set $y_{-1}=0$ and $\sw_{-1}$ to be any value, and hence for $i=0$ we have
\begin{equation}
\left\{
\begin{aligned}
\sw_0 &= \tcA\tran \Big(\sw_{\hspace{-0.3mm}-\hspace{-0.3mm}1}\hspace{-0.8mm}-\hspace{-0.8mm}\cM \grad \cJ^o(\sw_{\hspace{-0.3mm}-\hspace{-0.3mm}1})\Big), \label{zn-0}\\
\sy_0 &= \cV \sw_0. 
\end{aligned}
\right.
\end{equation}
The following auxiliary lemma, which was established in Part I\cite{yuan2017exact1}, is used in the subsequent convergence analysis.
\vspace{-1mm}
\begin{lemma}[\sc Nullspace of $V$]\label{lm:null-V}
	It holds that
	\eq{\label{null-V}
		\mathrm{null}(V) &= \mathrm{null}(P-AP) = \mathrm{span}\{\mathds{1}_N\}, \\
		\mathrm{null}( \cV ) &=  \mathrm{null}( \cP - \cA \cP ) = \mathrm{span}\{ \mathds{1}_N \otimes I_M \}.\label{xcn987-2}
	}
	\qd
\end{lemma}
\vspace{-2mm}
In this article, we will establish the linear convergence of exact diffusion using the primal-dual form \eqref{zn-1}. This is a challenging task due to the coupled dynamics among the agents. To facilitate the analysis, we first apply a useful coordinate transformation and characterize the error dynamics in this transformed domain. Then, we show analytically that exact diffusion is stable, converges linearly, and has a wider stability range than EXTRA consensus strategy\cite{shi2015extra}. We also compare the performance of exact diffusion to other existing linearly convergent algorithms besides EXTRA, such as DIGing\cite{nedich2016achieving} and Aug-DGM \cite{xu2015augmented,nedic2016geometrically} with numerical simulations.
}
\vspace{-2mm}

\section{Convergence of Exact Diffusion}\label{sec-convergence}
The purpose of the analysis in this section is to establish the exact convergence of $w_{k,i}$ to $w^{\star}$, for all agents in the network, and to show that this convergence attains an exponential rate. 

\vspace{-4mm}
\subsection{The Optimality Condition}\label{subsec-optimality}
\begin{lemma}[\sc Optimality Condition]\label{lm-opt-cond-1}
	If condition \eqref{q-A} holds and block vectors $(\sw^\star, \sy^\star)$ exist that satisfy:
	\eq{
		\tcA\tran \cM \grad \cJ^o(\sw^\star) + \cP^{-1}\cV \sy^\star & = 0, \label{KKT-1-1} \\
		\cV \sw^\star &= 0. \label{KKT-2-2}
	}
	then it holds that the block entries of ${\sw}^{\star}$ satisfy: 
	\eq{\label{7280m}
	w_1^\star=w_2^\star=\cdots=w_N^\star=w^\star，
	}
	where $w^\star$ is the unique solution to problem \eqref{prob-dist}.
\end{lemma}
\begin{proof}
	From \eqref{xcn987-2}, we have
	\eq{\label{xcn78}
	\cV \sw^\star = 0 \Longleftrightarrow w_1^\star=w_2^\star=\cdots=w_N^\star.
	}
	Next we check $w_k^\star = w^\star$.
	Since $\cP > 0$, condition \eqref{KKT-1-1} is equivalent to
	\eq{
	\cP \tcA\tran \cM \grad \cJ^o(\sw^\star) + \cV \sy^\star  = 0. \label{KKT-1-2}
	}
	Let $\cI=\mathds{1}_N \otimes I_M \in \RR^{MN\times M}$. Multiplying by $\cI\tran$ gives
	\eq{
	0=&\ \cI\tran \big(\cP \tcA\tran  \cM \grad \cJ^o(\sw^\star) + \cV \sy^\star\big) \overset{(a)}{=}\cI\tran \cP \tcA\tran \cM \grad \cJ^o(\sw^\star)\nnb
	=&\ {\sum_{k=1}^{N} p_k\mu_k \grad J_k(w_k^\star) \overset{\eqref{q-A}}{=} \frac{1}{\beta} \sum_{k=1}^{N} q_k \grad J_k(w^\star_k)},
	}
	where equality (a) holds because $\cV$ is symmetric and \eqref{xcn987-2}. Since $\beta \neq 0$, we conclude that $\sum_{k=1}^{N} q_k \grad J_k(w^\star_k)=0$,
	which shows that the entries $\{w_{k}^{\star}\}$, which are identical, must coincide with the minimizer $w^{\star}$ of \eqref{prob-dist}. 
\end{proof}
Observe that since ${\cJ}^{\star}(w)$ is assumed strongly-convex, then the solution to problem \eqref{prob-dist}, $w^\star$, is unique, and hence $\sw^\star$ is also unique. However, since $\cV$ is rank-deficient, there can be multiple solutions ${\sy}^{\star}$ satisfying \eqref{7280m}. Using an argument similar to \cite{shi2014linear,shi2015extra}, we can show that among all possible ${\sy}^{\star}$, there is a unique solution ${\sy}^{\star}_o$ lying in the column span of ${\cV}$.
\begin{lemma}[\sc Particular solution pair]\label{sy in V range sapce}
	When condition \eqref{q-A} holds and $\cJ^{o}(w)$ defined by \eqref{prob-consensus} is strongly-convex, there exists a unique pair of variables $(\sw^\star, \sy^\star_o)$, in which $\sy^\star_o$ lies in the range space of $\cV$, that satisfies conditions \eqref{KKT-1-1}-\eqref{KKT-2-2}.  
\end{lemma}
\begin{proof}
	{First we prove that there always exist some block vectors $(\sw^\star, \sy^\star)$ satisfying \eqref{KKT-1-1}--\eqref{KKT-2-2}. Indeed, when $\cJ^o(w)$ is strongly-convex, the solution to problem \eqref{prob-dist}, $w^\star$, exists and is unique. Let ${\sw}^{\star}=\mathds{1}_{N}\otimes w^{\star}$. We conclude from Lemma \ref{lm:null-V} that condition \eqref{KKT-2-2} holds. Next we check whether there exists some $\sy^\star$ such that 
		\eq{
		\cP^{-1}\cV \sy^\star = - \tcA\tran \cM \grad \cJ^o(\sw^\star),
		}
		or equivalently, 
		\eq{\label{k29}
			\cV \sy^\star & = - \cP \tcA\tran \cM \grad \cJ^o(\sw^\star) \nnb
			& = - \tcA \cP \cM \grad \cJ^o(\sw^\star) = -\frac{1}{\beta}\tcA \grad \cJ^\star(\sw^\star),
		}
		{where the last equality holds because
		\eq{
		\cP \cM \grad \cJ^o(\sw^\star)&=
		\ba{c}
		\mu_1 p_1 \grad J_1(w^\star)\\
		\vdots\\
		\mu_N p_N \grad J_N(w^\star)
		\ea \overset{\eqref{q-A}}{=}
		\ba{c}
		\frac{q_1}{\beta} \grad J_1(w^\star)\\
		\vdots\\
		\frac{q_N}{\beta}  \grad J_N(w^\star)
		\ea \nnb
		&\overset{\eqref{J-star}}{=} \frac{1}{\beta}\grad \cJ^\star(\sw^\star),
		}	
		}
		\hspace{-1.2mm}To prove the existence of $\sy^\star$, we need to show that $\tcA \grad \hspace{-1mm} \cJ^\star\hspace{-0.5mm}(\hspace{-0.5mm}\sw^\star\hspace{-0.5mm})$  lies in $\mathrm{range}(\cV)$. {Indeed, observe that 
		\eq{\label{23nsd}
		\hspace{-2mm}\cI\tran \tcA \grad \cJ^\star(\sw^\star) = \cI\tran \grad \cJ^\star(\sw^\star) \overset{(a)}{=} \sum_{k=1}^{N}q_k \grad J_k(w^\star) = 0
		}
		where the equality (a) holds because of equation \eqref{J-star}. Equality \eqref{23nsd} implies that $\tcA \grad \cJ^\star(\sw^\star)$ is orthogonal to $\mathrm{span}(\cI)$, i.e., $\mathrm{span}(\mathds{1}_N \otimes I_M)$. With \eqref{xcn987-2} we have \vspace{1mm}
		\eq{
		\tcA \grad \cJ^\star(\sw^\star) \perp \mathrm{null}(\cV) 
		\Leftrightarrow &\ \tcA \grad \cJ^\star(\sw^\star) \in \mathrm{range}(\cV\tran) \nnb
		\Leftrightarrow &\ \tcA \grad \cJ^\star(\sw^\star) \in \mathrm{range}(\cV),\vspace{1.5mm}
		}
		where the last ``$\Leftrightarrow$" holds because $\cV$ is symmetric.}

%
	}
	
	We now establish the existence of the unique pair $({ \sw}^{\star}, {\sy}_o^{\star})$. Thus, let $({\sw}^{\star},{\sy}^{\star})$ denote an arbitrary solution to \eqref{7280m}. Let further ${\sy}_o^{\star}$ denote the projection of ${\sy}^{\star}$ onto the column span of ${\cV}$. It follows that $\cV(\sy^\star - \sy^\star_o) = 0$ and, hence, $\cV\sy^\star = \cV \sy_o^\star$. Therefore, the pair $(\sw^\star, \sy_o^\star)$ also satisfies conditions \eqref{KKT-1-1}-\eqref{KKT-2-2}. 
	
	Next we verify the uniqueness of $\sy_o^\star$ by contradiction. Suppose there is a different $\sy_1^\star$ lying in $\cR(\cV)$ that also satisfies  condition \eqref{KKT-1-1}. We let $\sy^\star_o=\cV \sx^\star_o$ and $\sy^\star_1=\cV \sx^\star_1$. 
	Substituting $\sy^\star_o$ and $\sy^\star_1$ into condition \eqref{KKT-1-1}, we have
	\eq{
		\tcA\tran \cM \grad \cJ^o(\sw^\star) + \cP^{-1}\cV^2 \sx_o^\star & = 0, \label{ns-1}\\
		\tcA\tran \cM \grad \cJ^o(\sw^\star) + \cP^{-1}\cV^2 \sx_1^\star & = 0. \label{ns-2}
	}
	Subtracting \eqref{ns-2} from \eqref{ns-1} and recall $\cP>0$, we have $\cV^2(\sx_o^\star - \sx_1^\star) = 0$, which leads to $\cV(\sx_o^\star - \sx_1^\star) = 0 \Longleftrightarrow \sy_o^\star = \sy_1^\star$. This contradicts the assumption that $\sy_o^\star \neq \sy_1^\star$. 
\end{proof}
Using the above auxiliary results, we will show that $(\sw_i, \sy_i)$ generated through the exact diffusion \eqref{zn-1} will converge exponentially fast to $(\sw^\star, \sy^\star_o)$.

\subsection{Error Recursion}\label{subsec-error-recursion}
Let ${\sw}^{\star}=\mathds{1}_{N}\otimes w^{\star}$, which corresponds to a block vector with $w^{\star}$ repeated $N$ times. Introduce further the error vectors
\eq{\label{error}
\widetilde{\sw}_i={\sw}^{\star}-{\sw}_i,\;\;\;\;
\widetilde{\sy}_i=\sy_o^{\star}-{\sy}_i.
}
The first step in the convergence analysis is to examine the evolution of these error quantities. Multiplying the second recursion of \eqref{zn-1} by $\cV$ from the left gives:
%
%
\eq{
\cV  \sy_i = \cV  \sy_{i-1} + \frac{1}{2}(\cP-\cP \cA) \sw_i. \label{V-tran y}
}
Substituting \eqref{V-tran y} into the first recursion of \eqref{zn-1}, we have
\begin{equation}
	\left\{
	\begin{aligned}
	\tcA\tran \twd_i &\hspace{-0.5mm}=\hspace{-0.5mm} \tcA\tran \Big(\twd_{i-1}\hspace{-0.8mm}+\hspace{-0.8mm}\cM \grad \cJ^o(\sw_{i\hspace{-0.3mm}-\hspace{-0.3mm}1})\Big) \hspace{-0.8mm}+\hspace{-0.8mm} \cP^{-1} \cV  \sy_{i}, \label{ed-1-subtracrt} \\
	\tyd_i &= \tyd  _{i-1} - \cV \sw_i. 
	\end{aligned}
	\right.
\end{equation}
Subtracting optimality conditions \eqref{KKT-1-1}--\eqref{KKT-2-2} from \eqref{ed-1-subtracrt} leads to
\begin{equation}
	\hspace{-0.8mm}\left\{
	\begin{aligned}
	\hspace{-1.5mm}\tcA\tran \twd_i &\hspace{-0.5mm}=\hspace{-0.5mm} \tcA\tran \Big(\hspace{-0.8mm}\twd_{i-1}\hspace{-0.8mm}+\hspace{-0.8mm}\cM \big[\grad \cJ^o(\sw_{i\hspace{-0.3mm}-\hspace{-0.3mm}1})\hspace{-0.8mm}-\hspace{-0.8mm}\grad \cJ^o(\sw^\star)\big]\hspace{-0.8mm}\Big) \hspace{-0.8mm}-\hspace{-0.8mm} \cP^{-1} \cV  \tyd_{i}, \label{ed-1-subtracrt-1} \\
	\hspace{-1.5mm}\tyd_i &= \tyd_{i-1} +\cV \twd_i. 
	\end{aligned}
	\right.
\end{equation}
Next we examine the difference $\grad \cJ^o(\sw_{i\hspace{-0.3mm}-\hspace{-0.3mm}1})-\grad \cJ^o(\sw^\star)$. To begin with, we get from \eqref{grad-J^o} that 
\eq{\label{237sdb}
\hspace{-2mm}\grad \cJ^o(\sw_{i-1}) \hspace{-0.8mm}-\hspace{-0.8mm} \grad \cJ^o(\sw^\star) \hspace{-1mm}=\hspace{-1mm}
\ba{c}
\hspace{-2mm}\grad J_1(w_{1,i-1}) \hspace{-0.8mm}-\hspace{-0.8mm} \grad J_1(w^\star)\hspace{-2mm}\\
\hspace{-2mm}\vdots\hspace{-2mm}\\
\hspace{-2mm}\grad J_N(w_{N,i-1}) \hspace{-0.8mm}-\hspace{-0.8mm} \grad J_N(w^\star)\hspace{-2mm}
\ea
}
When $\grad J_k(w)$ is twice-differentiable (see Assumption \ref{ass-lip}), we can appeal to the mean-value theorem from Lemma D.1 in \cite{sayed2014adaptation}, which allows us to express each difference in \eqref{237sdb} in the following integral form in terms of Hessian matrices for any~ $k=1,2,\ldots,N$:
\eq{
 \grad J_k(w_{k,i-1}) \hspace{-1mm} - \hspace{-1mm}\grad J_k(w^\star)  \hspace{-0.8mm}=\hspace{-0.8mm}  -\Big(\hspace{-1mm} \int_0^1 \hspace{-2mm}\grad^2 \hspace{-1mm} J_k\hspace{-0.5mm} \big(\hspace{-0.5mm} w^\star \hspace{-1mm}-\hspace{-0.8mm} r \widetilde{w}_{k,i-1}\big)dr \hspace{-1mm}\Big)\widetilde{w}_{k,i-1}. \nn
}
If we let 
\eq{\label{H_k_i-1}
H_{k,i-1} \hspace{-1.5mm}\define \hspace{-1.5mm} \int_0^1 \grad^2 J_k\big(w^\star \hspace{-0.5mm}-\hspace{-0.5mm} r\widetilde{w}_{k,i-1}\big)dr \in \RR^{M\times M},
}
and introduce the block diagonal matrix:
\eq{\label{H_i-1}
\cH_{i-1} \hspace{-1.5mm}\define \hspace{-1.5mm} \mathrm{diag}\{H_{1,i-1},H_{2,i-1},\cdots,H_{N,i-1}\},
}
then we can rewrite \eqref{237sdb} in the form:
\eq{
\grad \cJ^o(\sw_{i-1}) - \grad \cJ^o(\sw^\star) = - \cH_{i-1} \twd_{i-1}.\label{xcnh}
}
Substituting into \eqref{ed-1-subtracrt-1} we get
\begin{equation}
	\left\{
	\begin{aligned}
	\tcA\tran \twd_i &\hspace{-0.5mm}=\hspace{-0.5mm} \tcA\tran (I_{MN}-\cM\cH_{i-1})\twd_{i-1} - \cP^{-1} \cV \tyd_{i}, \label{ed-1-subtracrt-1-1} \\
	\tyd_i &= \tyd_{i-1} +\cV \twd_i. 
	\end{aligned}
	\right.
\end{equation}
which is also equivalent to
\eq{\label{xcnh09}
&\ \ba{cc}
\tcA\tran & \cP^{-1}\cV \\
-\cV & I_{MN}
\ea
\ba{c}
\twd_i\\
\tyd_i
\ea \nnb
=&\ 
\ba{cc}
\tcA\tran (I_{MN}-\cM\cH_{i-1}) & 0\\
0 & I_{MN}
\ea
\ba{c}
\twd_{i-1}\\
\tyd_{i-1}
\ea.
}
Using the relations $\tcA\tran = \frac{I_{MN} + \cA\tran}{2}$ and $\cV^2=\frac{\cP - \cP\cA\tran}{2}$, it is easy to verify that
\eq{\label{nxcm987}
\ba{cc}
\tcA\tran & \cP^{-1}\cV \\
-\cV & I_{MN}
\ea^{-1} = 
\ba{cc}
I_{MN} &  -\cP^{-1}\cV \\
\cV & I_{MN}-\cV \cP^{-1} \cV
\ea.
}
Substituting into \eqref{nxcm987} gives
\eq{
\ba{c}
\hspace{-1mm}\twd_i\hspace{-1mm}\\
\hspace{-1mm}\tyd_i\hspace{-1mm}
\ea
&\hspace{-1mm}=\hspace{-1mm}
\ba{cc}
\hspace{-2mm}\tcA\tran (I_{MN}-\cM\cH_{i-1}) &  -\cP^{-1}\cV \hspace{-2mm}\\
\hspace{-2mm}\cV\tcA\tran (I_{MN}-\cM\cH_{i-1}) & I_{MN}-\cV \cP^{-1} \cV\hspace{-2mm}
\ea
\ba{c}
\hspace{-1mm}\twd_{i-1}\hspace{-1mm}\\
\hspace{-1mm}\tyd_{i-1}\hspace{-1mm}
\ea.
}
That is, the error vectors evolve according to:
\eq{
\boxed{	
\ba{c}
\hspace{-1mm}\twd_i\hspace{-1mm}\\
\hspace{-1mm}\tyd_i\hspace{-1mm}
\ea = (\cB - \cT_{i-1})\ba{c}
\hspace{-1mm}\twd_{i-1}\hspace{-1mm}\\
\hspace{-1mm}\tyd_{i-1}\hspace{-1mm}
\ea} \label{error-recursion}
}
where
\eq{
\cB&\define 
\ba{cc}
\tcA\tran &  -\cP^{-1}\cV \\
\cV\tcA\tran & I_{MN}-\cV \cP^{-1} \cV
\ea, \\
\cT_{i}&\define 
\ba{cc}
\tcA\tran\cM\cH_{i} &  0 \\
\cV\tcA\tran\cM\cH_{i} & 0
\ea.\label{T-defi}
}
Relation \eqref{error-recursion} is the error dynamics for the exact diffusion algorithm. We next examine its convergence properties.  

\vspace{-2mm}
\subsection{Proof of Convergence} \label{sec-convergence-proof}
We first introduce a common assumption.
\begin{assumption}[\sc Conditions on cost functions]
	\label{ass-lip} 
	Each $J_k(w)$ is twice differentiable, and its Hessian matrix satisfies 
	\eq{\label{xzh2300}
	\grad^2 J_k(w) \le \delta I_M. 
	}
	Moreover, there exists at least one agent $k_o$ such that $J_{k_o}(w)$ is $\nu$-strongly convex, i.e. 
	\eq{\label{xcn23987}
	\grad^2 J_{k_o}(w) > \nu I_M.
	}
%
	\rightline \qed
\end{assumption}
{\color{black}Note that when $J_k(w)$ is twice differentiable, condition \eqref{xzh2300} is equivalent to requiring each $\grad J_k(w)$ to be $\delta$-Lipschitz continuous \cite{sayed2014adaptation}. In addition,}  
condition \eqref{xcn23987} ensures the strong convexity of $\cJ^o(w)$ and ${\cJ}^{\star}(w)$, and the uniqueness of 
$w^o$ and $w^{\star}$. It follows from \eqref{xzh2300}--\eqref{xcn23987} and the definition \eqref{H_k_i-1} that 
%
\eq{\label{H-properties}
H_{k,i-1} \le \delta I_M,\ \forall k\quad \mbox{and}\quad
H_{k_o,i-1}  \ge \nu I_M.
}

The direct convergence analysis of recursion \eqref{error-recursion} is challenging. To facilitate the analysis, we identify a convenient change of basis and transform \eqref{error-recursion} into another equivalent form that is easier to handle. To do that, we first let
\eq{
B\define 
\ba{cc}
\tA\tran & - P^{-1}V \\
V\tA\tran & I_N - VP^{-1}V 
\ea\in \RR^{2N\times 2N}.
}
It holds that $\cB=B\otimes I_M$. In the following lemma we introduce a decomposition for matrix $B$ that will be fundamental to the subsequent analysis.
\begin{lemma}[\sc Fundamental Decomposition]\label{lm-B-decomposition}
	The matrix $B$ admits the following eigendecomposition
	\eq{
	B &= X D X^{-1}, \label{B-deco-0}
	}
	where 
	\eq{
	D=\ba{ccc}
	I_2 & \vline &0 \\
	\hline
	0 & \vline & D_1
	\ea，
	}
	and $D_1\in \RR^{(2N-2)\times (2N-2)}$ is a diagonal matrix with complex entries. The magnitudes of the diagonal entries satisfy
	\eq{\label{uwehn}
		&\hspace{-3mm} |D_1(2k\hspace{-0.8mm}-\hspace{-0.8mm}3,2k\hspace{-0.8mm}-\hspace{-0.8mm}3)|=|D_1(2k\hspace{-0.8mm}-\hspace{-0.8mm}2,2k\hspace{-0.8mm}-\hspace{-0.8mm}2)|=\sqrt{\lambda_{k}(\tA)}<1, \nnb
		& \hspace{4.7cm}\forall\ k=2,3,\cdots N.
	}
Moreover,
\eq{\label{X and X_inv}
X = 
\ba{ccc}
R & \vline &X_R
\ea, \quad
X^{-1}=
\ba{c}
L\\
\hline
X_L
\ea,
}
where $X_R\in \RR^{2N\times (2N-2)}$ and $X_L\in \RR^{(2N-2) \times 2N}$, and $R$ and $L$ are given by
\eq{\label{R and L}
R\hspace{-1mm}=\hspace{-1mm}\ba{cc}
\hspace{-2mm}\mathds{1}_N & 0\hspace{-2mm}\\
\hspace{-2mm}0 & \mathds{1}_N\hspace{-2mm}
\ea\hspace{-1mm}\in \hspace{-0.5mm}\RR^{2N\times 2},
L\hspace{-1mm}=\hspace{-1mm}\ba{cc}
\hspace{-1mm}p\tran & 0\hspace{-1mm} \\
\hspace{-1mm}0 & \frac{1}{N}\mathds{1}_N\tran
\ea\in \RR^{2\times 2N}.
}

\end{lemma}
\begin{proof}
	See Appendix \ref{appdx-Lemma-fd}.
\end{proof}
{\noindent{\bf Remark 1. (Other possible decompositions)}
The eigendecomposition \eqref{B-deco-0} for $B$ is not unique because we can always scale $X$ and $X^{-1}$ to achieve different decompositions. In this paper, we will study the following family of decompositions:
\eq{
B = X^\prime D (X^\prime)^{-1},
}
where 
\eq{\label{X and X_inv-prime}
	X^\prime = 
	\ba{ccc}
	R & \vline & \frac{1}{c}X_R
	\ea, \quad
	(X^\prime)^{-1}=
	\ba{c}
	L\\
	\hline
	c X_L
	\ea,
}
and $c$ can be set to any nonzero constant value. We will exploit later the choice of $c$ in identifying the stability range for exact diffusion.\hspace{7.15cm} \qed}

{
For convenience, we introduce the vectors:
\eq{\label{tsdhb}
&r_1=
\ba{c}
\hspace{-1.8mm}\mathds{1}_N \hspace{-1.8mm}\\
\hspace{-1.8mm}0\hspace{-1.8mm}
\ea,
r_2=
\ba{c}
\hspace{-1.8mm}0\hspace{-1.8mm}\\
\hspace{-1.8mm}\mathds{1}_N\hspace{-1.8mm}
\ea,
\ell_1=
\ba{c}
\hspace{-1.8mm}p\hspace{-1.8mm}\\
\hspace{-1.8mm}0\hspace{-1.8mm}
\ea,
\ell_2=
\ba{c}
\hspace{-1.8mm}0\hspace{-1.8mm}\\
\hspace{-1.8mm}\frac{1}{N}\mathds{1}_N\hspace{-1.8mm}
\ea,
}
so that
\eq{\label{R and L-2}
R = [r_1\ r_2], \quad L = \ba{c}\ell_1\tran\\ \ell_2\tran \ea. 
}
Using \eqref{B-deco-0}--\eqref{R and L-2}, we write
\eq{\label{cB-decompoision}
	\cB&=(X^\prime \otimes I_M) (D \otimes I_M) ( (X^\prime)^{-1} \otimes I_M)
	\define \cX' \cD (\cX')^{-1} \nnb
	&= 
	\ba{ccc}
	\hspace{-2mm}\cR_1 & \cR_2 & \frac{1}{c}\cX_{R}\hspace{-2mm}
	\ea
	\ba{ccc}
	I_{M} & 0 & 0\\
	0 & I_M & 0\\
	0 & 0 & \cD_1 
	\ea
	\ba{c}
	\cL_1\tran \\
	\cL_2\tran \\
	c \cX_L
	\ea,
}
where $\cD_1 = D_1\otimes I_M$, 
\eq{\label{R and L-kron}
	&\cR_1=\ba{c}\hspace{-1mm}\cI\hspace{-1mm}\\ \hspace{-1mm}0\hspace{-1mm}\ea \in \RR^{2NM\times M},\;\; \cR_2 = \ba{c}0 \\ \cI\ea\in \RR^{2NM\times M},}
\eq{
\cL_1 = \ba{c}\cp\\0 \ea\in \RR^{2NM\times M},\;\; \cL_2 = \ba{c}\hspace{-1.8mm}0\hspace{-1.8mm}\\\hspace{-1.8mm}\frac{1}{N}\cI \hspace{-1.8mm}\ea \in \RR^{2NM\times M},
}
while $\cX_R=X_R\otimes I_M\in \RR^{2NM\times 2(N-1)M}$ and $\cX_L=x_L\otimes I_M \in \RR^{2(N-1)M\times 2NM}$. Moreover, we are also introducing
\eq{
\cI\hspace{-1mm}=\hspace{-1mm}\mathds{1}_N \otimes I_M \in \RR^{NM \times M},\; \overline{\cP} \hspace{-1mm}=\hspace{-1mm} p\otimes I_M \in \RR^{NM \times M},\label{xcn287}
}
where the variable  $\overline{\cP}$ defined above is different from the earlier variable ${\cP}=P\otimes I_M\in\real^{NM\times NM}$.
Multiplying both sides of \eqref{error-recursion} by ${(\cX')}^{-1}$:
%
\eq{
\hspace{-1mm}(\cX')^{-1}
	\ba{c}
	\hspace{-1mm}\twd_i\hspace{-1mm}\\
	\hspace{-1mm}\tyd_i\hspace{-1mm}
	\ea
	\hspace{-1mm}=\hspace{-1mm}&\; 
	[(\cX')^{-1}(\cB-\cT_{i-1}) \cX'] (\cX')^{-1}
	\ba{c}
	\hspace{-1mm}\twd_{i-1}\hspace{-1mm}\\
	\hspace{-1mm}\tyd_{i-1}\hspace{-1mm}
	\ea
}
leads to
\eq{\label{recursion-transform-2}
	\ba{c}
	\hspace{-1mm}\bar{\sx}_i\hspace{-1mm}\\
	\hspace{-1mm}\widehat{\sx}_i\hspace{-1mm}\\
	\hspace{-1mm}\check{\sx}_i \hspace{-1mm}
	\ea
	\hspace{-1mm}=\hspace{-1mm}
	&\; 
	\left( 
	\ba{ccc}
	I_M & 0 & 0\\ 0 & I_M &0 \\0 & 0 & \cD_1
	\ea
	- \cS_{i-1}
	\right)
	\ba{c}
	\hspace{-1mm}\bar{\sx}_{i-1}\hspace{-1mm}\\
	\hspace{-1mm}\widehat{\sx}_{i-1}\hspace{-1mm}\\
	\hspace{-1mm}\check{\sx}_{i-1}\hspace{-1mm}
	\ea,	
}
where we defined
\eq{\label{x-bar and x-check}
	\ba{c}
	\hspace{-1mm}\bar{\sx}_i\hspace{-1mm}\\
	\hspace{-1mm}\widehat{\sx}_i\hspace{-1mm}\\
	\hspace{-1mm}\check{\sx}_i \hspace{-1mm}
	\ea \define&\; (\cX')^{-1}\ba{c}
	\hspace{-1mm}\twd_i\hspace{-1mm}\\
	\hspace{-1mm}\tyd_i\hspace{-1mm}
	\ea = 
	\ba{c}
	\cL_1\tran \\
	\cL_2\tran \\
	c \cX_L
	\ea
	\ba{c}
	\hspace{-1mm}\twd_i\hspace{-1mm}\\
	\hspace{-1mm}\tyd_i\hspace{-1mm}
	\ea,
}
and
\eq{\label{S}
\cS_{i-1}\define&(\cX')^{-1}\cT_{i-1}\cX'\nnb
\hspace{-2mm}=&\ba{ccc}
	\hspace{-2mm}\cL_1\tran \cT_{i-1}\cR_1 &  \cL_1\tran \cT_{i-1}\cR_2 &  \frac{1}{c}\cL_1\tran\cT_{i-1}\cX_R \hspace{-2mm}\\
	 \hspace{-2mm}\cL_2\tran \cT_{i-1}\cR_1&  \cL_2\tran \cT_{i-1}\cR_2 &  \frac{1}{c}\cL_2\tran\cT_{i-1}\cX_R \hspace{-2mm}\\
	\hspace{-2mm}c\cX_L\cT_{i-1}\cR_1 &  c\cX_L\cT_{i-1}\cR_2 & \cX_L\cT_{i-1}\cX_R \hspace{-2mm}
	\ea.
}
To evaluate the block entries of ${\cS}_{i-1}$, we partition 
\eq{\label{usd9}
	\cX_R = 
	\ba{cc}
	\cX_{R,u} \\
	\cX_{R,d}
	\ea,
}
where $\cX_{R,u}\in \RR^{NM \times 2(N-1)M}$ and $\cX_{R,d}\in \RR^{NM \times 2(N-1)M}$. Then, it can be verified that
\eq{\label{S-1st}
\cL_1\tran \cT_{i-1}\cR_1 &=  \overline{\cP}\tran\cM\cH_{i-1}\cI,\\
\cL_1\tran \cT_{i-1}\cR_2 &=  0,\\
\frac{1}{c}\cL_1\tran \cT_{i-1}\cX_R &= \frac{1}{c} \overline{\cP}\tran\cM\cH_{i\hspace{-0.3mm}-\hspace{-0.3mm}1}\cX_{R,u}.\label{cxbwm8}
}
While
\eq{
\cL_2\tran \cT_{i-1} = \ba{cc}
\hspace{-1.5mm}0 \hspace{-1mm}&\hspace{-1mm} \frac{1}{N}\cI\tran\hspace{-1.5mm}
\ea\hspace{-1.5mm}
\ba{cc}
\hspace{-1.5mm}\tcA\tran\cM\cH_{i-1} &  0\hspace{-1.5mm} \\
\hspace{-1.5mm}\cV\tcA\tran\cM\cH_{i-1} & 0\hspace{-1.5mm}
\ea \overset{\eqref{xcn987-2}}{=} \ba{cc}
\hspace{-1.5mm}0 \hspace{-1mm}&\hspace{-1mm} 0\hspace{-1.5mm}
\ea,
}
Therefore, it follows that
\eq{\label{S-2nd}
\cL_2\tran \cT_{i-1}\cR_1 =0,\;\; \cL_2\tran \cT_{i-1}\cR_2 =0,\;\; \frac{1}{c}\cL_2\tran\cT_{i-1}\cX_R =0.
}
Substituting \eqref{S}, \eqref{S-1st}--\eqref{cxbwm8} and \eqref{S-2nd} into \eqref{recursion-transform-2}, we have
\eq{\label{znhg}\footnotesize 
	\ba{c}
	\hspace{-2mm}\bar{\sx}_i\hspace{-2mm} \\
	\hspace{-2mm}\widehat{\sx}_i\hspace{-2mm} \\
	\hspace{-2mm}\check{\sx}_i\hspace{-2mm}
	\ea
	\hspace{-1.2mm}&=\hspace{-1.2mm}
	\footnotesize 
	\ba{ccc}
	\hspace{-3mm}I_{\hspace{-0.3mm} M} \hspace{-1.2mm}-\hspace{-1.2mm} \overline{\cP}\tran\hspace{-1.2mm}\cM\cH_{i\hspace{-0.3mm}-\hspace{-0.3mm}1}\cI &  \hspace{-0.8mm}0 &\hspace{-1.3mm} -\frac{1}{c}\overline{\cP}\tran\hspace{-1.2mm}\cM\cH_{i\hspace{-0.3mm}-\hspace{-0.3mm}1}\cX_{R,u} \hspace{-2mm}\\
	\hspace{-3mm}0 &  \hspace{-1.3mm}I_{\hspace{-0.3mm} M} &\hspace{-3.3mm} 0 \hspace{-2mm} \\
	\hspace{2mm}-c\cX_L\cT_{i-1}\cR_1 &  \hspace{-1.3mm}-c\cX_L\cT_{i-1}\cR_2 \hspace{-1.3mm} &  \cD_1 \hspace{-1mm}-\hspace{-1mm} \cX_L\cT_{i-1}\cX_R \hspace{-2mm}
	\ea 
	\hspace{-2mm}\ba{c}
	\hspace{-2.5mm}\bar{\sx}_{i\hspace{-0.5mm}-\hspace{-0.5mm}1}\hspace{-2.5mm} \\
	\hspace{-2.5mm}\widehat{\sx}_{i\hspace{-0.5mm}-\hspace{-0.5mm}1}\hspace{-2.5mm} \\
	\hspace{-2.5mm}\check{\sx}_{i\hspace{-0.5mm}-\hspace{-0.5mm}1}\hspace{-2.5mm}
	\ea
}
From the second line of \eqref{znhg}, we get
\eq{\label{28cn00}
	\widehat{\sx}_i = \widehat{\sx}_{i-1}.
}
As a result, $\widehat{\sx}_i$ will stay at $0$ only if the initial value $\widehat{\sx}_{0} = 0$. From the definition of $\cL_2$ in \eqref{tsdhb} and \eqref{x-bar and x-check} we have
\eq{\label{hx_0=0}
	\widehat{\sx}_0 &= \cL_2\tran \ba{c}
	\hspace{-1mm}\twd_0\hspace{-1mm}\\
	\hspace{-1mm}\tyd_0\hspace{-1mm}
	\ea = \frac{1}{N}\cI\tran \tyd_0 \nnb
	&\overset{\eqref{error}}{=}\frac{1}{N}\cI\tran (\sy^\star_o - \sy_0) \overset{\eqref{zn-0}}{=} \frac{1}{N}\cI\tran (\sy^\star_o - \cV \sw_0).
}
Recall from Lemma \ref{sy in V range sapce} that $\sy_o^\star$ lies in the $\mathrm{range}(\cV)$, so that $\sy^\star_o - \cV \sw_0$ also lies in $\mathrm{range}(\cV)$. From Lemma \ref{lm:null-V} we conclude that $\widehat{\sx}_0=0$. Therefore, from \eqref{28cn00} we have 
\eq{\label{xzcn}
	\widehat{\sx}_i=0, \quad \forall i\ge 0
}
With \eqref{xzcn}, recursion \eqref{znhg} is equivalent to
\eq{\label{final-recursion}
	\hspace{-3mm}
	\boxed{
	\ba{c}
	\hspace{-2mm}\bar{\sx}_i\hspace{-2mm}\\
	\hspace{-2mm}\check{\sx}_i\hspace{-2mm}
	\ea \hspace{-1.5mm}=\hspace{-1.5mm}
	\ba{cc}
	\hspace{-2mm}I_M \hspace{-1mm}-\hspace{-1mm}{\overline{\cP}\tran\cM\cH_{i\hspace{-0.4mm}-\hspace{-0.4mm}1}\cI} & -\frac{1}{c}\overline{\cP}\tran\cM\cH_{i\hspace{-0.4mm}-\hspace{-0.4mm}1}\cX_{R,u}\hspace{-2mm}\\
	\hspace{-2mm}- c\cX_L \cT_{i-1} \cR_1 & \cD_1 - \cX_L \cT_{i-1} \cX_R \hspace{-2mm}
	\ea \hspace{-1.5mm}
	\ba{c}
	\hspace{-2mm}\bar{\sx}_{i\hspace{-0.4mm}-\hspace{-0.4mm}1}\hspace{-2mm}\\
	\hspace{-2mm}\check{\sx}_{i\hspace{-0.4mm}-\hspace{-0.4mm}1}\hspace{-2mm}
	\ea}
}
The convergence of the above recursion is stated as follows.
}

\begin{theorem}[\sc Linear Convergence]\label{theom-convergence}
	Suppose each cost function $J_k(w)$ satisfies Assumption \ref{ass-lip}, the left-stochastic matrix $A$ satisfies the local balance condition \eqref{local-balance}, and also condition \eqref{q-A} holds. The exact diffusion recursion \eqref{zn-1} converges exponentially fast to $(\sw^\star, \sy^\star_o)$ for step-sizes satisfying 
	\eq{\label{diffusion-range-step-size}
		\mu_{\max}\le \frac{p_{k_o}\tau_{k_o}\nu (1-\lambda)}{2\sqrt{p_{\max}}\alpha_d \delta^2},
	}
	where $\lambda \hspace{-1mm}=\hspace{-1mm} \sqrt{\lambda_2(\tA)}\hspace{-1mm}<\hspace{-1mm}1$,  $\tau_{k_o}\hspace{-1mm}=\hspace{-1mm}\mu_{k_o}/\mu_{\max}$, $p_{\max}\hspace{-1mm}=\hspace{-1mm}\max_k\{p_k\}$ and
	\eq{\label{T-d}
	\alpha_d \hspace{-0.5mm} \define \hspace{-0.5mm} \|\cX_L\| \|\cT_d\| \|\cX_R\|, \mbox{ where } \cT_d \define 
	\ba{cc}
	\tcA\tran & 0\\
	\cV\tcA\tran & 0
	\ea.
	} 
%
	The convergence rate for the error variables is given by
	\eq{
	\left\|
	\ba{cc}
	\twd_i\\
	\tyd_i
	\ea
	\right\|^2 \le C \rho^i,
	}
	where $C$ is some constant and $\rho=1-O(\mu_{\max})$, namely,
	\eq{
		\hspace{-1mm}\rho =& \max\Big\{ 1-p_{k_o}\tau_{k_o}\nu \mu_{\max} + \frac{2\sqrt{p_{\max}}\alpha_d \delta^2 \mu^2_{\max}}{1-\lambda},\nnb
		&\hspace{1cm} \lambda\hspace{-0.5mm}+\hspace{-0.5mm}\frac{\sqrt{p_{\max}}\alpha_d \delta^2\mu_{\max}}{p_{k_o}\tau_{k_o}\nu} \hspace{-0.5mm}+\hspace{-0.5mm} \frac{2\alpha_d^2 \delta^2\mu_{\max}^2}{1-\lambda} \Big\} < 1.
	}
\end{theorem}

\begin{proof} 
	See Appendix \ref{app-theom-conv}.
\end{proof} 

{\color{black}
With similar arguments shown above, we can also establish the convergence property of 
the exact diffusion algorithm 1' from Part I \cite{yuan2017exact1}. Compared to the above convergence analysis, the error dynamics for algorithm 1' will now be perturbed by a mismatch term caused by the power iteration. Nevertheless, once the analysis is carried out we arrive at a similar conclusion.  


	\begin{theorem}[\sc Linear convergence of Algorithm $1^\prime$] \label{them-algorithm-prime}Under the conditions of Theorem \ref{theom-convergence}, there exists a positive constant $\bar{\mu} > 0$ such that for step-sizes satisfying $\mu < \bar{\mu}$, the exact diffusion Algorithm 1' will converge exponentially fast to $(\sw^\star, \sy_o^\star)$.
	\end{theorem}
	\begin{proof}
		See Appendix \ref{app-algorithm-prime}.
	\end{proof}
	
}

%

%

\section{Stability Comparison with EXTRA} \label{sec-comparision}
\subsection{Stability Range of EXTRA}
In the case where the combination matrix $A$ is symmetric and {\em doubly-stochastic}, and all agents choose the {\em same} step-size $\mu$, the exact diffusion recursion \eqref{zn-1} reduces to 
\begin{equation}
\left\{
\begin{aligned}
\sw_i &= \tcA \Big(\sw_{i\hspace{-0.3mm}-\hspace{-0.3mm}1}\hspace{-0.8mm}-\hspace{-0.8mm}\mu \grad \cJ^o(\sw_{i\hspace{-0.3mm}-\hspace{-0.3mm}1})\Big)\hspace{-0.8mm}-\hspace{-0.8mm}\cP^{-1}\cV \sy_{i-1}, \label{zn-2}\\
\sy_i &= \sy_{i-1} + \cV \sw_i. 
\end{aligned}
\right.
\end{equation}
where $\cP=I_{MN}/N$.
In comparison, the EXTRA consensus algorithm \cite{shi2015extra} has the following form for the same ${\cP}$ (recall though that exact diffusion \eqref{zn-1} was derived and is applicable to a larger class of balanced left-stochastic matrices and is not limited to symmetric doubly stochastic matrices; it also allows for heterogeneous step-sizes):
	\begin{equation}
	\left\{
	\begin{aligned}
	\sw_i^e &= \tcA \sw_{i\hspace{-0.3mm}-\hspace{-0.3mm}1}^e\hspace{-0.8mm}-\hspace{-0.8mm}\mu \grad \cJ^o(\sw_{i\hspace{-0.3mm}-\hspace{-0.3mm}1}^e)\hspace{-0.8mm}-\hspace{-0.8mm}\cP^{-1}\cV \sy^e_{i-1}, \label{extra-pd-2}\\
	\sy^e_i &= \sy^e_{i-1} + \cV \sw^e_i, 
	\end{aligned}
	\right.
	\end{equation}
where we are using the notation $\sw_i^e$ and $\sy_i^e$ to refer to the primal and dual iterates in the EXTRA implementation. Similar to \eqref{zn-0}, the initial condition for \eqref{extra-pd-2} is 
\begin{equation}
\left\{
\begin{aligned}
\sw^e_0 &= \tcA \sw^e_{\hspace{-0.3mm}-\hspace{-0.3mm}1}\hspace{-0.8mm}-\hspace{-0.8mm}\mu \grad \cJ^o(\sw^e_{\hspace{-0.3mm}-\hspace{-0.3mm}1}), \label{zn-0-extra}\\
\sy^e_0 &= \cV \sw^e_0. 
\end{aligned}
\right.
\end{equation}
Comparing \eqref{zn-2} and \eqref{extra-pd-2} we observe one key difference; the diffusion update in \eqref{zn-2} involves a traditional gradient descent step in the form of ${\sw}_{i-1}-\mu\nabla {\cal J}^{o} ({\sw}_{i-1})$. This step starts from ${\sw}_{i-1}$ and evaluates the graduate vector at the same location. The result is then multiplied by the combination policy $\widetilde{\cal A}$. The same is {\em not} true for exact consensus in \eqref{extra-pd-2}; we observe an asymmetry in its update: the gradient vector is evaluated at ${\sw}_{i-1}^{e}$ while the starting point is at a different location given by $\widetilde{\cal A}{\sw}_{i-1}^e$. This type of asymmetry was shown in \cite{sayed2014adaptive,sayed2014adaptation} to result in  instabilities for the traditional consensus implementation in comparison to the traditional diffusion implementation. It turns out that a similar  problem continues to exist for the EXTRA consensus solution \eqref{extra-pd-2}. In particular, we will show that its stability range is smaller than exact diffusion (i.e., the latter is stable for a larger range of step-sizes, which in turn helps attain faster convergence rates). We will illustrate this behavior in the simulations in some detail. Here, though, we establish these observations analytically. The arguments used to examine the stability range of EXTRA consensus are similar to what we did in Section  \ref{sec-convergence} for exact diffusion; we shall therefore be brief and highlight only the differences.  

As already noted in \cite{shi2015extra}, the optimality conditions for the EXTRA consensus algorithm require the existence of block vectors $({\sw}^{\star},{\sy}^{\star})$ such that 
\eq{
	\mu \grad \cJ^o(\sw^\star) + \cP^{-1}\cV \sy^\star & = 0, \label{extra-KKT-1-1} \\
	\cV \sw^\star &= 0. \label{extra-KKT-2-2}
}
Moreover, as argued in Lemma \ref{sy in V range sapce}, there also exists a unique pair of variables $(\sw^\star, \sy^\star_o)$, in which $\sy^\star_o$ lies in the range space of $\cV$, that satisfies \eqref{extra-KKT-1-1}--\eqref{extra-KKT-2-2}. Now we introduce the block error vectors:
\eq{\label{extra-error}
	\widetilde{\sw}^e_i={\sw}^{\star}-{\sw}_i^e,\;\;\;\;
	\widetilde{\sy}^e_i=\sy_o^{\star}-{\sy}_i^e,
}
and examine the evolution of these error quantities. Using similar arguments in Section \ref{subsec-error-recursion}, and recalling the facts that $\tcA$ is symmetric doubly-stochastic, and $\cM=\mu I_{MN}$, we arrive at the error recursion for EXTRA consensus (see Appendix \ref{app-sta-EXTRA-error-recursion}): 
\eq{
	\ba{c}
	\hspace{-1mm}\twd_i^e\hspace{-1mm}\\
	\hspace{-1mm}\tyd_i^e\hspace{-1mm}
	\ea
	&\hspace{-1mm}=\hspace{-1mm}
	\ba{cc}
	\hspace{-2mm}\tcA - \mu \cH_{i-1} &  -\cP^{-1}\cV \hspace{-2mm}\\
	\hspace{-2mm}\cV (\tcA -\mu \cH_{i-1}) & I_{MN}-\cV \cP^{-1} \cV\hspace{-2mm}
	\ea
	\ba{c}
	\hspace{-1mm}\twd^e_{i-1}\hspace{-1mm}\\
	\hspace{-1mm}\tyd^e_{i-1}\hspace{-1mm}
	\ea \nnb
	&\hspace{-2mm}\define (\cB^e - \cT^e_{i-1})\ba{c}
	\hspace{-1mm}\twd^e_{i-1}\hspace{-1mm}\\
	\hspace{-1mm}\tyd^e_{i-1}\hspace{-1mm}
	\ea, \label{extra-error-recursion-2}
}
where
\eq{
	\cB^e\define 
	\ba{cc}
	\hspace{-1.5mm}\tcA &  -\cP^{-1}\cV \hspace{-1.5mm}\\
	\hspace{-1.5mm}\cV\tcA & I_{MN}\hspace{-1mm}-\hspace{-1mm}\cV \cP^{-1} \cV\hspace{-1.5mm}
	\ea, \cT^e_{i}\define 
	\ba{cc}
	\hspace{-1.5mm}\mu\cH_{i} &  0 \hspace{-1.5mm}\\
	\hspace{-1.5mm}\mu\cV\cH_{i} & 0 \hspace{-1.5mm}
	\ea.\label{extra-T-defi}
}
It is instructive to compare \eqref{extra-error-recursion-2}--\eqref{extra-T-defi} with \eqref{error-recursion}--\eqref{T-defi}. These recursions capture the error dynamics for the exact consensus and diffusion strategies. Observe that $\cB^e = \cB$ when $\tcA$ is symmetric and $\cM=\mu I_{MN}$. Therefore, $\cB^e$ has the same eigenvalue decomposition as in \eqref{cB-decompoision}--\eqref{xcn287}. 
With similar arguments to \eqref{B-deco-0}--\eqref{final-recursion}, we conclude that the reduced error recur-sion for EXTRA consensus takes the form (see Appendix \ref{app-sta-EXTRA-reduced}):
\eq{\label{final-recursion-extra}
	\ba{c}
	\hspace{-2mm}\bar{\sx}^e_i\hspace{-2mm}\\
	\hspace{-2mm}\check{\sx}^e_i\hspace{-2mm}
	\ea \hspace{-1.5mm}=\hspace{-1.5mm}
	\ba{cc}
	\hspace{-2mm}I_M \hspace{-1mm}-\hspace{-1mm}{\mu \overline{\cP}\tran\cH_{i\hspace{-0.4mm}-\hspace{-0.4mm}1}\cI} \hspace{-1mm}&\hspace{-1mm} -\frac{\mu}{c}\overline{\cP}\tran\cH_{i\hspace{-0.4mm}-\hspace{-0.4mm}1}\cX_{R,u}\hspace{-2mm}\\
	\hspace{-2mm}- c\cX_L \cT^e_{i-1} \cR_1 \hspace{-1mm}&\hspace{-1mm} \cD_1 - \cX_L \cT^e_{i-1} \cX_R \hspace{-2mm}
	\ea \hspace{-1.5mm}
	\ba{c}
	\hspace{-2mm}\bar{\sx}^e_{i\hspace{-0.4mm}-\hspace{-0.4mm}1}\hspace{-2mm}\\
	\hspace{-2mm}\check{\sx}^e_{i\hspace{-0.4mm}-\hspace{-0.4mm}1}\hspace{-2mm}
	\ea.
}

\noindent Following the same proof technique as for Theorem \ref{theom-convergence}, we can now establish the following result concerning stability conditions and convergence rate for EXTRA consensus. 
\begin{theorem}[\sc Linear Convergence of EXTRA]\label{lm-convergence-extra}
	Suppose each cost function $J_k(w)$ satisfies Assumption \ref{ass-lip}, and the combination matrix $A$ is primitive, symmetric and doubly-stochastic. The EXTRA recursion \eqref{extra-error-recursion-2} converges exponentially fast to $(\sw^\star, \sy^\star_o)$ for step-sizes $\mu$ satisfying 
	\eq{\label{extra-range-step-size}
		\mu\le \frac{\nu (1-\lambda)}{2\sqrt{N}\alpha_e\delta^2},
	}
	where $\lambda = \sqrt{\lambda_2(\tA)}<1$ and
	\eq{\label{T-e}
	\alpha_e = \|\cX_L\| \|\cT_e\| \|\cX_R\|,\mbox{ where }
	\cT_e = 
	\ba{cc}
	I_{MN} & 0\\
	\cV & 0
	\ea.
	}
	The convergence rate for the error variables is given by
	\eq{
		\left\|
		\ba{cc}
		\twd^e_i\\
		\tyd^e_i
		\ea
		\right\|^2 \le C \rho^i,
	}
	where $C$ is some constant and $\rho=1-O(\mu_{\max})$, namely,
	\eq{\label{rho-e}
		\rho_e =& \max\Big\{ 1-\frac{\nu}{N} \mu_{\max} + \frac{2\alpha_e \delta^2 \mu^2_{\max}}{\sqrt{N}(1-\lambda)},\nnb
		&\hspace{1cm} \lambda+\frac{\sqrt{N}\alpha_e \delta^2\mu_{\max}}{\nu} + \frac{2\alpha_e^2 \delta^2\mu_{\max}^2}{1-\lambda} \Big\} < 1.
	}
\end{theorem}
\begin{proof}
	See Appendix \ref{app-conv-extra}.
\end{proof}

\subsection{Comparison of Stability Ranges }
When $\tcA$ is symmetric and $\cM=\mu I_{MN}$, from Theorem \ref{theom-convergence} we get the stability range of exact diffusion:
\eq{\label{diffusion-range-step-size-compare}
	\mu\le \frac{\nu (1-\lambda)}{2\sqrt{N}\|\cX_L\| \|\cT_d\| \|\cX_R\|\delta^2},
}
where
\eq{\label{T-d-compare}
	\cT_d = 
	\ba{cc}
	\tcA & 0\\
	\cV\tcA & 0
	\ea.
}
Comparing \eqref{diffusion-range-step-size-compare} with \eqref{extra-range-step-size}, we observe that the expressions differ by the terms $\|\cT_e\|$ and $\|\cT_d\|$. We therefore need to compare these two norms. 

Notice that
\eq{
	\|\cT_e\|^2 &= \lambda_{\max}(\cT_e\tran \cT_e)=\lambda_{\max}(I_{MN} + \cV^2),\\
	\|\cT_d\|^2 &= \lambda_{\max}(\cT_d\tran \cT_d)= \lambda_{\max}\big(\tcA(I_{MN} + \cV^2)\tcA \big). 
}
It is easy to recognize that $\lambda_{\max}(I_{MN} + \cV^2) = \lambda_{\max}(I_{N} + V^2)$. Now, since $A$ is assumed symmetric doubly-stochastic and $P={1\over N} I_N$,  we have
\eq{
	\hspace{-3mm}	I_{N}+V^2 &= I_{N}+ \frac{P - P A}{2} \nnb
	&= I_{N}+\frac{I_{N} - A}{2N} = \frac{(2N+1)I_{N} - A}{2N}, \label{xcngyhu}
}
Moreover, since $A$ is primitive, symmetric and doubly stochastic, we can decompose it as 
\eq{\label{A-deco}
	A = U \Lambda\, U\tran,
} 
where $U$ is orthogonal,  $\Lambda=\diag\{\lambda_1\hspace{-0.5mm}(\hspace{-0.5mm}A\hspace{-0.3mm}),\cdots\hspace{-0.5mm}, \lambda_N(A)\}$ and 
\eq{\label{23n8}
	1 = \lambda_1(A) > \lambda_2(A) \ge \cdots \ge \lambda_N(A) >-1.
}
With this decomposition, expression \eqref{xcngyhu} can be rewritten as
\eq{
	I_{N}+V^2 = U \frac{(2N+1)I_{N} - \Lambda}{2N} U\tran. \label{xcnwedhg8}
} 
from which we conclude that
\eq{\label{zxc6hs90}
	\boxed{
	\lambda_{\max}(I_{N} + V^2) = \frac{(2N+1) - \lambda_N(A)}{2N}
}
}

Similarly, $\lambda_{\max}(\tcA (I_{MN}+\cV^2) \tcA)=\lambda_{\max}(\tA (I_{N} + V^2) \tA)$. Using $\tA = \frac{I_N + A}{2}$, and equations \eqref{A-deco} and \eqref{xcnwedhg8}, we have
\eq{
	&\hspace{-5mm} \tA (I_{N}+ V^2) \tA \nnb
	=&\ \left(\frac{I_{N} + A}{2}\right) \left( \frac{(2N+1)I_{N} - A}{2N} \right)\left(\frac{I_{N} + A}{2}\right) }
\eq{= U \left(\hspace{-1mm}\frac{I_{N} + \Lambda}{2}\hspace{-1mm}\right) \left(\hspace{-1mm}\frac{(2N+1)I_{N} - \Lambda}{2N} \hspace{-1mm}\right) \left(\frac{I_{N} + \Lambda}{2}\right) U\tran.
}
Therefore, we have
\eq{
	&\ \lambda_{\max}\left(\tA (I_{N}+ V^2) \tA\right) \nnb
	=&\ \max_k\left\{\left(\frac{\lambda_k(A) + 1}{2}\right)^2 \left(\frac{2N+1 - \lambda_k(A)}{2N} \right) \right\} \label{xn3wh8-0} \nnb
	\overset{(a)}{\le}&\ \max_k\left\{\left(\frac{\lambda_k(A) + 1}{2}\right)^2\right\} \max_k\left\{ \frac{2N+1 - \lambda_k(A)}{2N} \right\}  \nnb
	\overset{\eqref{23n8}}{=}&\ \frac{2N+1 - \lambda_N(A)}{2N}. 
}
It is worth noting that the ``$=$" sign cannot hold in (a) because 
\eq{
\argmax_k\left\{\left(\frac{\lambda_k(A) + 1}{2}\right)^2\right\} &= 1,\\
\argmax_k\left\{ \frac{2N+1 - \lambda_k(A)}{2N} \right\} &= N.
}
In other words, $\left(\frac{\lambda_k(A) + 1}{2}\right)^2$ and $\frac{2N+1 - \lambda_k(A)}{2N}$ cannot reach their maximum values at the same $k$. As a result,
\eq{\label{compare-Td-Te}
	\|\cT_d\|^2 < \|\cT_e\|^2 \Longrightarrow \alpha_d < \alpha_e.
}
This means that the upper bound on $\mu$ in \eqref{extra-range-step-size} is smaller than the upper bound on $\mu$ in \eqref{diffusion-range-step-size-compare}.

We can also compare the convergence rates of EXTRA consensus and exact diffusion when both algorithms converge. When $\tcA$ is symmetric and $\cM=\mu I_{MN}$, from Theorem \ref{theom-convergence} we get the convergence rate of exact diffusion:
\eq{\label{rho-d}
\rho_d =& \max\Big\{ 1-\frac{\nu}{N} \mu_{\max} + \frac{2\alpha_d \delta^2 \mu^2_{\max}}{\sqrt{N}(1-\lambda)},\nnb
&\hspace{1cm} \lambda+\frac{\sqrt{N}\alpha_d \delta^2\mu_{\max}}{\nu} + \frac{2\alpha_d^2 \delta^2\mu_{\max}^2}{1-\lambda} \Big\}.
}
It is clear from \eqref{rho-d} and \eqref{rho-e} that EXTRA consensus and exact diffusion have the same convergence
rate to first-order in $\mu_{\max}$, namely,
\eq{
\widehat{\rho}_d = 1 - \frac{\nu}{N}\mu_{\max} = \widehat{\rho}_e
} 
More generally, when higher-order terms in $\mu_{\max}$ cannot be ignored, it holds that $\rho_d<\rho_e$ because $\alpha_d<\alpha_e$ (see  \eqref{compare-Td-Te}). In this situation, exact diffusion converges faster than EXTRA.


{\vspace{-1.8mm}
\subsection{An Analytical Example}
In this subsection we illustrate the stability of exact diffusion by considering the example of mean-square-error (MSE) networks \cite{sayed2014adaptation}. Suppose $N$ agents are observing streaming data $\{\d_k(i), \u_{k,i}\}$ that satisfy the regression model
\eq{\label{regresson-model}
\d_k(i)=\u_{k,i}\tran w^o +\v_k(i),
}
where $w^o$ is unknown and $\v_k(i)$ is the noise process that is independent of the regression data $\u_{k,j}$ for any $k,j$. Furthermore, we assume $\u_{k,i}$ is zero-mean with covariance matrix $R_{u,k}=\bE\u_{k,i}\u_{k,i}\tran > 0$, and $\v_k(i)$ is also zero-mean with power $\sigma_{v,k}^2=\bE \v_k^2(i)$. We denote the cross covariance vector between $\d_k(i)$
and $\u_{k,i}$ by $r_{du,k} = \bE\d_k(i)\u_{k,i}$.
To discover the unknown $w^o$, the agents cooperate to solve the following mean-square-error problem:
\eq{\label{MSE-network}
\min_{w\in \RR^M}\ \textstyle{\frac{1}{2}\sum_{k=1}^{N}} \bE \big(\d_k(i)-\u_{k,i}\tran w\big)^2.
}
It was shown in Example 6.1 of \cite{sayed2014adaptation} that the global minimizer of problem \eqref{MSE-network} coincides with the unknown $w^o$ in \eqref{regresson-model}.

When $R_{u,k}$ and $r_{du,k}$ are unknown and only realizations of $\u_{k,i}$ and $\d_k(i)$ are observed by agent $k$, one can employ the diffusion algorithm with stochastic gradient descent to solve \eqref{MSE-network}. However, when $R_{u,k}$ and $r_{du,k}$ are known in advance, problem \eqref{MSE-network} reduces to deterministic optimization problem:
\eq{\label{MSE-network-determ}
	\min_{w\in \RR^M}\ \frac{1}{2}\sum_{k=1}^{N} \big( w\tran R_{u,k}\hspace{0.3mm} w - 2 r_{du,k}\tran w \big).
}
We can then employ the exact diffusion or the EXTRA consensus algorithm to solve \eqref{MSE-network-determ}.

To illustrate the stability issue, it is sufficient to consider a network with $2$ agents (see Fig. \ref{fig:2-agent-MSE}) and with diagonal Hessian matrices, i.e., 
%
\eq{\label{xcn99}
R_{u,1}=R_{u,2}=\sigma^2 I_M.
}
We assume the agents use the combination weights {$\{a, 1-a\}$} with $a\in (0,1)$, so that
\eq{\label{2-MSE-A}
{A=
\ba{cc}
a & 1-a\\
1-a & a
\ea \in \RR^{2\times 2},
}}
which is symmetric and doubly stochastic. The two agents employ the same step-size $\mu$ (or $\mu^e$ in the EXTRA recursion). It is worth noting that the following analysis can be extended to $N$ agents with some more algebra.

Under \eqref{xcn99}, we have $H_1=H_2=\sigma^2 I_M$ and $\cH=\diag\{H_1,H_2\}=\sigma^2 I_{2M}$. For the matrix $A$ in \eqref{2-MSE-A}, we have
\eq{\label{zn288}
\lambda_1(A)=1,\quad \lambda_2(A)=2a-1\in (-1,1),
}
and $p=[0.5;0.5]$, $P=0.5 I_2$. 

\begin{figure}
	\centering
	\includegraphics[scale=0.35]{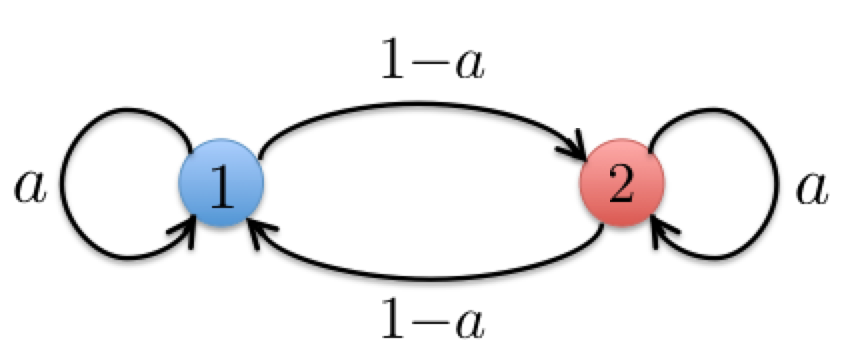}
	\caption{A two-agent network using combination weights $\{a, 1-a\}$}
	\label{fig:2-agent-MSE}
	\vspace{-5mm}
\end{figure} 

Let $\tzd_i = [\twd_i;\tyd_i]\in \RR^{2M}$, and $\tzd_i^e = [\twd_i^e;\tyd_i^e]\in \RR^{2M}$. The exact diffusion error recursion \eqref{error-recursion} and the EXTRA error recursion \eqref{extra-error-recursion-2} reduce to 
\eq{
\tzd_i = \cQ_d \tzd_{i-1},\label{special-ed}
}
\eq{
\tzd_i^e = \cQ_e \tzd_{i-1}^e,\label{special-extra}
}
where
\eq{
\cQ_d \hspace{-1mm}&=\hspace{-1mm}
\underbrace{\ba{cc}
	\hspace{-2mm}(1-\mu \sigma^2 )\tA &  -2 V \hspace{-2mm}\\
	\hspace{-2mm} (1-\mu \sigma^2 )V \tA & \tA \hspace{-2mm}
	\ea}_{Q_d} \otimes I_M, \\
\cQ_e \hspace{-1mm}&=\hspace{-1mm}
\underbrace{\ba{cc}
	\hspace{-2mm}\tA-\mu^e \sigma^2 I_{2} &  -2 V \hspace{-2mm}\\
	\hspace{-2mm} V(\tA-\mu^e \sigma^2 I_{2}) & \tA \hspace{-2mm}
	\ea}_{Q_e} \otimes I_M.
}
To guarantee the convergence of $\tzd_i$ and $\tzd_i^e$,  we need to examine the eigenstructure of the $4\times 4$ matrices $Q_d$ and $Q_e$. {The proof of the next lemma is quite similar to Lemma \ref{lm-B-decomposition}; if desired, see Appendix F of the arXiv version\cite{yuan2017exact2}}.

\begin{lemma}[\sc Eigenstructure of $Q_d$]\label{lm-Q_d-decom} The matrix $Q_d$ admits the following eigendecomposition
	\eq{\label{xn88}
		Q_d = X \overline{Q}_d X^{-1},
	}
	where 
	\eq{
	\overline{Q}_d = 
	\ba{cc}
	1 & 0 \\
	0 & E_d
	\ea
	}
	and
	\eq{\label{cn9999}
		E_d = 
		\ba{ccc}
		\hspace{-2mm}1-\mu\sigma^2 & 0 & 0\hspace{-2mm}\\
		\hspace{-2mm}0 & (1-\mu\sigma^2)a & -\sqrt{2-2a}\hspace{-2mm}\\
		\hspace{-2mm}0 & (1-\mu\sigma^2)a\sqrt{\frac{1-a}{2}} & a\hspace{-2mm}
		\ea.
	}
	Moreover, the matrices $X$ and $X^{-1}$ are given by
	\eq{\label{m87dh}
	X=\ba{cc}
	r& X_R
	\ea,\quad
	X^{-1}=
	\ba{c}
	\ell\tran\\
	X_L
	\ea,
	}
	where $X_R\in \RR^{4\times 3}$, $X_L\in \RR^{3\times 4}$, and 
	\eq{\label{xcn3}
	r=\frac{1}{2}
	\ba{c}
	0\\
	\mathds{1}_2
	\ea\in \RR^{4}, \quad
	\ell=
	\ba{cc}
	0\\
	\mathds{1}_2
	\ea\in \RR^{4}.
	}
	\rightline \qed
\end{lemma}
It is observed that $Q_d$ always has an eigenvalue at $1$, which implies that $Q_d$ is not stable no matter what the step-size $\mu$ is. However, this eigenvalue does not influence the convergence of recursions \eqref{special-ed}. To see that, from Lemma \ref{lm-Q_d-decom} we have
\eq{\label{ngy7}
\cQ_d &= \cX \overline{\cQ}_d \cX^{-1} =
\ba{cc}
\hspace{-2mm}R & \cX_R\hspace{-2mm}
\ea
\ba{cc}
I_M & 0\\
0 & \cE_d
\ea
\ba{c}
L\tran\\
\cX_L
\ea
}
where $\cX_R=X_R\otimes I_M$, $\cX_L=X_L\otimes I_M$, $\cE_d = E_d\otimes I_M$, and
\eq{
R = \frac{1}{2}
\ba{c}
0 \\
\mathds{1}_2\otimes I_M
\ea,
\quad 
L=
\ba{c}
0 \\
\mathds{1}_2\otimes I_M
\ea.
}
Let
\eq{
\ba{c}
\widehat{\sz}_i\\
\check{\sz}_i
\ea=\cX^{-1}
\tzd_i=
\ba{c}
L\tran \tzd_i\\
\cX_L \tzd_i
\ea.\label{h82}
}
The exact diffusion recursion \eqref{special-ed} can be transformed into 
\eq{
\ba{c}
\widehat{\sz}_i\\
\check{\sz}_i
\ea = 
\ba{cc}
I_M & 0\\
0 & \cE_d
\ea
\ba{c}
\widehat{\sz}_{i-1}\\
\check{\sz}_{i-1}
\ea,
}
which can be further divided into two separate recursions:
\eq{
\widehat{\sz}_i = \widehat{\sz}_{i-1}, \quad 
\check{\sz}_i = \cE_d \check{\sz}_{i-1}.
}
Therefore, $\widehat{\sz}_i=0$ if $\widehat{\sz}_0=0$. Since $\sy_0=\cV \sw_0$ and $\sy_o^\star \in \mathrm{range}(\cV)$, we have $\tyd_0 = \sy_o^\star - \sy_0 \in \mathrm{range}(\cV)$. Therefore,
\eq{
\widehat{\sz}_0 &\overset{\eqref{h82}}{=} L\tran \tzd_0 = 
\ba{cc}
0 & (\mathds{1}_2\otimes I_M)\tran 
\ea
\ba{c}
\twd_0\\
\tyd_0
\ea \overset{\eqref{xcn987-2}}{=} 0,
}
As a result, we only need to focus on the other recursion:
\eq{\label{cngw7}
\check{\sz}_i = \cE_d \check{\sz}_{i-1},\quad 
\mbox{where}
\quad
\cE_d = E_d \otimes I_M.
}
If we select the step-size $\mu$ such that all eigenvalues of $E_d$ stay inside the unit-circle, then we guarantee the convergence of $\check{\sz_i}$ and, hence, $\tzd_i$.

\begin{lemma}[\sc Stability of exact diffusion]\label{lm-sta-ed}
	When $\mu$ is chosen such that 
	\eq{\label{mu-range}
		0 < \mu \sigma^2 < 2,
		}
	all eigenvalues of $E_d$ will lie inside the unit-circle, which implies that $\tzd_i$ in \eqref{special-ed} converges to $0$, i.e., $\tzd_i \to 0$.
\end{lemma} 
\begin{proof}
	See Appendix \ref{app-sta-ed}.
\end{proof}
Next we turn to the EXTRA error recursion \eqref{special-extra}. 
\begin{lemma}[\sc Instability of EXTRA]\label{lm-sta-ex}
	When $\mu^e$ is chosen such that 
	\eq{\label{mu-range-extra}
		 \mu^e \sigma^2 \ge  a+1,
	}
	it holds that $\tzd_i^e$ generated through EXTRA \eqref{special-extra} will diverge.
\end{lemma} 
\begin{proof}
	See Appendix \ref{app-sta-extra}.
\end{proof}
Comparing the statements of Lemmas \ref{lm-sta-ed} and \ref{lm-sta-ex}, and since $1+a<2$, exact diffusion has a larger range of stability than EXTRA (i.e., exact diffusion is stable for a wider range of step-size values). In particular, {if agents place small weights on their own data}, i.e., when $a\approx 0$, the stability range for exact diffusion will be almost twice as large as that of EXTRA.} 

{\vspace{-2mm}
\section{Numerical Experiments}\label{sec-simulation}
In this section we compare the performance of the proposed exact diffusion algorithm with existing linearly convergent algorithms such as EXTRA\cite{shi2015extra}, DIGing\cite{nedich2016achieving}, and Aug-DGM\cite{xu2015augmented,nedic2016geometrically}. In all figures, the $y$-axis indicates the relative error, i.e.,  $\|\sw_i-\sw^o\|^2/\|\sw_0-\sw^o\|^2$, where $\sw_i=\col\{w_{1,i},\cdots,w_{N,i}\}\in \RR^{NM}$ and $\sw^o=\col\{w^o,\cdots,w^o\}\in \RR^{NM}$. All simulations employ the connected network topology
with $N = 20$ nodes shown in Fig.4 of Part I\cite{yuan2017exact1}.

\vspace{-4mm}
\subsection{Distributed Least-squares}\label{subsec:expe-ls}
In this experiment, we focus on the least-squares problem:
\eq{\label{prob-ls}
w^o = \argmin_{w\in \RR^M}\quad \frac{1}{2}\sum_{k=1}^{N} \|U_k w - d_k\|^2.
}
The simulation setting is the same as Sec. VI.A of Part I\cite{yuan2017exact1}.

In the simulation we compare exact diffusion with EXTRA, DIGing, and Aug-DGM. 
These algorithms work with symmetric doubly-stochastic or right-stochastic matrices $A$. Therefore, we now employ doubly-stochastic matrices for a proper comparison. 
Moreover, there are two information combinations per iteration in DIGing and Aug-DGM algorithms, and each information combination corresponds to one round of communication. In comparison, there is only one information combination (or round of communication) in EXTRA and exact diffusion. For fairness we will compare the algorithms based on the amount of communications, rather than the iterations. {\color{black}In the figures, we use one unit amount of communication to represent $2ME$ communicated variables, where $M$ is the dimension of the variable while $E$ is the number of edges in the network.} The problem setting is the same as in the simulations in Part I, except that $A$ is generated through the Metropolis rule \cite{sayed2014adaptation}. {\color{black}In the top plot in Fig. \ref{fig:4-algs}, all algorithms are carefully adjusted to reach their fastest convergence. It is observed that exact diffusion is slightly better than EXTRA, and both of them are more communication efficient than DIGing and Aug-DGM.} When a larger step-size $\mu=0.02$ is chosen for all algorithms, it is observed that EXTRA and DIGing diverge while exact diffusion and Aug-DGM converge, and exact diffusion is much faster than Aug-DGM algorithm.

{\color{black} We also compare exact diffusion with Push-EXTRA \cite{xi2015linear,zeng2015extrapush} and Push-DIGing \cite{nedich2016achieving} for non-symmetric combination policies. We consider the unbalanced network topology shown in Fig. 6 in Part I \cite{yuan2017exact1}. The combination matrix is generated through the averaging rule. Note that the Perron eigenvector $p$ is known beforehand for such combination matrix $A$, and we can therefore substitute $p$ directly into the recursions of Push-EXTRA and Push-DIGing. In the simulation, all algorithms are adjusted to reach their fastest convergence. In Fig. \ref{fig:5-algs}, it is observed that exact diffusion is the most communication efficient among all three algorithms. This figure illustrates that exact diffusion has superior performance for locally-balanced combination policies.
}

\begin{figure}
	\centering
	\includegraphics[scale=0.38]{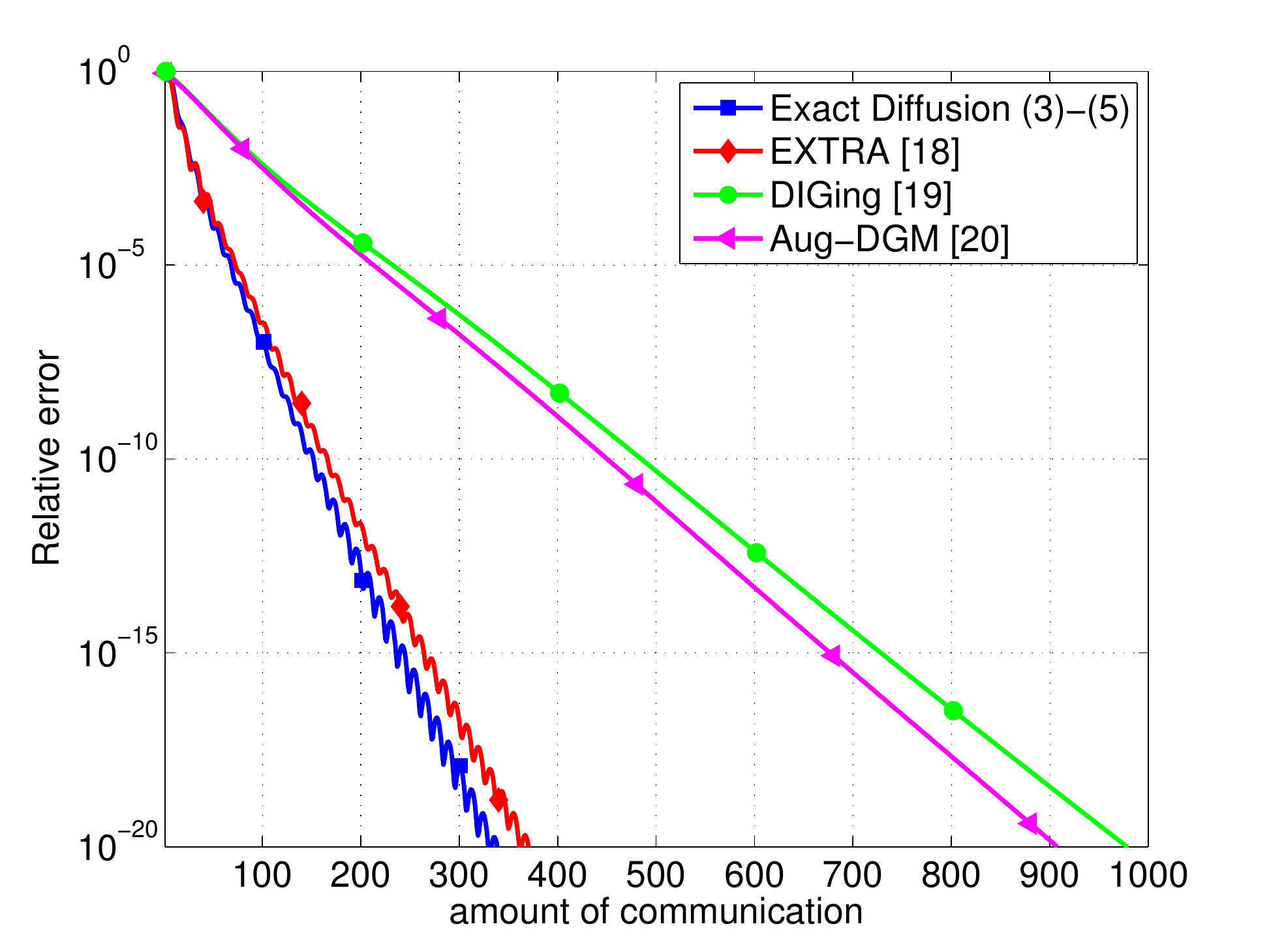}
	\includegraphics[scale=0.38]{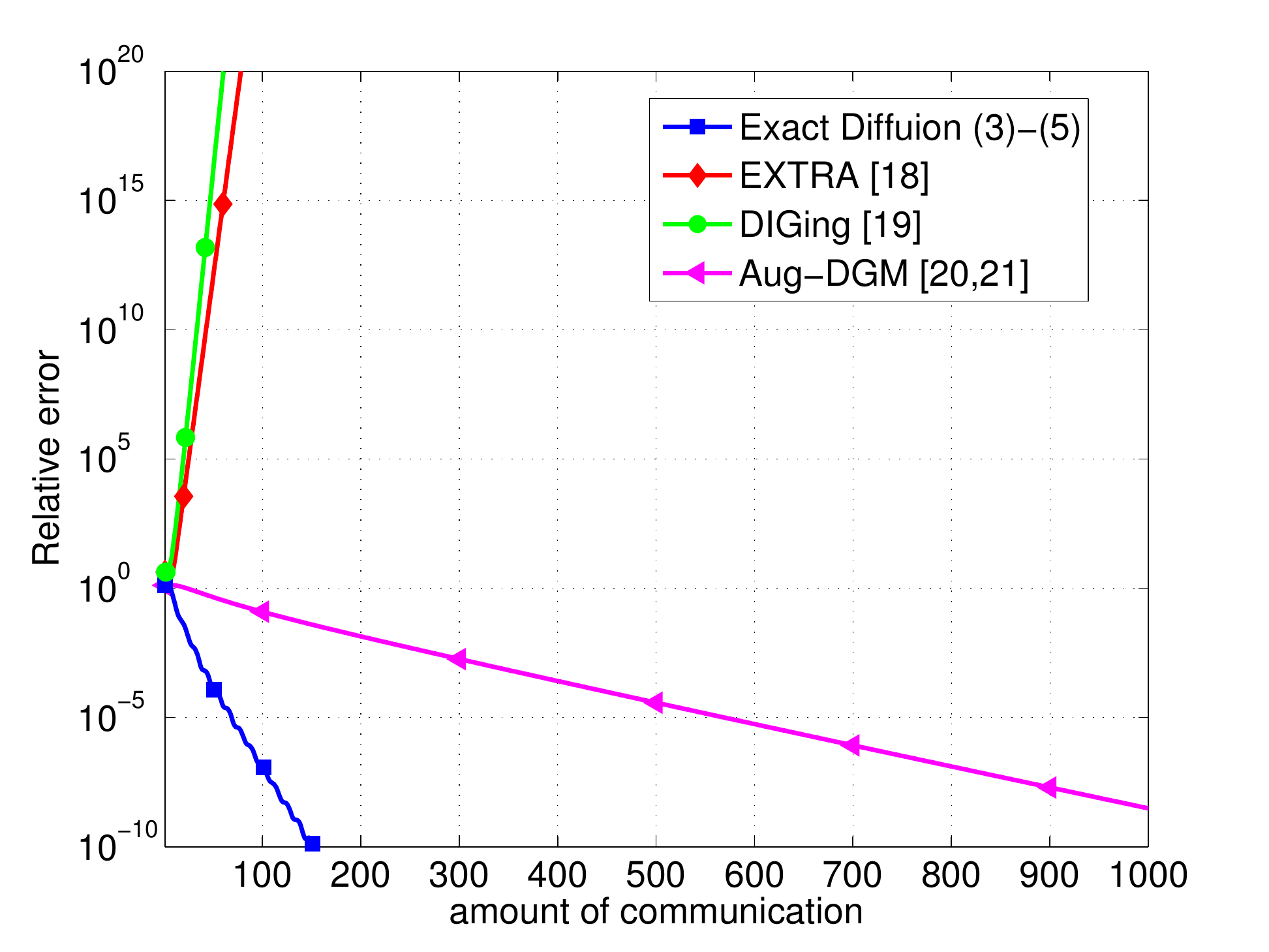}
	\vspace{-3mm}
	\caption{\footnotesize Convergence comparison between exact diffusion, EXTRA, DIGing, and Aug-DGM for distributed least-squares problem \eqref{prob-ls}. {\color{black}In the top plot, the step-sizes for Exact diffusion, EXTRA, DIGing and Aug-DGM are 0.013, 0.007, 0.0028 and 0.003. In the bottom plot, all step-sizes are set as 0.04.}}
	\label{fig:4-algs}
	\vspace{-5mm}
\end{figure} 

\begin{figure}
	\centering
	\includegraphics[scale=0.38]{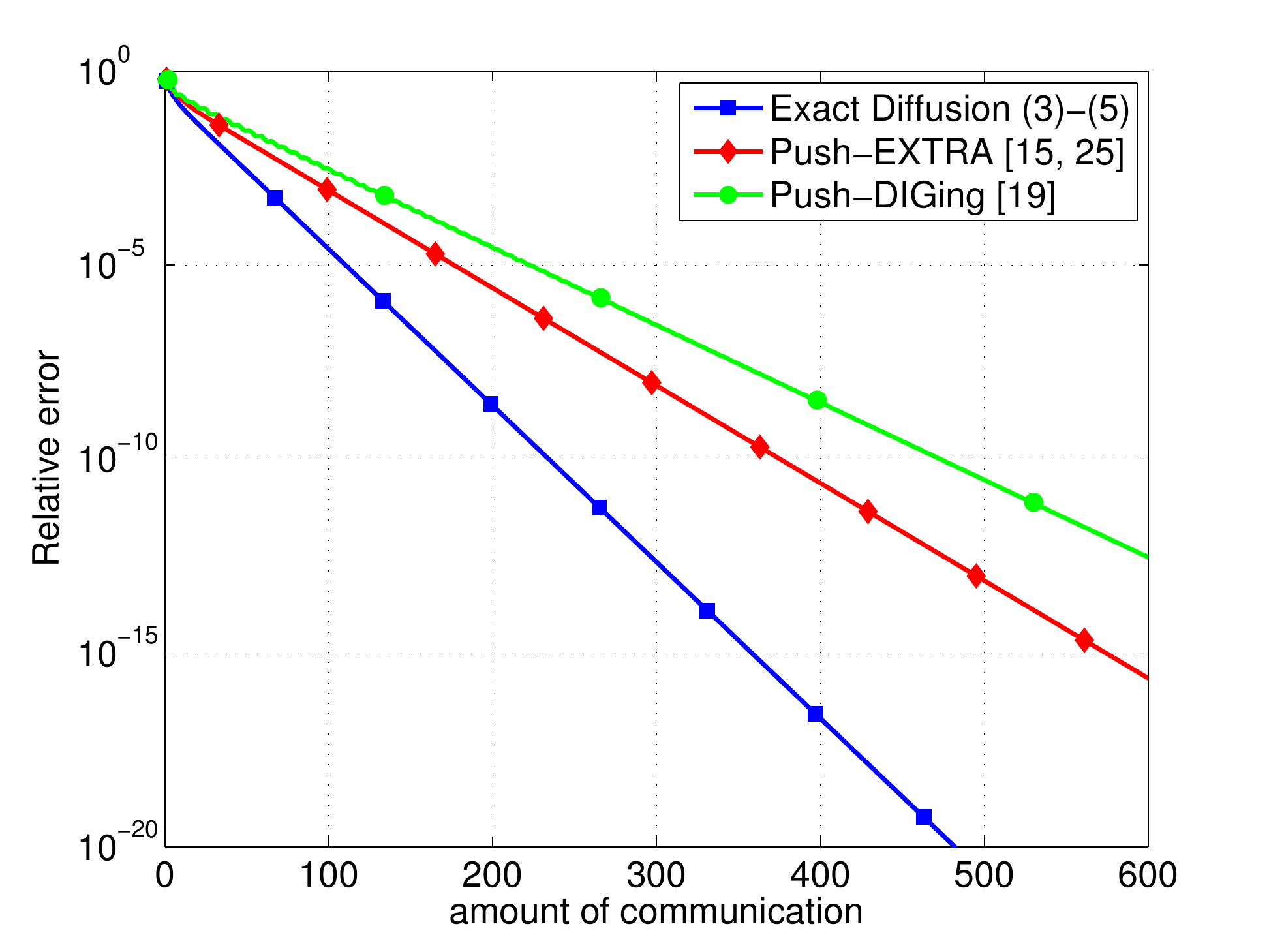}
	\vspace{-3mm}
	\caption{\footnotesize Convergence comparison between exact diffusion, EXTRA, DIGing, and Aug-DGM for distributed least-squares problem \eqref{prob-ls} with non-symmetric combination policy.}
	\label{fig:5-algs}
	\vspace{-5mm}
\end{figure}

}

\vspace{-4mm}
\subsection{Distributed Logistic Regression}
We next consider a pattern classification scenario. Each agent $k$ holds local data samples $\{h_{k,j}, \gamma_{k,j}\}_{j=1}^L$, where $h_{k,j}\in \RR^M$ is a feature vector and $\gamma_{k,j}\in \{-1,+1\}$ is the corresponding label. Moreover, the value $L$ is the number of local samples at each agent. All agents will cooperatively solve the regularized logistic regression problem:
\eq{\label{prob-lr}
w^o=\argmin_{w\in \RR^M} \sum_{k=1}^{N} \Big[\frac{1}{L}\sum_{\ell=1}^{L}\ln\big(1\hspace{-1mm}+\hspace{-1mm}\exp(-\gamma_{k,\ell} h_{k,\ell}\tran w)\big) \hspace{-1mm}+\hspace{-1mm} \frac{\rho}{2}\|w\|^2 \Big].
}
The simulation setting is the same as Sec. VI.B of Part I\cite{yuan2017exact1}.

In this simulation, we also compare exact diffusion with EXTRA, DIGing, and Aug-DGM. A symmetric doubly-stochastic $A$ is generated through the Metropolis rule. {\color{black}In the top plot in Fig. \ref{fig:lr-4-algs}, all algorithms are carefully adjusted to reach their fastest convergence. It is observed that exact diffusion is the most communication efficient among all algorithms.} When a larger step-size $\mu=0.04$ is chosen for all algorithms in the bottom plot in Fig. \ref{fig:lr-4-algs}, it is observed that both exact diffusion and Aug-DGM are still able to converge linearly to $w^o$, while EXTRA and DIGing fail to do so. Moreover, exact diffusion is observed much more communication efficient than Aug-DGM.
\begin{figure}
	\centering
	\includegraphics[scale=0.39]{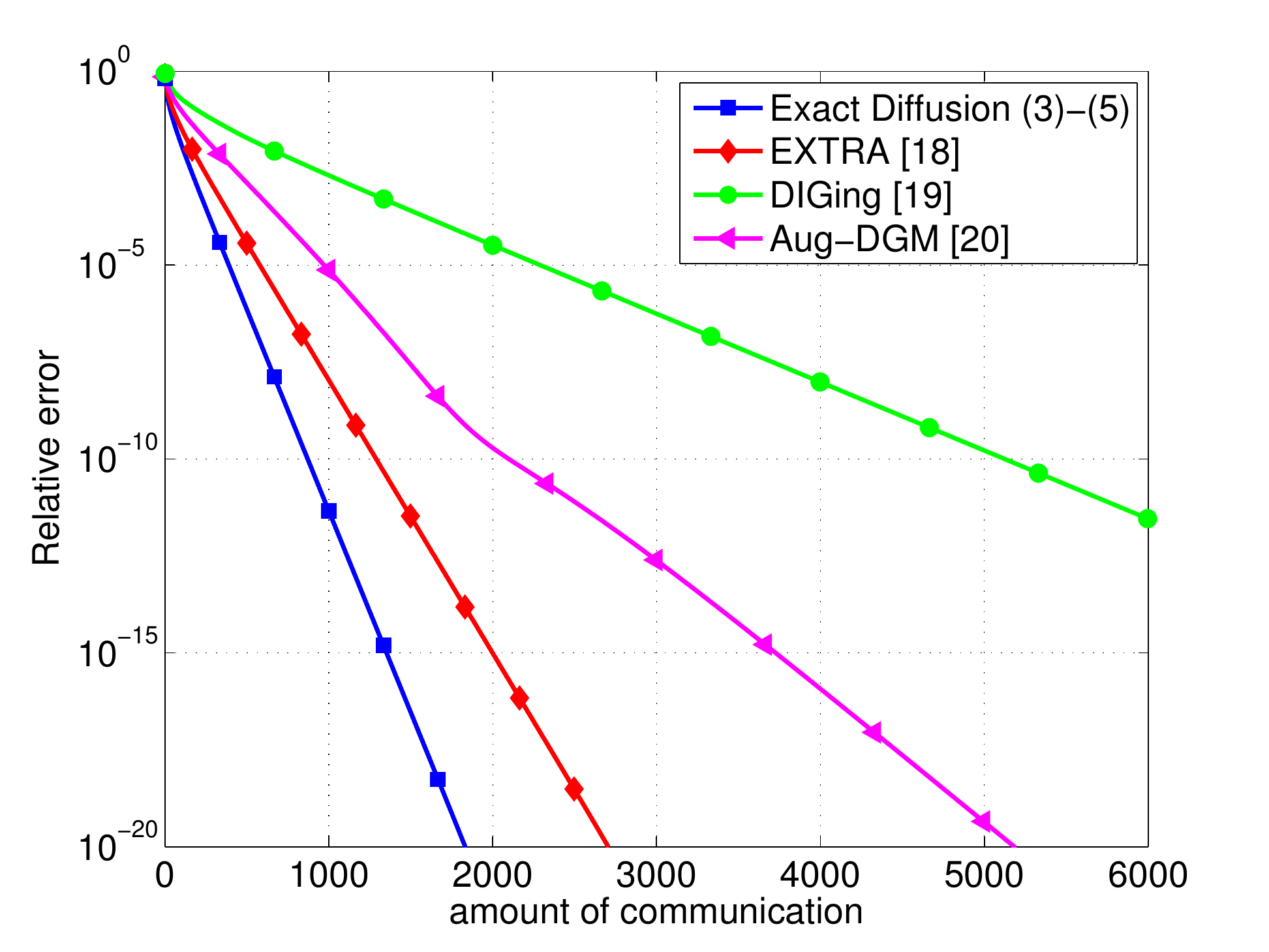}
	\includegraphics[scale=0.39]{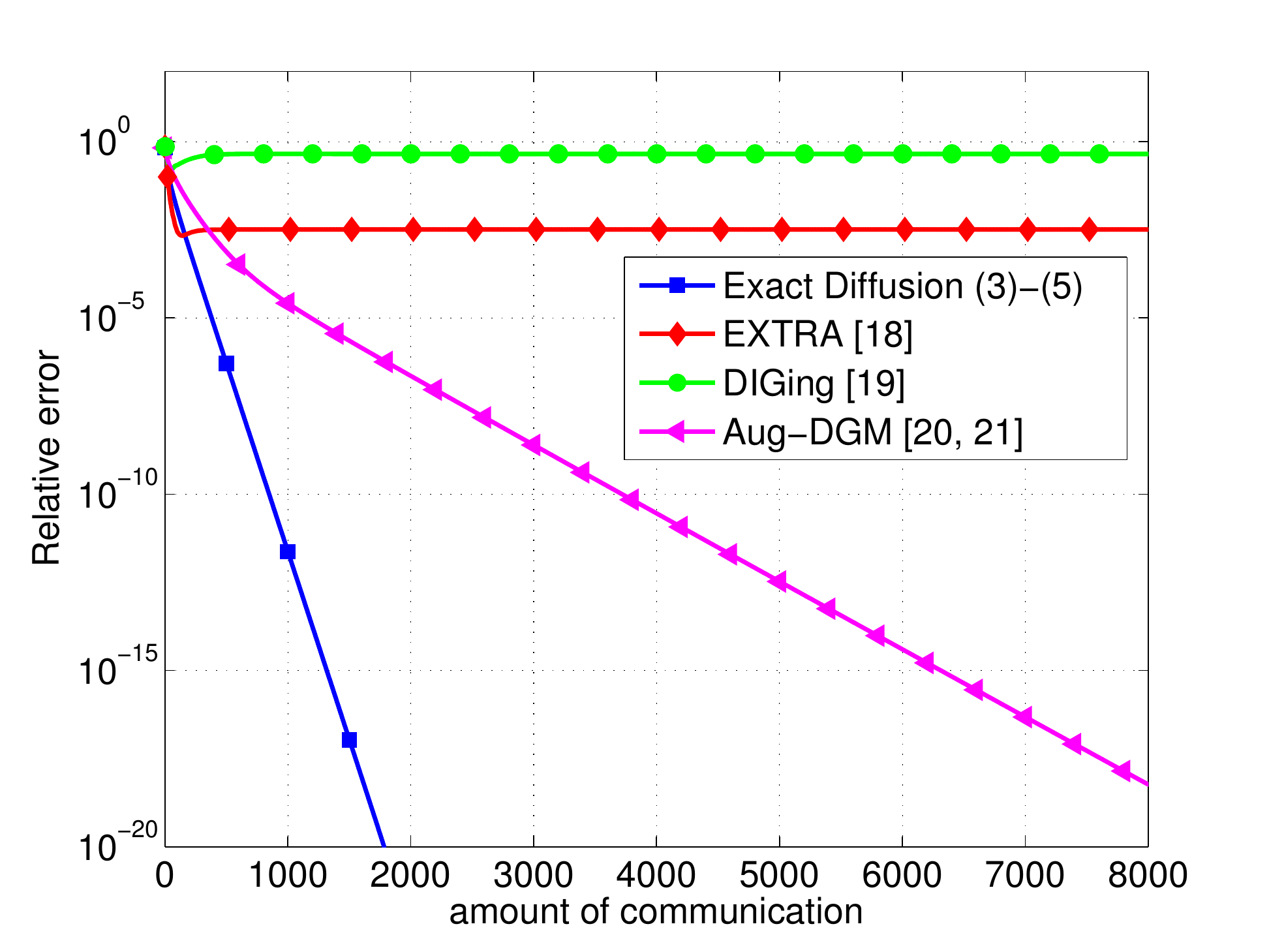}
	\vspace{-3mm}
	\caption{\footnotesize Convergence comparison between exact diffusion, EXTRA, DIGing, and Aug-DGM for problem \eqref{prob-lr}. {\color{black}In the top plot, the step-sizes for Exact Diffusion, EXTRA, DIGing and AUG-DGM are 0.041, 0.028, 0.014 and 0.033.} In the bottom plot, all step-sizes are set as 0.04.}
	\label{fig:lr-4-algs}
	\vspace{-5mm}
\end{figure}


\appendices
\section{Proof of Lemma \ref{lm-B-decomposition}}\label{appdx-Lemma-fd}

	Define $V^\prime \define V + \mathds{1}_N\, p\tran \in \RR^{N\times N}$,
	we claim that $V^\prime$ is a full rank matrix. Suppose to the contrary that there exists some $x\neq 0$ such that $V'x=0$, i.e.,$(V + \mathds{1}_N\, p\tran) x = Vx + (p\tran x)\mathds{1}_N=0,$
	which requires
	\eq{\label{xch388}
		Vx = - (p\tran x)\mathds{1}_N.
	}
	When $p\tran x \neq 0$, relation \eqref{xch388} implies that $\mathds{1}_N \in \mathrm{range}(V)$. {However, from Lemma \ref{lm:null-V} we know that
	\eq{
		\mathrm{null}(V)=\mathrm{span}\{\mathds{1}_N\} 
		\Longleftrightarrow&\ \mathrm{range}(V\tran)^\perp = \mathrm{span}\{\mathds{1}_N\} \nnb
		\Longleftrightarrow&\ \mathrm{range}(V)^\perp = \mathrm{span}\{\mathds{1}_N\}, \label{i2378d}
	}
	where the last ``$\Leftrightarrow$" holds because $V$ is symmetric.} Relation \eqref{i2378d} is contradictory to $\mathds{1}_N \in \mathrm{range}(V)$. Therefore, $V^\prime x \neq 0$. When $p\tran x =0$, relation \eqref{xch388} implies that $Vx=0$, which together with Lemma \ref{lm:null-V} implies that $x=c \mathds{1}_N$ for some constant $c\neq 0$. However, since $p\tran \mathds{1}_N=1$, we have $p\tran x = c \neq 0$, which also contradicts with $p\tran x =0$. As a result, $V'$ has full rank and hence $\left(V'\right)^{-1}$ exists.


	
	With $V^\prime=V + \mathds{1}_N$ and the fact $V\mathds{1}_N = 0$ {(see Lemma \ref{lm:null-V})}, we also have
	\eq{
		& VV^\prime =V (V + \mathds{1}_N\, p\tran) = V^2 + V \mathds{1}_N\, p\tran = V^2, \label{VVprime-1}\\
		& V^\prime (I_N - \mathds{1}_N\, p\tran) = (V + \mathds{1}_N\, p\tran)(I_N - \mathds{1}_N\, p\tran)=V.\label{VVprime-2}
	}
	With relations \eqref{VVprime-1} and \eqref{VVprime-2}, we can verify that 
	\eq{\label{B-deco-1}
		B \hspace{-1mm}&= \hspace{-1mm}
		\ba{cc}
		\hspace{-2mm}I_N\hspace{-1mm} & \hspace{-1mm}0\hspace{-2mm}\\
		\hspace{-2mm}0\hspace{-1mm} & \hspace{-1mm}{V}^\prime\hspace{-2mm}
		\ea\hspace{-1.5mm}
		\ba{cc}
		\hspace{-2mm}\tA\tran \hspace{-1mm} & \hspace{-0.5mm} - P^{-1}V^2\hspace{-2mm} \\
		\hspace{-2mm}(V^\prime)^{-\hspace{-0.4mm}1}V\tA\tran \hspace{-0.5mm} & \hspace{-1mm} I_N \hspace{-0.8mm}-\hspace{-0.8mm} (V^\prime)^{-\hspace{-0.4mm}1} V P^{-\hspace{-0.4mm}1} V^2\hspace{-2mm}
		\ea\hspace{-1.5mm}
		\ba{cc}
		\hspace{-2mm}I_N\hspace{-1mm} & \hspace{-1mm}0\hspace{-2mm}\\
		\hspace{-2mm}0\hspace{-1mm} & \hspace{-1mm}({V}^\prime)^{-1}\hspace{-2mm}
		\ea \nnb
		&\overset{(a)}{=} \hspace{-1mm}
		\ba{cc}
		\hspace{-2mm}I_N\hspace{-1mm} & \hspace{-1mm}0\hspace{-2mm}\\
		\hspace{-2mm}0\hspace{-1mm} & \hspace{-1mm}{V}^\prime\hspace{-2mm}
		\ea\hspace{0mm}
		\ba{cc}
		\hspace{-2mm}\tA\tran \hspace{-1mm} & \hspace{-0.5mm} \tA\tran - I_N \hspace{-2mm} \\
		\hspace{-2mm}\tA\tran - \mathds{1}_N\, p\tran \hspace{-0.5mm} & \hspace{-1mm} \tA\tran \hspace{-2mm}
		\ea\hspace{0mm}
		\ba{cc}
		\hspace{-2mm}I_N\hspace{-1mm} & \hspace{-1mm}0\hspace{-2mm}\\
		\hspace{-2mm}0\hspace{-1mm} & \hspace{-1mm}({V}^\prime)^{-1}\hspace{-2mm}
		\ea
	}
	where in (a) we used $V^2\hspace{-1mm}=\hspace{-1mm}(P \hspace{-1mm}-\hspace{-1mm} PA)/2$ and $\tA\tran \hspace{-1mm}=\hspace{-1mm} (I_N \hspace{-1mm}+\hspace{-1mm} A\tran)/2$. Using $A=Y\Lambda Y^{-1}$ from Lemma 3 of Part I\cite{yuan2017exact1}, we have
	%
	\eq{\label{m88}
		\tA\tran = (Y^{-1})\tran \overline{\Lambda} Y\tran,\quad 
		\tA\tran \hspace{-1mm}-\hspace{-1mm} I_N \hspace{-0.5mm}=\hspace{-0.5mm} (Y^{-1})\tran (\overline{\Lambda} \hspace{-1mm}-\hspace{-1mm} I_N) Y\tran
	}
	where $\overline{\Lambda} \define  \left(I_N + \Lambda\right)/2$. Obviously, $\overline{\Lambda}>0$ is also a real diagonal matrix. If we let $\overline{\Lambda}=\mathrm{diag}\{\lambda_1(\tA),\cdots, \lambda_N(\tA)\}$, it holds that
	\eq{\label{tA&tA-I}
		\lambda_k(\tA) = (\lambda_k(A)+1)/{2} >0,\quad \forall\, k =1,\cdots, N,
	}
	and $\lambda_1(\tA)=1$. Moreover, we can also verify that 
	\eq{\label{tA-1p'}
		\tA\tran - \mathds{1}_N\, p\tran = (Y^{-1})\tran \overline{\Lambda}_1
		Y\tran,
	}
	where $
		\overline{\Lambda}_1={\rm diag}\{0, \lambda_2(\tA),\cdots, \lambda_N(\tA)\}.$
	This is because the vectors $\mathds{1}_N\tran$ and $p$ are the left- and right-eigenvectors of $\tA$.
	Combining relations \eqref{tA&tA-I} and \eqref{tA-1p'}, we have
	\eq{\label{B-deco-2}
		& \ba{cc}
		\hspace{-2mm}\tA\tran \hspace{-1mm} & \hspace{-0.5mm} \tA\tran - I_N \hspace{-2mm} \\
		\hspace{-2mm}\tA\tran - \mathds{1}_N\, p\tran \hspace{-0.5mm} & \hspace{-1mm} \tA\tran \hspace{-2mm}
		\ea \nnb
		=&
		\ba{cc}
		\hspace{-2mm}(Y^{-1})\tran\hspace{-1mm} & \hspace{-1mm}0\hspace{-2mm}\\
		\hspace{-2mm}0\hspace{-1mm} & \hspace{-1mm}(Y^{-1})\tran\hspace{-2mm}
		\ea\hspace{0mm}
		\ba{cc}
		\hspace{-2mm}\overline{\Lambda} \hspace{-1mm} & \hspace{-0.5mm} \overline{\Lambda} - I_N \hspace{-2mm} \\
		\hspace{-2mm}\ \overline{\Lambda}_1 \hspace{-0.5mm} & \hspace{-1mm} \overline{\Lambda} \hspace{-2mm}
		\ea\hspace{0mm}
		\ba{cc}
		\hspace{-2mm}Y\tran\hspace{-1mm} & \hspace{-1mm}0\hspace{-2mm}\\
		\hspace{-2mm}0\hspace{-1mm} & \hspace{-1mm} Y\tran \hspace{-2mm}
		\ea.
	}
	With permutation operations, it holds that
	\eq{\label{B-deco-3}
		\ba{cc}
		\hspace{-2mm}\overline{\Lambda} \hspace{-1mm} & \hspace{-0.5mm} \overline{\Lambda} - I_N \hspace{-2mm} \\
		\hspace{-2mm}\ \overline{\Lambda}_1 \hspace{-0.5mm} & \hspace{-1mm} \overline{\Lambda} \hspace{-2mm}
		\ea
		=&
		\Pi
		\ba{cccc}
		E_1 & 0 & \cdots & 0 \\
		0 & E_2 & \cdots & 0 \\
		\vdots & \vdots & \ddots & \vdots \\
		0 & 0 & \cdots & E_N
		\ea
		\Pi\tran,
	}
	where $\Pi \in \RR^{N\times N}$ is a permutation matrix, and
	\vspace{1mm} 
	\eq{
		E_1 \hspace{-1mm}=\hspace{-1mm} 
		\ba{cc}
		\hspace{-2mm}1 & 0 \hspace{-2mm}\\
		\hspace{-2mm}0 & 1\hspace{-2mm}
		\ea\hspace{-1mm}, \ 
		E_k\hspace{-1mm}=\hspace{-1mm}
		\ba{cc}
		\hspace{-2mm}\lambda_k(\tA) & \lambda_k(\tA)-1\hspace{-2mm} \\
		\hspace{-2mm}\lambda_k(\tA) & \lambda_k(\tA)\hspace{-2mm}
		\ea\hspace{-1mm},\ \forall k = 2,\cdots, N. 
	}
	Now we seek the eigenvalues of $E_k$. Let $d$ denote an eigenvalue of $E_k$. The characteristic polynomial of $E_k$ is 
	\eq{
		d^2-2\lambda_k(\tA) d + \lambda_k(\tA) = 0. 
	}
	Therefore, we have
	\eq{
		d=\frac{2\lambda_k(\tA) \pm \sqrt{4\lambda_k^2(\tA) - 4 \lambda_k(\tA)}}{2}. 
	}
	Since $\lambda_k(\tA)\in(0,1)$ when $k=2,3,\cdots,N$, it holds that  $4\lambda_k^2(\tA) < 4 \lambda_k(\tA)$. Therefore, $d$ is a complex number, and its magnitude is $\sqrt{\lambda_k(\tA)}$. Therefore, $E_k$ can be diagonalized as 
%
	\eq{\label{B-deco-5}
		E_k = Z_k 
		\ba{cc}
		d_{k,1} & 0\\
		0 & d_{k,2}
		\ea
		Z_k^{-1}
	}
	where $d_{k,1}$ and $d_{k,2}$ are complex numbers and 
	\eq{\label{B-deco-6}
		|d_{k,1}| = |d_{k,2}| = \sqrt{\lambda_k(\tA)} < 1.
	}
	Define $Z$ and $\overline{X}$ as
	\eq{
	Z\define& \diag\{I_2, Z_2,Z_3, \cdots, Z_N\}
	}
	\eq{
	\overline{X} \define& \ba{cc}I_N&0\\0&V'\ea \ba{cc}(Y^{-1})^{\sf T}&0\\0& (Y^{-1})^{\sf T}\ea \Pi\, Z
	}
	Since each factor in $\overline{X}$ is invertible, $\overline{X}^{\hspace{0.2mm}-1}$ must exist. Combining \eqref{B-deco-1} and \eqref{B-deco-2}--\eqref{B-deco-6}, we finally  arrive at 
	\eq{\label{xcnh26}
	B = \overline{X} D \overline{X}^{-1}, \mbox{ where } D = \ba{cc} I_2 & 0\\0 & D_1 \ea,
	}
	and $D_1$ has the structure claimed in \eqref{uwehn}.\vspace{-1mm}

	Therefore, we have established so far the form of the eigenvalue decomposition of $B$. In this decomposition, each $k$-th column of $\overline{X}$ is a right-eigenvector associated with the eigenvalue $D(k,k)$, and each $k$-th row of $\overline{X}^{-1}$ is the left-eigenvector associated with $D(k,k)$. Recall, however, that eigenvectors are not unique. We now verify that we can find eigenvector matrices $\overline{X}$ and $\overline{X}^{-1}$ that have the structure shown in \eqref{X and X_inv} and \eqref{R and L}. To do so, it is sufficient to examine whether the two columns of $R$ are independent right-eigenvectors associated with eigenvalue $1$, and the two rows of $L$ are independent left-eigenvectors associated with $1$. Let 
	\eq{\label{tsdhb-app-0}
		R=
		\ba{cc}
		\hspace{-1mm}r_1 & r_2\hspace{-1mm}
		\ea,\mbox{ where }\ 
		r_1 \hspace{-1mm} \define \hspace{-1mm}
		\ba{c}
		\hspace{-1.8mm}\mathds{1}_N \hspace{-1.8mm}\\
		\hspace{-1.8mm}0\hspace{-1.8mm}
		\ea, \quad 
		r_2 \hspace{-1mm} \define \hspace{-1mm}
		\ba{c}
		\hspace{-1.8mm}0\hspace{-1.8mm}\\
		\hspace{-1.8mm}\mathds{1}_N\hspace{-1.8mm}
		\ea.
	}
	Obviously, $r_1$ and $r_2$ are independent. Since
	\eq{
		Br_1 = r_1, \quad Br_2 = r_2,
	}
	we know $r_1$ and $r_2$ are right-eigenvectors associated with eigenvalue $1$. As a result, an eigenvector matrix $X$ can be chosen in the form $
	X=
	\ba{ccc}
	R& \vline & X_R
	\ea,
	$
	where each $k$-th column of $X_R$ corresponds to the right-eigenvector associated with eigenvalue $D_1(k,k)$. Similarly, we let 
	\eq{\label{tsdhb-app-0-L}
		L=
		\ba{c}
		\hspace{-1mm}\ell_1\tran\hspace{-1mm}\\
		\hspace{-1mm}\ell_2\tran\hspace{-1mm}
		\ea, \mbox{ where }\ 
		\ell_1\define
		\ba{c}
		\hspace{-1mm}p\hspace{-1mm}\\
		\hspace{-1mm}0\hspace{-1mm}
		\ea,
		\ell_2\define
		\ba{c}
		\hspace{-1mm}0\hspace{-1mm}\\
		\hspace{-1mm}\frac{1}{N}\mathds{1}_N\hspace{-1mm}
		\ea.
	}
	It is easy to verify that  $\ell_1$ and $\ell_2$ are independent left-eigenvectors associated with eigenvalue $1$. Moreover, since $LR=I_2$, $X^{-1}$ has the structure
	\eq{\label{nh289}
	X^{-1}=\ba{c}
	L\\
	X_L
	\ea,
	}
	where each $k$-th row of $X_L$ corresponds to a left-eigenvector associated with eigenvalue $D_1(k,k)$.

\section{Proof of Theorem \ref{theom-convergence}}
	\label{app-theom-conv}
From the first line of recursion \eqref{final-recursion}, we have
\eq{\label{x_bar-recursion}
	\hspace{-1mm}
	\bar{\sx}_i \hspace{-1mm}= \hspace{-1mm} \left(I_M \hspace{-1mm} - \hspace{-1mm} {\overline{\cP}\tran \hspace{-0.8mm}\cM\cH_{i\hspace{-0.3mm}-\hspace{-0.3mm}1}\cI}\right)\bar{\sx}_{i\hspace{-0.3mm}-\hspace{-0.3mm}1} \hspace{-1mm} - \hspace{-1mm}\frac{1}{c}
	\overline{\cP}\tran \hspace{-0.8mm}\cM\cH_{i-1}\cX_{R,u} \check{\sx}_{i-1}.
} 
Squaring both sides and using Jensen's inequality \cite{boyd2004convex} gives
\eq{\label{x_bar_square}
	\hspace{-1mm} \|\bar{\sx}_i\|^2 \hspace{-1mm}=\hspace{-1mm}&\ \left\|\hspace{-1mm}\left(I_M \hspace{-1mm} - \hspace{-1mm} {\overline{\cP}\tran \hspace{-0.8mm}\cM\cH_{i\hspace{-0.3mm}-\hspace{-0.3mm}1}\cI}\right) \bar{\sx}_{i-1} \hspace{-1mm}-\hspace{-1mm} \frac{1}{c}
	\overline{\cP}\tran\cM\cH_{i-1}\cX_{R,u} \check{\sx}_{i-1}\hspace{-0.5mm}\right\|^2 \nnb
	\le &\ \frac{1}{1-t}\left\|I_M \hspace{-1mm} - \hspace{-1mm} {\overline{\cP}\tran \hspace{-0.8mm}\cM\cH_{i\hspace{-0.3mm}-\hspace{-0.3mm}1}\cI}\right\|^2\|\bar{\sx}_{i-1}\|^2 \nnb
	&\quad + \frac{1}{t}\frac{1}{c^2}\|\overline{\cP}\tran\cM\cH_{i-1}\cX_{R,u}\|^2 \|\check{\sx}_{i-1}\|^2
}
for any $t\in (0,1)$. Using $\tau_k=\mu_k/\mu_{\max}$, we obtain
\eq{
	\hspace{-2mm} {\overline{\cP}\tran\cM\cH_{i-1}\cI} =&\ \mu_{\max}\sum_{k=1}^{N}p_k \tau_k H_{k,i-1}\vspace{-5mm} \nnb
	\hspace{2mm} \overset{\eqref{H-properties}}{\ge} &\ {\mu_{\max}p_{k_o}\tau_{k_o}\nu I_M} \define {\sigma_{11} \mu_{\max}I_M},
}
where $\sigma_{11} = p_{k_o} \tau_{k_o}\nu$. Similarly, we can also obtain
\eq{
	{\overline{\cP}\tran\cM\cH_{i-1}\cI} & = \mu_{\max}\sum_{k=1}^{N}p_k \tau_k H_{k,i-1}\nnb
	&\overset{\eqref{H-properties}}{\le} \left(\sum_{k=1}^{N}p_k \tau_k \right)\delta\mu_{\max}  I_M \hspace{-1.5mm} \overset{(a)}{\le} \delta\mu_{\max}  I_M,
}
where inequality $(a)$ holds because $\tau_k<1$ and $\sum_{k=1}^{N}p_k=1$. It is obvious that $\delta >\sigma_{11}$. As a result, we have
\eq{\label{xzcn8237}
	\hspace{-2mm}(1\hspace{-0.8mm}-\hspace{-0.8mm}\delta \mu_{\max})I_M \hspace{-0.8mm}\le\hspace{-0.8mm} I_{M} \hspace{-0.8mm}-\hspace{-0.8mm} {\overline{\cP}\tran\cM\cH_{i-1}\cI} \le (1\hspace{-0.8mm}-\hspace{-0.8mm}\sigma_{11}\mu_{\max})I_M
}
which implies that when the step-size  satisfy  
\eq{\label{xchayu}
	\mu_{\max}<1/\delta,
} 
it will hold that
\eq{\label{uiox98}
	&\ \|I_{M} \hspace{-0.8mm}-\hspace{-0.8mm} {\overline{\cP}\tran\cM\cH_{i-1}\cI} \|^2 \le  (1-\sigma_{11}\mu_{\max})^2.
}
On the other hand, we have
\eq{\label{bguj87}
	\frac{1}{c^2}\|\overline{\cP}\tran\cM\cH_{i-1}\cX_{R,u}\|^2&\ \le \frac{1}{c^2}\|\overline{\cP}\tran\cM\|^2 \|\cH_{i-1}\|^2 \|\cX_{R,u}\|^2 \nn
}
\eq{
	&\ \overset{(a)}{\le} \frac{1}{c^2}\left(\sum_{k=1}^{N}(\tau_k p_k)^2\right)\delta^2 \|\cX_{R,u}\|^2 \mu^2_{\max} \nnb
	&\ \overset{(b)}{\le} \frac{p_{\max}}{c^2} \delta^2 \|\cX_{R,u}\|^2 \mu^2_{\max}
}
where 
 inequality (b) holds because $\tau_k<1$, $p_k^2 < p_k p_{\max}$ (where $p_{\max}= \max_k\{p_k\}$) and $\sum_{k=1}^{N}p_k=1$. {Inequality (a) follows by noting that $\overline{\cP}\tran \cM = \mu_{\max}[p_1\tau_1,\cdots,p_N \tau_N]\otimes I_M$.
Introducing $s=[p_1\tau_1,p_2\tau_2,\cdots,p_N \tau_N]\tran\in \RR^{N}$, we have
\eq{
\|\overline{\cP}\tran\hspace{-1mm} \cM\|^2 \hspace{-1mm}=& \mu^2_{\max}\|s\tran\hspace{-1mm} \otimes\hspace{-0.7mm} I_M\|^2\hspace{-1mm}=\hspace{-1mm}\mu^2_{\max}\lambda_{\max}\Big(\hspace{-0.5mm}(s\otimes I_M)(s\tran \otimes I_M)\hspace{-0.5mm}\Big)\nnb
=& \mu^2_{\max}\lambda_{\max}\Big(s s\tran \otimes I_M\Big) = \mu^2_{\max}\lambda_{\max}(s s\tran) \nnb
=& \mu_{\max}^2 \|s\|^2 =\mu_{\max}^2 {\sum_{k=1}^{N}}(p_k\tau_k)^2. 
}
Recall \eqref{usd9} and by introducing $E=\ba{cc}
I_{MN} & 0_{MN}
\ea$,
we have $\cX_{R,u}=E \cX_R$. Therefore, it holds that
\eq{\label{xcn398}
\|\cX_{R,u}\|^2 \le \|E\|^2 \|\cX_R\|^2=\|\cX_R\|^2.
}
Substituting \eqref{xcn398} into \eqref{bguj87}, we have
\eq{\hspace{-2mm}
\frac{1}{c^2}\|\overline{\cP}\tran\hspace{-2mm}\cM\cH_{i-1}\cX_{R,u}\hspace{-0.5mm}\|^2 \hspace{-1mm}\le\hspace{-1mm} \frac{p_{\max}\delta^2}{c^2}\hspace{-0.5mm} \|\cX_{R}\|^2 \mu^2_{\max} \hspace{-1mm}\define\hspace{-1mm} \sigma_{12}^2 \mu^2_{\max}
}
where $\sigma_{12} \define \sqrt{p_{\max}}\delta \|\cX_R\|/c$.} Notice that $\sigma_{12}$ is independent of $\mu_{\max}$. Substituting \eqref{uiox98} and \eqref{bguj87} into \eqref{x_bar_square}, we get
\eq{\label{x_bar_square-2}
	\|\bar{\sx}_i\|^2 \hspace{-1mm}\le\hspace{-1mm} &\ \frac{1}{1-t}(1-\sigma_{11}\mu_{\max})^2\|\bar{\sx}_{i-1}\|^2 + \frac{1}{t} \sigma^2_{12}\mu_{\max}^2 \|\check{\sx}_{i-1}\|^2 \nnb
	\hspace{-1mm}=\hspace{-0.2mm} & (1\hspace{-1mm}-\hspace{-1mm}\sigma_{11}\mu_{\max})\|\hspace{-0.3mm}\bar{\sx}_{i-1}\hspace{-0.3mm}\|^2 \hspace{-1mm}+\hspace{-1mm} ({\sigma^2_{12}}/{\sigma_{11}})\mu_{\max}\|\hspace{-0.3mm}\check{\sx}_{i-1}\hspace{-0.3mm}\|^2
}
where we are selecting $t = \sigma_{11}\mu_{\max}$.

Next we check the second line of recursion \eqref{final-recursion}:
\eq{
	\check{\sx}_i =&\; -c\cX_L\cT_{i-1}\cR_1 \bar{\sx}_{i-1}+ (\cD_1 - \cX_L\cT_{i-1}\cX_R)\check{\sx}_{i-1}\nnb
	=&\; \cD_1\check{\sx}_{i-1} - \cX_L\cT_{i-1}(c\cR_1 \bar{\sx}_{i-1}
	+ \cX_R\check{\sx}_{i-1}). \label{yujs99}
}
Squaring both sides and using Jensen's inequality again, 
\eq{
	\|\check{\sx}_i\|^2 =& \footnotesize\|\cD_1 \check{\sx}_{i-1} - \cX_L\cT_{i-1}(c\cR_1 \bar{\sx}_{i-1}
	+ \cX_R\check{\sx}_{i-1})\|^2\nn\\
	\le & \frac{\|\cD_1\|^2}{t} \|\check{\sx}_{i-1}\|^2 \hspace{-0.8mm}+\hspace{-0.8mm} \frac{1}{1-t}\| \cX_L\cT_{i-1}\hspace{-0.5mm}(\hspace{-0.5mm}c\cR_1 \bar{\sx}_{i-1}
	\hspace{-1mm}+\hspace{-1mm} \cX_R\check{\sx}_{i-1}\hspace{-0.5mm})\hspace{-0.5mm}\|^2	\nn}
\eq{
\le & \frac{\|\cD_1\|^2}{t} \|\check{\sx}_{i-1}\|^{\hspace{-0.3mm}2} \hspace{-1mm}+\hspace{-1mm} \frac{2c^2}{1-t} \| \cX_L\cT_{i-1}\cR_1\|^2 \|\bar{\sx}_{i-1}\|^2\nn\\
	&\hspace{0.63cm}+ \frac{2}{1-t} \|\cX_L\cT_{i-1}\cX_R\|^2 \| \check{\sx}_{i-1}\|^2.
}
where $t\in(0,1)$. From Lemma \ref{lm-B-decomposition} we have that $\lambda \define \|D_1\| = \sqrt{\lambda_2(\tA)}<1$. By setting $t=\lambda$, we reach
\eq{\label{yuwe9}
	\|\check{\sx}_i\|^2 
	\leq&\ \lambda \|\check{\sx}_{i-1}\|^2 + 2c^2 \| \cX_L\cT_{i-1}\cR_1\|^2 \|\bar{\sx}_{i-1}\|^2/(1-\lambda)\nn\\
	&\;\;\;\;\;\hspace{1.03cm}+ {2} \|\cX_L\cT\cX_R\|^2 \| \check{\sx}_{i-1}\|^2/({1-\lambda}).
}
We introduce the matrix $\Gamma={\rm diag}\{\tau_1 I_M,\cdots, \tau_N I_M\}$, and note that we can write $\cM=\mu_{\max}\Gamma$. Substituting it into \eqref{T-defi}, 
\eq{\label{xha9867}
	\cT_{i-1} &= \mu_{\max}
	\ba{cc}
	\tcA\tran \Gamma \cH_{i-1} & 0 \\
	\cV \tcA\tran \Gamma \cH_{i-1} & 0 
	\ea \nnb
	&= \mu_{\max} 
	\underbrace{\ba{cc}
		\tcA\tran & 0 \\
		\cV \tcA\tran & 0
		\ea}_{\define \cT_d}
	\ba{cc}
	\Gamma \cH_{i-1} & 0 \\
	0 & \Gamma \cH_{i-1}
	\ea,
	%
}
which implies that
\eq{\label{nui89}
	\|\cT_{i-1}\|^2 \le \mu_{\max}^2 \|\cT_d\|^2 \left( \max_{1\le k\le N} \|H_{k,i-1}\|^2 \right) \le  \|\cT_d\|^2 \delta^2 \mu_{\max}^2.
}
We also emphasize that $\|\cT_d\|^2$ is independent of $\mu_{\max}$. With inequality \eqref{nui89}, we further have
\eq{
	\hspace{-2mm}c^2\| \cX_L\cT_{i-1} \cR_1\|^2 \hspace{-0.8mm} &\le \hspace{-0.8mm} c^2\mu_{\max}^2 \|\cX_L\|^2 \|\cT_d\|^2 \|\cR_1\|^2\delta^2  \hspace{-0.8mm}\define\hspace{-0.8mm} \sigma^2_{21} \mu_{\max}^2 \label{xcbnb978-1}\\
	\hspace{-5mm}\| \cX_L\cT_{i-1}\cX_R\|^2 \hspace{-0.8mm} &\le\hspace{-0.8mm} \mu_{\max}^2 \hspace{-0.5mm}\|\cX_L\|^2\hspace{-0.5mm} \|\cT_d\|^2\hspace{-0.5mm} \|\hspace{-0.3mm}\cX_R\hspace{-0.5mm}\|^2\hspace{-0.5mm} \delta^2 \hspace{-1mm}\define\hspace{-1mm} \sigma^2_{22} \mu_{\max}^2\label{xcbnb978-2}
}
since $\|\cR_1\|=1$, and where $\sigma_{21}$ and $\sigma_{22}$ are defined as
\eq{
	\sigma_{21}=c\|\cX_L\| \|\cT_d\| \delta,\ \ \sigma_{22}=\|\cX_L\| \|\cT_d\| \|\cX_R\|\delta.
}
With \eqref{xcbnb978-1} and \eqref{xcbnb978-2}, inequality \eqref{yuwe9} becomes 
\eq{\label{yuwe9-2}
	\| \check{\sx}_i\|^2 
	\hspace{-1mm}\leq\hspace{-1mm} \left(\hspace{-1mm}\lambda \hspace{-0.8mm}+\hspace{-0.8mm} \frac{2\sigma_{22}^2 \mu^2_{\max}}{1-\lambda}\hspace{-1mm}\right)\hspace{-1mm} \|\check{\sx}_{i-1}\|^2 \hspace{-1mm}+\hspace{-1mm} \frac{2\sigma_{21}^2 \mu^2_{\max}}{1-\lambda}  \|\bar{\sx}_{i-1}\|^2.
}
Combining \eqref{x_bar_square-2} and \eqref{yuwe9-2}, we arrive at the inequality recursion
\eq{\label{xngyi}
	\hspace{-3mm}
	\ba{c}
	\hspace{-2mm}\|\bar{\sx}_i\|^2\hspace{-2mm} \\
	\hspace{-2mm}\|\check{\sx}_i\|^2 \hspace{-2mm}
	\ea
	\preceq
	\underbrace{\ba{cc}
		\hspace{-2mm}1-\sigma_{11}\mu_{\max}\hspace{-1mm} & \hspace{-1mm} \frac{\sigma_{12}^2}{\sigma_{11}}\mu_{\max}\hspace{-2mm}\\
		\hspace{-2mm}\frac{2\sigma_{21}^2\mu^2_{\max}}{1-\lambda} \hspace{-1mm}&\hspace{-1mm} \lambda + \frac{2\sigma_{22}^2 \mu^2_{\max}}{1-\lambda}\hspace{-2mm}
		\ea}_{\define G}
	\ba{c}
	\hspace{-2mm}\|\bar{\sx}_{i-1}\|^2\hspace{-2mm} \\
	\hspace{-2mm}\|\check{\sx}_{i-1}\|^2 \hspace{-2mm}
	\ea.
}
Now we check the spectral radius of the matrix $G$. Recall the fact that the spectral radius of a matrix is upper bounded by any of its norms. Therefore,
\eq{
	\hspace{-2mm}\rho(G) &\le \|G\|_1 = \max\Big\{1-\sigma_{11}\mu_{\max} + \frac{2\sigma_{21}^2\mu^2_{\max}}{1-\lambda},\nnb
	&\hspace{2.8cm}  \lambda + \frac{\sigma_{12}^2}{\sigma_{11}}\mu_{\max} +  \frac{2\sigma_{22}^2 \mu^2_{\max}}{1-\lambda}\Big\},
} 
where we already know that $\lambda<1$. To guarantee $\rho(G)<1$, it is enough to select the step-size parameter small enough to satisfy
\eq{
	1-\sigma_{11}\mu_{\max} + \frac{2\sigma_{21}^2\mu^2_{\max}}{1-\lambda} &< 1,\label{upper-bound-1} \\
	\lambda + \frac{\sigma_{12}^2}{\sigma_{11}}\mu_{\max} +  \frac{2\sigma_{22}^2 \mu^2_{\max}}{1-\lambda} & <1. \label{upper-bound-2}
}
To get a simpler upper bound, we transform \eqref{upper-bound-2} such that
\eq{
	&\ \lambda + \frac{\sigma_{12}^2}{\sigma_{11}}\mu_{\max} +  \frac{2\sigma_{22}^2 \mu^2_{\max}}{1-\lambda} \nnb
	=&\ \lambda \hspace{-1mm}+\hspace{-1mm} \frac{2\sigma_{12}^2}{\sigma_{11}}\mu_{\max}  \hspace{-1mm}-\hspace{-1mm} \left( \frac{\sigma_{12}^2}{\sigma_{11}}\mu_{\max} \hspace{-1mm}-\hspace{-1mm} \frac{2\sigma_{22}^2 \mu^2_{\max}}{1-\lambda} \right) \hspace{-1mm}\le \lambda \hspace{-1mm}+\hspace{-1mm} \frac{2\sigma_{12}^2}{\sigma_{11}}\mu_{\max}, \label{xng8}
}
where the last inequality holds when 
\eq{\label{zxchsdj8}
	\mu_{\max} \le \frac{\sigma_{12}^2(1-\lambda)}{2\sigma_{11}\sigma_{22}^2}.
} 
If, in addition, we let \eqref{xng8} be less than $1$, which is equivalent to selecting
\eq{\label{zxgh8}
	\mu_{\max}\le \frac{\sigma_{11}(1-\lambda)}{2\sigma_{12}^2},
}
then we guarantee equality \eqref{upper-bound-2}. Combing \eqref{upper-bound-1}, \eqref{zxchsdj8} and \eqref{zxgh8}, we have
\eq{\label{upper-bd}
\hspace{-3mm}\mu_{\max}\le \min\left\{ \frac{\sigma_{11}(1-\lambda)}{2\sigma_{21}^2},\frac{\sigma_{12}^2(1-\lambda)}{2\sigma_{11}\sigma_{22}^2},\frac{\sigma_{11}(1-\lambda)}{2\sigma_{12}^2} \right\}
}
This together with \eqref{xchayu}, i.e.
\eq{\label{upper-bd-2}
\mu_{\max}< {1}/{\delta}
}
will guarantee $\|G\|_1$ to be less than $1$. In fact, the upper bound in \eqref{upper-bd} can be further simplified.
%
%
From the definitions of $\sigma_{11}$, $\sigma_{12}$, $\sigma_{21}$ and $\sigma_{22}$, we have
\eq{
	\frac{\sigma_{11}(1-\lambda)}{2\sigma_{21}^2} = &\
	\frac{p_{k_o}\tau_{k_o}\nu (1-\lambda)}{2 c^2 \|\cX_L\|^2\|\cT_d\|^2\delta^2},\label{term-2}\\
	\frac{\sigma_{12}^2(1-\lambda)}{2\sigma_{11}\sigma_{22}^2} = &\ \frac{p_{\max}(1-\lambda)}{2p_{k_o}\tau_{k_o}\nu \|\cX_L\|^2\|\cT_d\|^2 c^2}, \label{term-4}\\
	\frac{\sigma_{11}(1-\lambda)}{2\sigma_{12}^2} = &\ 
	\frac{p_{k_o}\tau_{k_o}\nu (1-\lambda)c^2}{2p_{\max}\|\cX_R\|^2\delta^2}.\label{term-3}	
}
First, notice that 
\eq{
\frac{\sigma_{11}(1-\lambda)}{2\sigma_{21}^2} \Big/\frac{\sigma_{12}^2(1-\lambda)}{2\sigma_{11}\sigma_{22}^2}= \frac{(p_{k_o}\tau_{k_o}\nu)^2}{p_{\max}\delta^2} < 1
}
because $p_{k_o}<p_{\max}$, $\tau_{k_o}<1$ and $\nu<\delta$. Therefore, the inequality in \eqref{upper-bd} is equivalent to
\eq{
	\mu_{\max} &\le  \min\left\{ \frac{p_{k_o}\tau_{k_o}\nu (1-\lambda)}{2 c^2 \|\cX_L\|^2\|\cT_d\|^2\delta^2}, \frac{p_{k_o}\tau_{k_o}\nu (1-\lambda)c^2}{2p_{\max}\|\cX_R\|^2\delta^2} \right\} \nnb
	&= \frac{p_{k_o}\tau_{k_o}\nu (1-\lambda)}{2\delta^2} \min\left\{ \frac{1}{\|\cX_L\|^2 \|\cT_d\|^2 c^2},\frac{c^2}{p_{\max}\|\cX_R\|^2}\right\}. \label{cxn2778}
}
It is observed that the constant value $c$ affects the upper bound in \eqref{cxn2778}. If $c$ is sufficiently large, then the first term in \eqref{cxn2778} dominates and $\mu_{\max}$ has a narrow feasible set. On the other hand, if $c$ is sufficiently small, then the second term dominates and $\mu_{\max}$ will also have a narrow feasible set. To make the feasible set of $\mu_{\max}$ as large as possible, we should optimize $c$ to maximize 
\eq{
\min\Big\{ \frac{1}{\|\cX_L\|^2 \|\cT_d\|^2 c^2},\ \ \frac{c^2}{p_{\max}\|\cX_R\|^2}\Big\}.
}
Notice that the first term $1/(\|\cX_L\|^2 \|\cT_d\|^2 c^2)$ is monotone decreasing with $c^2$, while the second term $c^2/\|\cX_R\|^2$ is monotone increasing with $c^2$. Therefore, when 
\eq{
	\frac{1}{\|\cX_L\|^2 \|\cT_d\|^2 c^2} = \frac{c^2}{p_{\max}\|\cX_R\|^2} \Longleftrightarrow c^2 = \frac{\sqrt{p_{\max}}\|\cX_R\|}{\|\cX_L\| \|\cT_d\|},\label{c2}
}
we get the maximum upper bound for $\mu_{\max}$, i.e.
\eq{
	\mu_{\max}\le \frac{p_{k_o}\tau_{k_o}\nu (1-\lambda)}{2\sqrt{p_{\max}}\|\cX_L\| \|\cT_d\| \|\cX_R\|\delta^2}.
}
Next we compare the above upper bound with $1/\delta$. Recall that for any matrix $A$, its spectral radius is smaller than its $2-$induced norm so that
\eq{\label{238ss}
	\|\cT_d\| \ge \rho(\cT_d)\overset{\eqref{xha9867}}{=} \rho(\tcA)=1.
}
Moreover, recall from Lemma \ref{lm-B-decomposition} that $X_L X_R=I_{2(N-1)}$, so that
$\cX_L \cX_R=X_L X_R \otimes I_{M} = I_{2M(N-1)}$, which implies that 
\eq{\label{2378n}
	\|\cX_L\|\|\cX_R\| \ge \|\cX_L \cX_R\|= 1.
}
Using relations \eqref{238ss} and \eqref{2378n}, and recalling that $p_{k_o}\le p_{\max} < \sqrt{p_{\max}}$, $\tau_{k_o}<1$, $1-\lambda<1$ and $\nu<\delta$, we have
\eq{
	\frac{p_{k_o}\tau_{k_o}\nu (1-\lambda)}{2\sqrt{p_{\max}}\|\cX_L\| \|\cT_d\| \|\cX_R\|\delta^2}\le \frac{\nu }{\delta^2} < \frac{\delta}{\delta^2}=\frac{1}{\delta}.
}
Therefore, the upper bounds in \eqref{upper-bd}, \eqref{upper-bd-2} are determined by
\eq{\label{final-upper-bd}
	\mu_{\max}\le \frac{p_{k_o}\tau_{k_o}\nu (1-\lambda)}{2\sqrt{p_{\max}}\|\cX_L\| \|\cT_d\| \|\cX_R\|\delta^2}.
}
In other words, when $\mu_{\max}$ satisfies \eqref{final-upper-bd}, $\|G\|_1$ will be guaranteed to be less than $1$, i.e.,
\eq{\label{xsdn8}
	\hspace{-2mm} \|G\|_1&= \max\Big\{1-\sigma_{11}\mu_{\max} + \frac{2\sigma_{21}^2\mu^2_{\max}}{1-\lambda},\nnb
	&\hspace{2.8cm}  \lambda + \frac{\sigma_{12}^2}{\sigma_{11}}\mu_{\max} +  \frac{2\sigma_{22}^2 \mu^2_{\max}}{1-\lambda}\Big\} \nnb
	&= \max\Big\{ 1-p_{k_o}\tau_{k_o}\nu \mu_{\max} + \frac{2c^2\|\cX_L\|^2\|\cT_d\|^2\delta^2 \mu^2_{\max}}{1-\lambda}\nnb &\hspace{3mm}\lambda\hspace{-1mm}+\hspace{-1mm}\frac{p_{\max}\|\cX_R\|^2\delta^2}{c^2p_{k_o}\tau_{k_o}\nu}\mu_{\max} \hspace{-1mm}+\hspace{-1mm} \frac{2\|\cX_L\|^2 \|\cT_d\|^2 \|\cX_R\|^2\delta^2\mu_{\max}^2}{1-\lambda} \Big\} \nnb
	&\overset{\eqref{c2}}{=} \max\Big\{ 1-p_{k_o}\tau_{k_o}\nu \mu_{\max} + \frac{2\sqrt{p_{\max}} \alpha_d \delta^2 \mu^2_{\max}}{1-\lambda},\nnb
	&\hspace{1cm} \lambda \hspace{-1mm}+\hspace{-1mm} \frac{\sqrt{p_{\max}} \alpha_d \delta^2\mu_{\max}}{p_{k_o}\tau_{k_o}\nu} \hspace{-1mm}+\hspace{-1mm} \frac{2\alpha_d^2\delta^2\mu_{\max}^2}{1-\lambda} \Big\} < 1,
} 
where $\alpha_d \define \|\cX_L\|\|\cT_d\|\|\cX_R\|$. Let
\eq{
z_i\define \ba{c}
\|\bar{\sx}_i\|^2\\
\|\check{\sx}_i\|^2
\ea \succeq 0,
}
and note from \eqref{xngyi} that
\eq{
z_i \preceq G z_{i-1}.
}
Computing the $1$-norm of both sides gives
\eq{
\|\bar{\sx}_i\|^2 \hspace{-1mm}+\hspace{-1mm} \|\check{\sx}_i\|^2 &\hspace{-1mm}=\hspace{-1mm} \|z_i\|_1 \le \|G\|_1 \|z_{i-1}\|_1 = \rho (\|\bar{\sx}_{i-1}\|^2 \hspace{-1mm}+\hspace{-1mm} \|\check{\sx}_{i-1}\|^2),\nnb
&\hspace{-1mm}\le\hspace{-1mm} \rho^i(\|\bar{\sx}_{0}\|^2 \hspace{-1mm}+\hspace{-1mm} \|\check{\sx}_{0}\|^2), \label{xw3n}
}
where we define $\rho \define \|G\|_1$. Inequality \eqref{xw3n} is equivalent to
\eq{\label{xngyi-2}
	\hspace{-3mm}
	\left\|
	\ba{c}
	\hspace{-2mm}\bar{\sx}_i\hspace{-2mm} \\
	\hspace{-2mm}\check{\sx}_i \hspace{-2mm}
	\ea
	\right\|^2
	\le \rho^{i}
	\left\|
	\ba{c}
	\hspace{-2mm}\bar{\sx}_{0}\hspace{-2mm} \\
	\hspace{-2mm}\check{\sx}_{0} \hspace{-2mm}
	\ea
	\right\|^2,
}
By re-incorporating $\widehat{\sx}_i=0$, relation \eqref{xngyi-2} also implies that
\eq{\label{xngyi-3}
	\hspace{-3mm}
	\left\|
	\ba{c}
	\hspace{-2mm}\bar{\sx}_i\hspace{-2mm} \\
	\hspace{-2mm}\widehat{\sx}_i\hspace{-2mm} \\
	\hspace{-2mm}\check{\sx}_i \hspace{-2mm}
	\ea
	\right\|^2
	\le \rho^{i}
	\left\|
	\ba{c}
	\hspace{-2mm}\bar{\sx}_{0}\hspace{-2mm} \\
	\hspace{-2mm}\widehat{\sx}_0\hspace{-2mm} \\
	\hspace{-2mm}\check{\sx}_{0} \hspace{-2mm}
	\ea
	\right\|^2 \define C_0 \rho^i.
}
From \eqref{x-bar and x-check} we conclude that 
\eq{\label{237sdnkk}
	\left\|
	\ba{c}
	\hspace{-2mm}\twd_i\hspace{-2mm}\\
	\hspace{-2mm}\tyd_i\hspace{-2mm}
	\ea 
	\right\|^2
	\le \left\|\cX \right\|^2
	\left\|
	\ba{c}
	\hspace{-2mm}\bar{\sx}_i\hspace{-2mm}\\
	\hspace{-2mm}\widehat{\sx}_i\hspace{-2mm} \\
	\hspace{-2mm}\check{\sx}_i\hspace{-2mm}
	\ea
	\right\|^2 \le C \rho^i,
}
where the constant $C = \|\cX\|^2C_0$.
%

{\color{black}
\section{Proof of Theorem \ref{them-algorithm-prime}}
\label{app-algorithm-prime}
We define 
\eq{
\cM_i^\prime \define \mu_o \diag\{  q_1 I_M/z_{1,i}(1), \cdots, q_N I_M/z_{N, i}(N) \}.
}
Substituting recursions (98) and (99) from Part I \cite{yuan2017exact1} into expre-ssion (100) we obtain (compare with (93) from Part I \cite{yuan2017exact1}): 
\eq{
\sw_i \hspace{-0.8mm}=\hspace{-0.8mm} \tcA\tran \hspace{-1mm} \left[ 2\sw_{i\hspace{-0.3mm}-\hspace{-0.3mm}1} \hspace{-0.5mm}-\hspace{-0.5mm} \sw_{i\hspace{-0.3mm}-\hspace{-0.3mm}2} \hspace{-0.5mm}-\hspace{-0.5mm}\left( \hspace{-0.5mm}\cM_i^\prime \grad \cJ^o(\sw_{i\hspace{-0.3mm}-\hspace{-0.3mm}1}) \hspace{-0.5mm}-\hspace{-0.5mm} \cM_{i\hspace{-0.3mm}-\hspace{-0.3mm}1}^\prime \grad \cJ^o(\sw_{i\hspace{-0.3mm}-\hspace{-0.3mm}2}) \right) \right],
}
which can be rewritten into a primal-dual form (compare with (89) from Part I \cite{yuan2017exact1}):
\begin{equation}
\left\{
\begin{aligned}
\sw_i &= \tcA\tran \Big(\sw_{i\hspace{-0.3mm}-\hspace{-0.3mm}1}\hspace{-0.8mm}-\hspace{-0.8mm}\cM_i^\prime \grad \cJ^o(\sw_{i\hspace{-0.3mm}-\hspace{-0.3mm}1})\Big)\hspace{-0.8mm}-\hspace{-0.8mm}\cP^{-1}\cV \sy_{i-1}, \label{zn-1'}\\
\sy_i &= \sy_{i-1} + \cV \sw_i.
\end{aligned}
\right.
\end{equation}
For the initialization, we set $y_{-1}=0$ and $\sw_{-1}$ to be any value, and hence for $i=0$ we have
\begin{equation}
\left\{
\begin{aligned}
\sw_0 &= \tcA\tran \Big(\sw_{\hspace{-0.3mm}-\hspace{-0.3mm}1}\hspace{-0.8mm}-\hspace{-0.8mm}\cM_0^\prime \grad \cJ^o(\sw_{\hspace{-0.3mm}-\hspace{-0.3mm}1})\Big), \label{zn-0'}\\
\sy_0 &= \cV \sw_0. 
\end{aligned}
\right.
\end{equation}
Recursions \eqref{zn-1'} and \eqref{zn-0'} are very close to the standard exact diffusion recursions \eqref{zn-1} and \eqref{zn-0}, except that the step-size matrix $\cM_i^\prime$ is now changing with iteration $i$. Following the arguments \eqref{error} -- \eqref{ed-1-subtracrt}, we have
\begin{equation}
\left\{
\begin{aligned}
\tcA\tran \twd_i &\hspace{-0.5mm}=\hspace{-0.5mm} \tcA\tran \Big(\twd_{i-1}\hspace{-0.8mm}+\hspace{-0.8mm}\cM_i^\prime \grad \cJ^o(\sw_{i\hspace{-0.3mm}-\hspace{-0.3mm}1})\Big) \hspace{-0.8mm}+\hspace{-0.8mm} \cP^{-1} \cV  \sy_{i}, \label{ed-1-subtracrt-prime} \\
\tyd_i &= \tyd  _{i-1} - \cV \sw_i. 
\end{aligned}
\right.
\end{equation}
Subtracting optimality conditions \eqref{KKT-1-1}--\eqref{KKT-2-2} from \eqref{ed-1-subtracrt-prime} leads to
\begin{equation}
\hspace{-0.8mm}\left\{
\begin{aligned}
\hspace{-1.5mm}\tcA\tran \twd_i &\hspace{-0.5mm}=\hspace{-0.5mm} \tcA\tran \Big(\hspace{-0.8mm}\twd_{i-1}\hspace{-0.8mm}+\hspace{-0.8mm}\cM \big[\grad \cJ^o(\sw_{i\hspace{-0.3mm}-\hspace{-0.3mm}1})\hspace{-0.8mm}-\hspace{-0.8mm}\grad \cJ^o(\sw^\star)\big]\hspace{-0.8mm}\Big) \hspace{-0.8mm}-\hspace{-0.8mm} \cP^{-1} \cV  \tyd_{i} \\
& \hspace{5mm} + \tcA\tran(\cM_i^\prime - \cM) \grad \cJ^o(\sw_{i-1}), \label{ed-1-subtracrt-1-prime}  \\
\hspace{-1.5mm}\tyd_i &= \tyd_{i-1} +\cV \twd_i. 
\end{aligned}
\right.
\end{equation}
Comparing recursions \eqref{ed-1-subtracrt-1-prime} and \eqref{ed-1-subtracrt-1}, it is observed that recursion \eqref{ed-1-subtracrt-1-prime} has an extra ``mismatch" term, 
$
\tcA\tran(\cM_i^\prime - \cM) \grad \cJ^o(\sw_{i-1}).
$
This mismatch arises because we do not know the perron vector $p$ in advance. We need to run the power iteration (see recursion (97) from Part I\cite{yuan2017exact1}) to learn it. Intuitively, since $\cM_i^\prime \rightarrow \cM$ as $i\to \infty$, we can expect the mismatch term to vanish gradually. Let 
\eq{\label{ei}
e_i \define (\cM_i^\prime - \cM) \grad \cJ^o(\sw_{i-1}).
}
By following arguments \eqref{237sdb}--\eqref{xcnh09}, recursion \eqref{ed-1-subtracrt-1-prime} is equivalent to
\eq{\label{xcnh09-prime}
&\ \ba{cc}
\tcA\tran & \cP^{-1}\cV \\
-\cV & I_{MN}
\ea
\ba{c}
\twd_i\\
\tyd_i
\ea \nnb
=&\ 
\ba{cc}
\hspace{-1mm}\tcA\tran (I_{MN}-\cM\cH_{i-1}) & 0 \hspace{-1mm}\\
\hspace{-1mm}0 & I_{MN} \hspace{-1mm}
\ea
\hspace{-1mm}
\ba{c}
\hspace{-1.5mm}\twd_{i-1} \hspace{-2mm}\\
\hspace{-1.5mm}\tyd_{i-1}\hspace{-2mm}
\ea + 
\ba{c}
\hspace{-1.5mm}\tcA\tran\hspace{-1.5mm}\\
\hspace{-1.5mm}0\hspace{-1.5mm}
\ea e_i .
}
By following \eqref{nxcm987}--\eqref{T-defi}, recursion \eqref{xcnh09-prime} can be rewritten as
\eq{
\boxed{	
\ba{c}
\hspace{-1mm}\twd_i\hspace{-1mm}\\
\hspace{-1mm}\tyd_i\hspace{-1mm}
\ea = (\cB - \cT_{i-1})\ba{c}
\hspace{-1mm}\twd_{i-1}\hspace{-1mm}\\
\hspace{-1mm}\tyd_{i-1}\hspace{-1mm}
\ea
+ \cB_{\ell} e_i
} \label{error-recursion-prime}
}
where $\cB$ and $\cT_{i}$ are defined in \eqref{T-defi}, and
\eq{
\cB_{\ell} = 
\ba{c}
\tcA\tran\\
\cV\tcA\tran
\ea
}
Relation \eqref{error-recursion-prime} is the error dynamics for the exact diffusion algorithm $1^\prime$. Comparing \eqref{error-recursion-prime} with \eqref{error-recursion}, we find that algorithm $1^\prime$ is essentially the standard exact diffusion with error perturbation. Using Lemma \eqref{lm-B-decomposition} and by following arguments from \eqref{tsdhb} to \eqref{final-recursion}, we can transform the error dynamics \eqref{error-recursion-prime} into 
\eq{\label{final-recursion-prime}
\hspace{-3mm}
\ba{c}
\hspace{-2mm}\bar{\sx}_i\hspace{-2mm}\\
\hspace{-2mm}\check{\sx}_i\hspace{-2mm}
\ea \hspace{-0.5mm}= &\hspace{-0.5mm}
\ba{cc}
\hspace{-2mm}I_M \hspace{-1mm}-\hspace{-1mm}{\overline{\cP}\tran\cM\cH_{i\hspace{-0.4mm}-\hspace{-0.4mm}1}\cI} & -\frac{1}{c}\overline{\cP}\tran\cM\cH_{i\hspace{-0.4mm}-\hspace{-0.4mm}1}\cX_{R,u}\hspace{-2mm}\\
\hspace{-2mm}- c\cX_L \cT_{i-1} \cR_1 & \cD_1 - \cX_L \cT_{i-1} \cX_R \hspace{-2mm}
\ea \hspace{-1.5mm}
\ba{c}
\hspace{-2mm}\bar{\sx}_{i\hspace{-0.4mm}-\hspace{-0.4mm}1}\hspace{-2mm}\\
\hspace{-2mm}\check{\sx}_{i\hspace{-0.4mm}-\hspace{-0.4mm}1}\hspace{-2mm}
\ea \nnb
& + 
\ba{c}
\hspace{-2mm}\overline{\cP}\tran\hspace{-2mm}\\
\hspace{-2mm}c\cX_L \cB_{\ell}\hspace{-2mm}
\ea
\hspace{-1mm}e_i.
}
Next we analyze the convergence of the above recursion. From the first line we have
\eq{
\hspace{-1mm} \|\bar{\sx}_i\|^2 \hspace{-1mm}=\hspace{-1mm}&\ \left\| \left(I_M \hspace{-1mm} - \hspace{-1mm} {\overline{\cP}\tran \hspace{-0.8mm}\cM\cH_{i\hspace{-0.3mm}-\hspace{-0.3mm}1}\cI}\right) \bar{\sx}_{i-1}  \right. \nnb
& \left. \hspace{1cm} - \frac{1}{c}
\overline{\cP}\tran\cM\cH_{i-1}\cX_{R,u} \check{\sx}_{i-1}\hspace{-0.5mm} + \overline{\cP}\tran e_i \right\|^2 }
\eq{\label{x_bar_square-prime}
\le &\ \frac{1}{1-t}\left\|I_M \hspace{-1mm} - \hspace{-1mm} {\overline{\cP}\tran \hspace{-0.8mm}\cM\cH_{i\hspace{-0.3mm}-\hspace{-0.3mm}1}\cI}\right\|^2\|\bar{\sx}_{i-1}\|^2 \nnb
&\quad + \frac{2}{t}\frac{1}{c^2}\|\overline{\cP}\tran\cM\cH_{i-1}\cX_{R,u}\|^2 \|\check{\sx}_{i-1}\|^2 + \frac{2}{t}\|\overline{\cP}\tran\|^2\|e_i\|^2 \nnb
\le & (1\hspace{-1mm}-\hspace{-1mm}\sigma_{11}\mu_{\max})\|\hspace{-0.3mm}\bar{\sx}_{i-1}\hspace{-0.3mm}\|^2 \hspace{-1mm}+\hspace{-1mm} \frac{\sigma_{12}^2\mu_{\max}}{\sigma_{11}}\|\hspace{-0.3mm}\check{\sx}_{i-1}\hspace{-0.3mm}\|^2 + \frac{2\|e_i\|^2}{\sigma_{11}\mu_{\max}},
}
where the last inequality follows the arguments in \eqref{x_bar-recursion}--\eqref{x_bar_square-2}. From the second line of recursion \eqref{final-recursion-prime}, we have
\eq{\hspace{-2mm} \label{yuwe9-2-prime}
&\|\check{\sx}_i\|^2 \nnb
=& \footnotesize\|\cD_1 \check{\sx}_{i-1} \hspace{-1mm}-\hspace{-1mm} \cX_L\cT_{i-1}(c\cR_1 \bar{\sx}_{i-1}
+ \cX_R\check{\sx}_{i-1}) + c \cX_L \cB_{\ell}e_i\|^2\nn\\
\le & \frac{\|\cD_1\|^2}{t} \|\check{\sx}_{i-1}\|^{\hspace{-0.3mm}2} \hspace{-1mm}+\hspace{-1mm} \frac{2c^2}{1-t} \| \cX_L\cT_{i-1}\cR_1\|^2 \|\bar{\sx}_{i-1}\|^2\nn\\
&+ \frac{2}{1-t} \|\cX_L\cT_{i-1}\cX_R\|^2 \| \check{\sx}_{i-1}\|^2 + \frac{2c^2}{1-t}\|\cX_L \cB_{\ell}\|^2 \|e_i\|^2 \nnb
\le & \left(\hspace{-1mm}\lambda \hspace{-0.8mm}+\hspace{-0.8mm} \frac{2\sigma_{22}^2 \mu^2_{\max}}{1-\lambda}\hspace{-1mm}\right)\hspace{-1mm} \|\check{\sx}_{i-1}\|^2 \hspace{-1mm}+\hspace{-1mm} \frac{2\sigma_{21}^2 \mu^2_{\max}}{1-\lambda}  \|\bar{\sx}_{i-1}\|^2 \hspace{-1mm}+\hspace{-1mm} \frac{2c^2 d \|e_i\|^2}{1-\lambda},
}
where $d\define \|\cX_L \cB_{\ell}\|^2$ is independent of iteration $i$. Moreover, the last inequality holds because of arguments in \eqref{yujs99}--\eqref{yuwe9-2}. Combining \eqref{x_bar_square-prime} and \eqref{yuwe9-2-prime}, we arrive at the inequality recursion (compare with \eqref{xngyi}):
\eq{\label{xngyi-prime}
\hspace{-3mm}
\ba{c}
\hspace{-2mm}\|\bar{\sx}_i\|^2\hspace{-2mm} \\
\hspace{-2mm}\|\check{\sx}_i\|^2 \hspace{-2mm}
\ea
&\preceq
\underbrace{\ba{cc}
\hspace{-2mm}1-\sigma_{11}\mu_{\max}\hspace{-1mm} & \hspace{-1mm} \frac{\sigma_{12}^2}{\sigma_{11}}\mu_{\max}\hspace{-2mm}\\
\hspace{-2mm}\frac{2\sigma_{21}^2\mu^2_{\max}}{1-\lambda} \hspace{-1mm}&\hspace{-1mm} \lambda + \frac{2\sigma_{22}^2 \mu^2_{\max}}{1-\lambda}\hspace{-2mm}
\ea}_{\define G}
\ba{c}
\hspace{-2mm}\|\bar{\sx}_{i-1}\|^2\hspace{-2mm} \\
\hspace{-2mm}\|\check{\sx}_{i-1}\|^2 \hspace{-2mm}
\ea \nnb
&+ 
\ba{c}
\frac{2}{\sigma_{11} \mu_{\max}} \\
\frac{2c^2d}{1-\lambda}
\ea
\|e_i\|^2.
}
Next let us bound the mismatch term $\|e_i\|^2$. From \eqref{ei} we have
\eq{
e_i &= (\cM_i^\prime - \cM) \left( \grad \cJ^o(\sw_{i-1}) - \grad \cJ^o(\sw^\star) \right) \nnb
& \hspace{1cm} + (\cM_i^\prime - \cM)\grad \cJ^o(\sw^\star) \nnb
&\overset{\eqref{xcnh}}{=} -(\cM_i^\prime - \cM) \cH_{i-1} \twd_{i-1} + (\cM_i^\prime - \cM)\grad \cJ^o(\sw^\star).
}
which implies that
\eq{\label{ei-upper-bound}
\|e_i\|^2 & \le 2\delta^2 \|\cM_i^\prime - \cM\|^2 \|\twd_{i-1}\|^2 + 2\|\cM_i^\prime - \cM\|^2 g,
}
where $g\define \|\cJ^o(\sw^\star)\|^2$ is a constant independent of iteration. Recall that $\cM = M \otimes I_M$ and $\cM^\prime_i = M^\prime_i \otimes I_M$ where
\eq{
M &= \diag\{\mu_1, \mu_2, \cdots, \mu_N\}, \nnb
M^\prime_i &= \diag\left\{\frac{q_1\mu_o}{z_{1,i}(1)}, \cdots, \frac{q_N\mu_o}{z_{N,i}(N)}\right\}.
}
Using the relation $\mu_k = q_k \mu_o/p_k$ (see equation (13) from Part I\cite{yuan2017exact1}), we have
\eq{\label{zxncnc}
& M - M^\prime_i \nnb
= &\diag\left\{\frac{q_1\mu_o}{p_1}\left(1 - \frac{p_1}{z_{1,i}(1)}\right),\cdots, \frac{q_N\mu_o}{p_N}\left(1 - \frac{p_N}{z_{N,i}(N)}\right)\right\}\nnb
= &\diag\left\{\mu_1\left(1 - \frac{p_1}{z_{1,i}(1)}\right),\cdots, \mu_N \left(1 - \frac{p_N}{z_{N,i}(N)}\right)\right\} \nnb
= &\mu_{\max}\diag\left\{\tau_1\left(1 - \frac{p_1}{z_{1,i}(1)}\right),\cdots, \tau_N \left(1 - \frac{p_N}{z_{N,i}(N)}\right)\right\},
}
where $\tau_k = \mu_k/\mu_{\max}\le1$. 

Now we examine the convergence of $1-p_k/z_{k,i}(k)$. From the discussion in Policy 5 form Part I\cite{yuan2017exact1}, it is known that $\sz_i$ generated from the power iteration (see equation (37) from Part I) will converge to $[(\mathds{1}_N \otimes I_N)(p\tran \otimes I_N)]\sz_{-1}$. Therefore, 
\eq{
&\hspace{-1cm} \sz_i - [(\mathds{1}_N \otimes I_N)(p\tran \otimes I_N)]\sz_{-1} \nnb
= & \left[\left(\cA\tran\right)^{i+1} - (\mathds{1}_N \otimes I_N)(p\tran \otimes I_N)\right] \sz_{-1} \nnb
= & \left\{\left[ \left(A\tran\right)^{i+1} - \mathds{1}_N p\tran \right] \otimes I_N \right\} z_{-1} \nnb
= & \left\{ \left[ A\tran - \mathds{1}_N p\tran \right]^{i+1} \otimes I_N \right\} z_{-1}.
}
Recall from the discussion in Policy 5 from Part I\cite{yuan2017exact1} that
\eq{
[(\mathds{1}_N \otimes I_N)(p\tran \otimes I_N)]\sz_{-1} = \col\{p,\cdots,p\} \in \RR^{N^2}.
}
As a result,
\eq{\label{sdj}
|z_{k,i}(k) - p_k|^2 &\le \|\sz_i - [(\mathds{1}_N \otimes I_N)(p\tran \otimes I_N)]\sz_{-1} \|^2 \nnb
& \le \|A\tran - \mathds{1}_N p\tran\|^{2(i+1)} \|\sz_{-1}\|^2 \nnb
& = h\cdot \rho_{A}^{2(i+1)}, \quad \forall k = 1, \cdots, N.
}
where $h\define \|\sz_{-1}\|^2$ is a constant, and $\rho_A$ is the second largest eigenvalue magnitude of matrix $A$, i.e., $\rho_A = \max\{|\lambda_2(A)|, |\lambda_N(A)|\}$. Since $A$ is locally balanced, we know $A$ is diagonalizable with real eigenvalue in $(-1, 1]$, and it has a  single eigenvalue at $1$ (see Table I from Part I \cite{yuan2017exact1}), we conclude that $\rho_A < 1$. Also, recall from the discussion at the end of Policy 5 in Part I \cite{yuan2017exact1} that $z_{k,i}(k)>0$ is guaranteed when $\bar{a}_{kk}>0$. Let 
\eq{\label{23nsdb}
\alpha_k \define \min_i\{z_{k,i}(k) \} > 0,\quad \forall\ k = 1, \cdots, N
}
Combining \eqref{sdj} and \eqref{23nsdb}, it holds that for $k = 1, \cdots, N$,
\eq{\label{2bdn}
\left( 1 - \frac{p_k}{z_{k,i}(k)} \right)^2 \le \frac{h}{\alpha_k}\rho_A^{2(i+1)}= h_k \rho_A^{2(i+1)},
}
where we define $h_k \define h/\alpha_k$. Substituting \eqref{2bdn} into \eqref{zxncnc}, it holds that
\eq{\label{xnsd8}
	\|\cM_i^\prime - \cM\|^2 = \|M_i^\prime - M\|^2 \le \mu_{\max}^2 h^\prime  \rho_{A}^{2(i+1)},
	}
	where $h^\prime \define \max_k\{\tau^2_k h_k\}$ is a constant independent of iterations. Substituting \eqref{xnsd8} into \eqref{ei-upper-bound}, we have
\eq{\label{ei-upper-bound-2}
	\|e_i\|^2 & \le 2 \delta^2 \|\cM_i^\prime - \cM\|^2 \|\twd_{i-1}\|^2 + 2\|\cM_i^\prime - \cM\|^2 g \nnb
	& \le 2 \delta^2 \mu_{\max}^2 h^\prime  \rho_{A}^{2(i+1)} \|\twd_{i-1}\|^2 + 2 \mu_{\max}^2 h^\prime g \rho_{A}^{2(i+1)} \nnb
	&\le 2 \delta^2 \mu_{\max}^2 h^\prime  \rho_{A}^{2(i+1)} \hspace{-1mm} \left(  \|\twd_{i-1}\|^2  +  \|\tyd_{i-1}\|^2 \right) \nnb
	&\quad \quad  + 2 \mu_{\max}^2 h^\prime g \rho_{A}^{2(i+1)}.
	}
	Recall from \eqref{x-bar and x-check} that 
	\eq{\label{x-bar and x-check-2}
\ba{c}
\hspace{-1mm}\twd_i\hspace{-1mm}\\
\hspace{-1mm}\tyd_i\hspace{-1mm}
\ea
=
\cX'
\ba{c}
\hspace{-1mm}\bar{\sx}_i\hspace{-1mm}\\
\hspace{-1mm}\widehat{\sx}_i\hspace{-1mm}\\
\hspace{-1mm}\check{\sx}_i \hspace{-1mm}
\ea.
}
We therefore have
\eq{\label{23bnd}
	 \|\twd_{i}\|^2 +  \|\tyd_{i}\|^2 \le &\ \|\cX'\|^2 \left( \|\bar{\sx}_i\|^2 + \|\widehat{\sx}_i\|^2 + \|\check{\sx}_i\|^2  \right) \nnb
	= &\  \|\cX'\|^2 \left( \|\bar{\sx}_i\|^2  + \|\check{\sx}_i\|^2  \right),
	}
	where the last equality holds because $\widehat{\sx}_i = 0$ for $i=0,1,\cdots$ (see \eqref{xzcn}). Substituting \eqref{23bnd} into \eqref{ei-upper-bound-2}, we have
	\eq{\label{2ndnd}
	\|e_i\|^2 &\le 2 \delta^2 \mu_{\max}^2 h^\prime \|\cX'\|^2 \rho_{A}^{2(i+1)}\hspace{-1mm} \left( \|\bar{\sx}_{i-1}\|^2  \hspace{-1mm}+\hspace{-1mm} \|\check{\sx}_{i-1}\|^2 \right) \nnb
	&\quad + 2 \mu_{\max}^2 h^\prime g \rho_{A}^{2(i+1)}
	}																							Substituting \eqref{2ndnd} into \eqref{xngyi-prime}, we have
\eq{\label{xngyi-prime-2}
	\hspace{-3mm}
	\ba{c}
	\hspace{-2mm}\|\bar{\sx}_i\|^2\hspace{-2mm} \\
	\hspace{-2mm}\|\check{\sx}_i\|^2 \hspace{-2mm}
	\ea
	& \hspace{-1mm}\preceq \hspace{-1mm}
	{\ba{cc}
		\hspace{-2mm}1 \hspace{-0.6mm}- \hspace{-0.6mm} \sigma_{11}\mu_{\max} \hspace{-0.6mm}+\hspace{-0.6mm} b^\prime \mu_{\max}\rho_{A}^{2(i+1)}\hspace{-1mm} & \hspace{-1mm} \frac{\sigma_{12}^2}{\sigma_{11}}\mu_{\max} \hspace{-0.6mm}+\hspace{-0.6mm} b^\prime \mu_{\max}\rho_{A}^{2(i+1)}\hspace{-2mm}\\
		\hspace{-2mm}\frac{2\sigma_{21}^2\mu^2_{\max}}{1-\lambda}+c^\prime \mu^2_{\max}\rho_{A}^{2(i+1)} \hspace{-1mm}&\hspace{-1mm} \lambda \hspace{-0.6mm}+\hspace{-0.6mm} \frac{2\sigma_{22}^2 \mu^2_{\max}}{1-\lambda} \hspace{-0.6mm}+\hspace{-0.6mm} c^\prime \mu^2_{\max}\rho_{A}^{2(i+1)}\hspace{-2mm}
		\ea} \nnb
		& \quad \cdot
		\ba{c}
		\hspace{-2mm}\|\bar{\sx}_{i-1}\|^2\hspace{-2mm} \\
		\hspace{-2mm}\|\check{\sx}_{i-1}\|^2 \hspace{-2mm}
		\ea + 
		\ba{c}
		d^\prime \mu_{\max}\rho_{A}^{2(i+1)} \\
		e^\prime \mu^2_{\max}\rho_{A}^{2(i+1)}
		\ea,
		}
	where $b',c',d',e'$ are constants defined as
	\eq{
	b' &\hspace{-1mm}\define\hspace{-1mm} 4 \delta^2 h^\prime \|\cX'\|^2/\sigma_{11}, \quad c' \hspace{-1mm}\define\hspace{-1mm} 4 \delta^2 h^\prime \|\cX'\|^2 c^2d/(1 \hspace{-1mm}-\hspace{-1mm} \lambda), \\
	d' &\hspace{-1mm}\define\hspace{-1mm} 4 h' g/\sigma_{11}, \hspace{1.4cm} e' \hspace{-1mm}\define\hspace{-1mm} 4 h' g c^2 d/(1-\lambda).
	}
	These constants are independent of iterations. It can be verified that when iteration $i$ is large enough such that
\eq{
\rho_A^{2(i+1)} \le \min\left\{ \frac{\sigma_{11}}{2b'}, \frac{\sigma_{12}^2}{\sigma_{11} b'}, \frac{\sigma_{21}^2}{(1-\lambda)c^\prime}, \frac{\sigma_{22}^2}{(1-\lambda)c^\prime} \right\},
}
the inequality \eqref{xngyi-prime-2} becomes
\eq{\label{xngyi-prime-3}
\hspace{-3mm}
\ba{c}
\hspace{-2mm}\|\bar{\sx}_i\|^2\hspace{-2mm} \\
\hspace{-2mm}\|\check{\sx}_i\|^2 \hspace{-2mm}
\ea
&\preceq
\underbrace{\ba{cc}
\hspace{-2mm}1 \hspace{-0.6mm}- \hspace{-0.6mm} \frac{\sigma_{11}\mu_{\max}}{2}\hspace{-1mm} & \hspace{-1mm} \frac{2\sigma_{12}^2}{\sigma_{11}}\mu_{\max} \hspace{-2mm}\\
\hspace{-2mm}\frac{3\sigma_{21}^2\mu^2_{\max}}{1-\lambda} \hspace{-1mm}&\hspace{-1mm} \lambda \hspace{-0.6mm}+\hspace{-0.6mm} \frac{3\sigma_{22}^2 \mu^2_{\max}}{1-\lambda} \hspace{-2mm}
\ea}_{G^\prime} 
\ba{c}
\hspace{-2mm}\|\bar{\sx}_{i-1}\|^2\hspace{-2mm} \\
\hspace{-2mm}\|\check{\sx}_{i-1}\|^2 \hspace{-2mm}
\ea \nnb
& \quad \quad + 
{\ba{c}
d^\prime \mu_{\max} \\
e^\prime \mu^2_{\max}
\ea} \rho_{A}^{2(i+1)},
}
where we can prove $\rho \define \|G'\|_1 = 1 - O(\mu_{\max}) < 1$ by following arguments \eqref{xsdn8}. Inequality \eqref{xngyi-prime-3} further implies that 
\eq{\label{23ndn}
\left( \|\bar{\sx}_i\|^2 \hspace{-1mm}+\hspace{-1mm} \|\check{\sx}_i\|^2 \right) \le \rho \left( \|\bar{\sx}_{i-1}\|^2 \hspace{-1mm} + \hspace{-1mm} \|\check{\sx}_{i-1}\|^2 \right) + f' \rho_{A}^{2(i+1)} }
where $f'\define d'\mu_{\max} + e'\mu_{\max}^2 > 0$. Let $\beta = \max\{\rho, \rho_A\} < 1$. Inequality \eqref{23ndn} becomes
\eq{\label{23ndn-2}
	\left( \|\bar{\sx}_i\|^2 \hspace{-1mm}+\hspace{-1mm} \|\check{\sx}_i\|^2 \right) \le \beta \left( \|\bar{\sx}_{i-1}\|^2 \hspace{-1mm} + \hspace{-1mm} \|\check{\sx}_{i-1}\|^2 \right) \hspace{-1mm}+\hspace{-1mm} f' \beta^{2(i+1)}.
}
{\color{black}By adding $\gamma f' \beta^{2i+4}$, where $\gamma$ can be any positive constant to be chosen, to both sides of the above inequality, we get 
\eq{\label{2bs8}
&\hspace{-8mm}\left( \|\bar{\sx}_i\|^2 + \|\check{\sx}_i\|^2 \right)  + \gamma f' \beta^{2i+4}\nnb
\le &\ \beta \left( \|\bar{\sx}_{i-1}\|^2 + \|\check{\sx}_{i-1}\|^2 \right) + f' \beta^{2i+2} + \gamma f' \beta^{2i+4}\nn\\
= &\ \beta \left( \|\bar{\sx}_{i-1}\|^2 + \|\check{\sx}_{i-1}\|^2  + \frac{1+ \gamma \beta^2}{\beta}f' \beta^{2i+2}\right)
}
By setting
\eq{\label{23dns}
\gamma = \frac{1}{\beta - \beta^2} > 0, 
}
it can be verified that
\eq{\label{xnaswbn}
	\gamma = \frac{1+ \gamma \beta^2}{\beta}.
}}
Substituting \eqref{xnaswbn} into \eqref{2bs8}, we have
\eq{\label{2bs8-2}
	&\hspace{-1cm}\left( \|\bar{\sx}_i\|^2 + \|\check{\sx}_i\|^2 \right)  + \gamma f' \beta^{2(i+2)}\nnb
	\le &\ \beta \left( \|\bar{\sx}_{i-1}\|^2 + \|\check{\sx}_{i-1}\|^2  + \gamma f' \beta^{2(i+1)}\right).
}
As a result, the quantity $\left( \|\bar{\sx}_i\|^2 + \|\check{\sx}_i\|^2 \right)  + \gamma f' \beta^{2(i+2)}$ converges to $0$ linearly. Since $f'>0, \gamma>0$ and $\beta>0$, we can conclude that $\|\bar{\sx}_i\|^2 + \|\check{\sx}_i\|^2$, and hence $\|\twd_i\|^2 + \|\tyd_i\|^2$, converges to $0$ linearly.
}


\section{Error Recursion for EXTRA Consensus}
\label{app-sta-EXTRA-error-recursion}
Multiplying the second recursion of \eqref{extra-pd-2} by $\cV$ gives:
%
%
\eq{
	\cV  \sy^e_i = \cV  \sy^e_{i-1} + \frac{\cP-\cP \cA}{2} \sw_i^e. \label{extra-V-tran y}
}
Substituting into the first recursion of \eqref{extra-pd-2} gives
\eq{
	\tcA \sw_i^e &\hspace{-0.5mm}=\hspace{-0.5mm} \tcA \sw_{i\hspace{-0.3mm}-\hspace{-0.3mm}1}^e\hspace{-0.8mm}-\hspace{-0.8mm}\mu \grad \cJ^o(\sw_{i\hspace{-0.3mm}-\hspace{-0.3mm}1}^e)\hspace{-0.8mm}-\hspace{-0.8mm}\cP^{-1}\cV  \sy_{i}^e, \label{extra-sdfu}
}
From \eqref{extra-sdfu} and the second recursion in \eqref{extra-pd-2} we conclude that
\begin{equation}
	\left\{
	\begin{aligned}
		\tcA \twd_i^e &\hspace{-0.5mm}=\hspace{-0.5mm} \tcA\twd_{i-1}^e\hspace{-0.8mm}+\hspace{-0.8mm}\mu \grad \cJ^o(\sw_{i\hspace{-0.3mm}-\hspace{-0.3mm}1}^e) \hspace{-0.8mm}+\hspace{-0.8mm} \cP^{-1} \cV  \sy^e_{i}, \label{extra-1-subtracrt} \\
		\tyd_i^e &= \tyd_{i-1}^e - \cV \sw_i^e. 
	\end{aligned}
	\right.
\end{equation}
Subtracting the optimality condition \eqref{extra-KKT-1-1}--\eqref{extra-KKT-2-2} from \eqref{extra-1-subtracrt} leads to
\begin{equation}
	\left\{
	\begin{aligned}
		\tcA \twd_i^e &\hspace{-0.5mm}=\hspace{-0.5mm} (\tcA \hspace{-0.8mm}-\hspace{-0.8mm}\mu \cH_{i-1})\twd_{i-1}^e\hspace{-0.8mm} \hspace{-0.0mm}-\hspace{-0.3mm} \cP^{-1} \cV  \tyd_{i}^e, \label{extra-1-subtracrt-1} \\
		\tyd_i^e &= \tyd_{i-1}^e +\cV \twd_i^e. 
	\end{aligned}
	\right.
\end{equation}
which is also equivalent to
\eq{\label{extra-xcnh09}
	&\ \ba{cc}
	\hspace{-2mm}\tcA & \cP^{-1}\cV \hspace{-2mm}\\
	\hspace{-2mm}-\cV & I_{MN} \hspace{-2mm}
	\ea\hspace{-1mm}
	\ba{c}
	\hspace{-2mm}\twd_i^e\hspace{-2mm}\\
	\hspace{-2mm}\tyd_i^e\hspace{-2mm}
	\ea \hspace{-1mm}=\hspace{-1mm}
	\ba{cc}
	\hspace{-2mm}\tcA- \mu \cH_{i-1} & 0\hspace{-2mm}\\
	\hspace{-2mm}0 & I_{MN}\hspace{-2mm}
	\ea\hspace{-1mm}
	\ba{c}
	\hspace{-2mm}\twd^e_{i-1}\hspace{-2mm}\\
	\hspace{-2mm}\tyd^e_{i-1}\hspace{-2mm}
	\ea.
}
Using relations $\tcA = \frac{I_{MN} + \cA}{2}$ and $\cV^2=\frac{\cP - \cP\cA}{2}$, it is easy to verify that
\eq{\label{extra-nxcm987}
	\ba{cc}
	\tcA & \cP^{-1}\cV \\
	-\cV & I_{MN}
	\ea^{-1} = 
	\ba{cc}
	I_{MN} &  -\cP^{-1}\cV \\
	\cV & I_{MN}-\cV \cP^{-1} \cV
	\ea.
}
Substituting \eqref{extra-nxcm987} into \eqref{extra-xcnh09} gives \eqref{extra-error-recursion-2}--\eqref{extra-T-defi}.

\section{Error Recursion in Transformed Domain}
\label{app-sta-EXTRA-reduced}
Multiplying both sides of \eqref{extra-error-recursion-2} by $(\cX')^{-1}$:
%
\eq{
	(\cX')^{-1}
	\ba{c}
	\hspace{-1mm}\twd^e_i\hspace{-1mm}\\
	\hspace{-1mm}\tyd^e_i\hspace{-1mm}
	\ea
	=&\; 
	[(\cX')^{-1}(\cB^e-\cT^e_{i-1}) \cX'] (\cX')^{-1}
	\ba{c}
	\hspace{-1mm}\twd^e_{i-1}\hspace{-1mm}\\
	\hspace{-1mm}\tyd^e_{i-1}\hspace{-1mm}
	\ea
}
leads to
\eq{\label{extra-recursion-transform-2}
	\ba{c}
	\hspace{-1mm}\bar{\sx}^e_i\hspace{-1mm}\\
	\hspace{-1mm}\widehat{\sx}^e_i\hspace{-1mm}\\
	\hspace{-1mm}\check{\sx}^e_i \hspace{-1mm}
	\ea
	\hspace{-1mm}=\hspace{-1mm}
	&\; 
	\left( 
	\ba{ccc}
	I_M & 0 & 0\\ 0 & I_M &0 \\0 & 0 & \cD_1
	\ea
	- \cS^e_{i-1}
	\right)
	\ba{c}
	\hspace{-1mm}\bar{\sx}^e_{i-1}\hspace{-1mm}\\
	\hspace{-1mm}\widehat{\sx}^e_{i-1}\hspace{-1mm}\\
	\hspace{-1mm}\check{\sx}^e_{i-1}\hspace{-1mm}
	\ea,	
}
where we defined
\eq{\label{extra-x-bar and x-check}
	\ba{c}
	\hspace{-1mm}\bar{\sx}^e_i\hspace{-1mm}\\
	\hspace{-1mm}\widehat{\sx}^e_i\hspace{-1mm}\\
	\hspace{-1mm}\check{\sx}^e_i \hspace{-1mm}
	\ea \define&\; (\cX')^{-1}\ba{c}
	\hspace{-1mm}\twd^e_i\hspace{-1mm}\\
	\hspace{-1mm}\tyd^e_i\hspace{-1mm}
	\ea = 
	\ba{c}
	\cL_1\tran \\
	\cL_2\tran \\
	\cX_L
	\ea
	\ba{c}
	\hspace{-1mm}\twd^e_i\hspace{-1mm}\\
	\hspace{-1mm}\tyd^e_i\hspace{-1mm}
	\ea,
}
and
\eq{\label{S-extra}
	\hspace{-2mm}\cS^e_{i-1}\define&(\cX')^{-1}\cT^e_{i-1}\cX'\nnb
	\hspace{-4mm}=&\ba{ccc}
	\hspace{-2mm}\cL_1\tran \cT^e_{i-1}\cR_1 &  \cL_1\tran \cT^e_{i-1}\cR_2 &  \frac{1}{c}\cL_1\tran\cT^e_{i-1}\cX_R \hspace{-2mm}\\
	\hspace{-2mm}\cL_2\tran \cT^e_{i-1}\cR_1&  \cL_2\tran \cT^e_{i-1}\cR_2 &  \frac{1}{c}\cL_2\tran\cT^e_{i-1}\cX_R \hspace{-2mm}\\
	\hspace{-2mm}c\cX_L\cT^e_{i-1}\cR_1 &  c\cX_L\cT^e_{i-1}\cR_2 & \cX_L\cT^e_{i-1}\cX_R \hspace{-2mm}
	\ea.
}
To compute each entry of $\cS^e_{i-1}$, 
%
we let
\eq{
	\cX_R = 
	\ba{cc}
	\cX_{R,u} \\
	\cX_{R,d}
	\ea,
}
where $\cX_{R,u}\in \RR^{NM \times 2(N-1)M}$ and $\cX_{R,d}\in \RR^{NM \times 2(N-1)M}$. For the first line of $\cS^e_{i-1}$, it can be verified that
\eq{\label{S-1st-extra}
	\cL_1\tran \cT^e_{i-1}\cR_1 &=\hspace{-1mm} \mu\overline{\cP}\tran\cH_{i-1}\cI,\\
	\cL_1\tran \cT^e_{i-1}\cR_2 &= \hspace{-1mm} 0,\\
	\frac{1}{c}\cL_1\tran \cT^e_{i-1}\cX_R &= \hspace{-1mm} \frac{\mu}{c} \overline{\cP}\tran \cH_{i\hspace{-0.3mm}-\hspace{-0.3mm}1}\cX_{R,u}.
}
Likewise, noting that 
\eq{
	\cL_2\tran \cT^e_{i-1} = \ba{cc}
	\hspace{-1.5mm}0 \hspace{-1mm}&\hspace{-1mm} \frac{1}{N}\cI\tran\hspace{-1.5mm}
	\ea\hspace{-1.5mm}
	\ba{cc}
	\hspace{-1.5mm}\mu \cH_{i-1} &  0\hspace{-1.5mm} \\
	\hspace{-1.5mm}\mu \cV\cH_{i-1} & 0\hspace{-1.5mm}
	\ea \overset{\eqref{xcn987-2}}{=} \ba{cc}
	\hspace{-1.5mm}0 \hspace{-1mm}&\hspace{-1mm} 0\hspace{-1.5mm}
	\ea,
}
we find for the second line of ${\cS}_{i-1}^e$ that
\eq{\label{S-2nd-extra}
	c\cL_2\tran \cT^e_{i-1}\cR_1 =0,\;\; c\cL_2\tran \cT^e_{i-1}\cR_2 =0,\;\; \cL_2\tran\cT^e_{i-1}\cX_R =0.
}
Substituting \eqref{S-extra}, \eqref{S-1st-extra} and \eqref{S-2nd-extra} into \eqref{extra-recursion-transform-2}, we rewrite \eqref{extra-recursion-transform-2} as
\eq{\label{znhg-extra}\footnotesize
	\hspace{0mm}\ba{c}
	\hspace{-2mm}\bar{\sx}^e_i\hspace{-2mm} \\
	\hspace{-2mm}\widehat{\sx}^e_i\hspace{-2mm} \\
	\hspace{-2mm}\check{\sx}^e_i\hspace{-2mm}
	\ea
	\hspace{-1.2mm}&=\hspace{-1.2mm}\footnotesize
	\ba{ccc}
	\hspace{-2.5mm}I_{\hspace{-0.3mm} M} \hspace{-1.2mm}-\hspace{-1.2mm} \mu \overline{\cP}\tran\hspace{-0.3mm}\cH_{i\hspace{-0.3mm}-\hspace{-0.3mm}1}\cI &  \hspace{-0.8mm}0 &\hspace{-1.3mm} -\frac{\mu}{c}\overline{\cP}\tran\hspace{-0.3mm}\cH_{i\hspace{-0.3mm}-\hspace{-0.3mm}1}\cX_{R,u} \hspace{-2mm}\\
	\hspace{-3mm}0 &  \hspace{-1.3mm}I_{\hspace{-0.3mm} M} &\hspace{-3.3mm} 0 \hspace{-2mm} \\
	\hspace{-2.5mm}-c\cX_L\cT^e_{i-1}\cR_1 &  \hspace{-1.3mm}-c\cX_L\cT^e_{i-1}\cR_2 \hspace{-1.3mm} &  \cD_1 - \cX_L\cT^e_{i-1}\cX_R \hspace{-2mm}
	\ea 
	\hspace{-2mm}\ba{c}
	\hspace{-2mm}\bar{\sx}^e_{i\hspace{-0.5mm}-\hspace{-0.5mm}1}\hspace{-2mm} \\
	\hspace{-2mm}\widehat{\sx}^e_{i\hspace{-0.5mm}-\hspace{-0.5mm}1}\hspace{-2mm} \\
	\hspace{-2mm}\check{\sx}^e_{i\hspace{-0.5mm}-\hspace{-0.5mm}1}\hspace{-2mm}
	\ea
}
From the second line of \eqref{znhg-extra}, we get
\eq{\label{28cn00-extra}
	\widehat{\sx}^e_i = \widehat{\sx}^e_{i-1}.
}
As a result, $\widehat{\sx}^e_i$ will converge to $0$ only if the initial value $\widehat{\sx}^e_{0} = 0$. To verify that, from the definition of $\cL_2$ in \eqref{tsdhb} and \eqref{extra-x-bar and x-check} we have
\eq{\label{hx_0=0-extra}
	\widehat{\sx}_0^e &= \cL_2\tran \ba{c}
	\hspace{-1mm}\twd_0^e\hspace{-1mm}\\
	\hspace{-1mm}\tyd_0^e\hspace{-1mm}
	\ea = \frac{1}{N}\cI\tran \tyd_0^e  \nn
}
\eq{
	&\overset{\eqref{extra-error}}{=}\frac{1}{N}\cI\tran (\sy^\star_o - \sy_0^e) \overset{\eqref{zn-0-extra}}{=} \frac{1}{N}\cI\tran (\sy^\star_o - \cV \sw^e_0).
}
Recall that $\sy_o^\star$ lies in the $\cR(\cV)$, so that $\sy^\star_o - \cV \sw_0$ also lies in $\cR(\cV)$. Recall further from Lemma \ref{lm:null-V} that $\cI\tran \cV =0$, and conclude that $\widehat{\sx}^e_0=0$. Therefore, from \eqref{28cn00-extra} we have 
\eq{\label{xzcn-extra}
	\widehat{\sx}_i^e=0, \quad \forall i\ge 0
}
With \eqref{xzcn-extra}, recursion \eqref{znhg-extra} is equivalent to \eqref{final-recursion-extra}.	

\section{Proof of Theorem \ref{lm-convergence-extra}}
\label{app-conv-extra}
From the first line of recursion \eqref{final-recursion-extra}, we have
\eq{\label{x_bar-recursion-extra}
	\bar{\sx}^e_i =& \left(I_M \hspace{-1mm} - \hspace{-1mm} {\mu \overline{\cP}\tran \cH_{i\hspace{-0.3mm}-\hspace{-0.3mm}1}\cI}\right)\bar{\sx}^e_{i\hspace{-0.3mm}-\hspace{-0.3mm}1} \hspace{-1mm} - \hspace{-1mm}
	\frac{\mu}{c} \overline{\cP}\tran \cH_{i-1}\cX_{R,u} \check{\sx}^e_{i-1}.
} 
Squaring both sides and using Jensen's inequality gives
\eq{\label{x_bar_square-extra}
	\|\bar{\sx}^e_i\|^2=&\ \left\|\left(I_M \hspace{-1mm} - \hspace{-1mm} {\mu \overline{\cP}\tran \cH_{i\hspace{-0.3mm}-\hspace{-0.3mm}1}\cI}\right) \bar{\sx}^e_{i-1} - 
	\frac{\mu}{c} \overline{\cP}\tran\cH_{i-1}\cX_{R,u} \check{\sx}^e_{i-1}\right\|^2 \nnb
	\le &\ \frac{1}{1-t}\left\|I_M \hspace{-1mm} - \hspace{-1mm} {\mu \overline{\cP}\tran \cH_{i\hspace{-0.3mm}-\hspace{-0.3mm}1}\cI}\right\|^2\|\bar{\sx}^e_{i-1}\|^2 \nnb
	&\quad + \frac{1}{tc^2}\|\mu\overline{\cP}\tran\cH_{i-1}\cX_{R,u}\|^2 \|\check{\sx}^e_{i-1}\|^2
}
for any $t\in (0,1)$. For the term ${\mu \overline{\cP}\tran \cH_{i\hspace{-0.3mm}-\hspace{-0.3mm}1}\cI}$, we have
\eq{
	\hspace{-3mm} {\mu \overline{\cP}\tran \cH_{i-1}\cI} = \mu \sum_{k=1}^{N}p_k  H_{k,i-1} \overset{\eqref{H-properties}}{\ge}  \frac{\mu}{N} \nu I_M \define {\sigma^e_{11} \mu I_M},
}
where $\sigma_{11} =\nu/N$. Similarly, we can obtain the upper bound
\eq{
	{\mu \overline{\cP}\tran \cH_{i-1}\cI} & = \mu \sum_{k=1}^{N}p_k H_{k,i-1}\overset{\eqref{H-properties}}{\le} \left(\sum_{k=1}^{N}p_k  \right)\delta\mu  I_M \hspace{-1.5mm} \overset{(a)}{=} \delta\mu  I_M,
}
where equality $(a)$ holds because $\sum_{k=1}^{N}p_k=1$. It is obvious that $\delta >\sigma^e_{11}$. As a result, we have
\eq{\label{xzcn8237-extra}
	(1\hspace{-0.8mm}-\hspace{-0.8mm}\delta \mu)I_M \hspace{-0.8mm}\le\hspace{-0.8mm} I_{M} \hspace{-0.8mm}-\hspace{-0.8mm} {\mu \overline{\cP}\tran\cH_{i-1}\cI} \le (1\hspace{-0.8mm}-\hspace{-0.8mm}\sigma^e_{11}\mu)I_M,
}
which implies that when the step-size is sufficiently small to satisfy  
\eq{\label{xchayu-extra}
	\mu<1/\delta,
} 
it will hold that
\eq{\label{uiox98-extra}
	&\ \left\|I_{M} \hspace{-0.8mm}-\hspace{-0.8mm} {\mu \overline{\cP}\tran\cH_{i-1}\cI} \right\|^2 \le  (1-\sigma^e_{11}\mu_{\max})^2.
}
On the other hand, we have
\eq{\label{bguj87-extra}
	&\frac{1}{c^2}\|\mu\overline{\cP}\tran\cH_{i-1}\cX_{R,u}\|^2\nnb
	&\ \le \frac{\mu^2}{c^2} \|\overline{\cP}\tran\|^2 \|\cH_{i-1}\|^2 \|\cX_{R,u}\|^2 \nnb
	&\ \le \frac{1}{c^2}\left(\sum_{k=1}^{N}p_k^2\right)\delta^2 \|\cX_{R,u}\|^2 \mu^2 \nnb
	&\ =  \frac{\delta^2}{c^2N}\|\cX_{R,u}\|^2 \mu^2  \overset{\eqref{xcn398}}{\le} \frac{\delta^2}{c^2 N}\|\cX_{R}\|^2 \mu^2 \define (\sigma^e_{12})^2 \mu^2,
}
where $\sigma^e_{12}= \delta \|\cX_{R}\|/(c\sqrt{N})$ and the ``$=$" sign in the third line holds because $p_k = 1/N$. Notice that $\sigma^e_{12}$ is independent of $\mu$. Substituting \eqref{uiox98-extra} and \eqref{bguj87-extra} into \eqref{x_bar_square-extra}, we get
\eq{\label{x_bar_square-2-extra}
	&\hspace{-5mm} \|\bar{\sx}^e_i\|^2 \nnb
	\le &\ \frac{1}{1-t}(1-\sigma^e_{11}\mu)^2\|\bar{\sx}^e_{i-1}\|^2 + \frac{1}{t} (\sigma^e_{12})^2\mu^2 \|\check{\sx}^e_{i-1}\|^2 \nnb
	= &\ (1-\sigma^e_{11}\mu)\|\bar{\sx}^e_{i-1}\|^2 + \frac{(\sigma^e_{12})^2}{\sigma^e_{11}}\mu\|\check{\sx}^e_{i-1}\|^2,
}
where we are selecting $t = \sigma^e_{11}\mu$.

Next we check the second line of recursion \eqref{final-recursion-extra}, which amounts to
\eq{
	\check{\sx}^e_i =&\; -c\cX_L\cT^e_{i-1}\cR_1 \bar{\sx}^e_{i-1}+ (\cD_1 - \cX_L\cT^e_{i-1}\cX_R)\check{\sx}^e_{i-1}\nnb
	=&\; \cD_1\check{\sx}^e_{i-1} - \cX_L\cT^e_{i-1}(c\cR_1 \bar{\sx}^e_{i-1}
	+ \cX_R\check{\sx}^e_{i-1}). \label{yujs99-extra}
}
Squaring both sides of \eqref{yujs99-extra}, and using Jensen's inequality again, 
\eq{
	\|\check{\sx}^e_i\|^2 =& \|\cD_1 \check{\sx}^e_{i-1} - \cX_L\cT^e_{i-1}(c\cR_1 \bar{\sx}^e_{i-1}
	+ \cX_R\check{\sx}^e_{i-1})\|^2\nn\\
	\le & \frac{\|\cD_1\|^2}{t} \|\check{\sx}^e_{i-1}\|^2 \hspace{-0.8mm}+\hspace{-0.8mm} \frac{1}{1-t}\| \cX_L\cT^e_{i-1}(c\cR_1 \bar{\sx}^e_{i-1}
	+ \cX_R\check{\sx}^e_{i-1})\|^2	 \nnb
	\le & \frac{\|\cD_1\|^2}{t} \|\check{\sx}^e_{i-1}\|^2 + \frac{2c^2}{1-t} \| \cX_L\cT^e_{i-1}\cR_1\|^2 \|\bar{\sx}^e_{i-1}\|^2\nn\\
	&\hspace{0.63cm}+ \frac{2}{1-t} \|\cX_L\cT^e_{i-1}\cX_R\|^2 \| \check{\sx}^e_{i-1}\|^2.
}
where $t\in(0,1)$. From Lemma \ref{lm-B-decomposition} we have that $\lambda \define \|D_1\| = \sqrt{\lambda_2(\widetilde{A})}<1$. By setting $t=\lambda$, we reach
\eq{\label{yuwe9-extra}
	\|\check{\sx}^e_i\|^2 
	\leq&\ \lambda \|\check{\sx}^e_{i-1}\|^2 + \frac{2c^2}{1-\lambda} \| \cX_L\cT^e_{i-1}\cR_1\|^2 \|\bar{\sx}^e_{i-1}\|^2\nn\\
	&\;\;\;\;\;\hspace{1.03cm}+ \frac{2}{1-\lambda} \|\cX_L\cT^e_{i-1}\cX_R\|^2 \| \check{\sx}^2_{i-1}\|^2.
}
From the definition of $\cT^e_{i-1}$ in \eqref{extra-T-defi}, we have
\eq{\label{xha9867-extra}
	\cT^e_{i-1} \hspace{-1mm}=\hspace{-1mm} \mu
	\ba{cc}
	\hspace{-2mm}\cH_{i-1} & 0\hspace{-2mm} \\
	\hspace{-2mm}\cV  \cH_{i-1} & 0 \hspace{-2mm}
	\ea \hspace{-1mm}=\hspace{-1mm} \mu
	\underbrace{\ba{cc}
		\hspace{-2mm}I_{MN} & 0\hspace{-2mm} \\
		\hspace{-2mm}\cV  & 0\hspace{-2mm}
		\ea}_{\define \cT_e}
	\ba{cc}
	\hspace{-2mm}\cH_{i-1} & 0\hspace{-2mm} \\
	\hspace{-2mm}0 &  \cH_{i-1}\hspace{-2mm}
	\ea,
	%
}
which implies that
\eq{\label{nui89-extra}
	\|\cT^e_{i-1}\|^2 \le \mu^2 \|\cT_e\|^2 \left( \max_{1\le k\le N} \|H_{k,i-1}\|^2 \right) \le  \|\cT_e\|^2 \delta^2 \mu^2.
}
We also emphasize that $\|\cT_e\|^2$ is independent of $\mu$. With inequality \eqref{nui89-extra}, we further have
\eq{
	c^2\| \cX_L\cT^e_{i-1} \cR_1\|^2 &\le c^2 \mu^2 \|\cX_L\|^2 \|\cT_e\|^2 \|\cR_1\|^2\delta^2  \nnb
	&\define (\sigma^e_{21})^2 \mu^2 \label{xcbnb978-1-extra}\\
	\| \cX_L\cT^e_{i-1}\cX_R\|^2  &\le \mu^2 \|\cX_L\|^2 \|\cT_e\|^2 \|\cX_R\|^2 \delta^2 \nnb
	&\define (\sigma^e_{22})^2 \mu^2,\label{xcbnb978-2-extra}
}
notice that $\|\cR_1\|=1$, $\sigma^e_{21}$ and $\sigma^e_{22}$ are defined as
\eq{
	\sigma^e_{21}=c \|\cX_L\| \|\cT_e\| \delta,\ \ \sigma^e_{22}=\|\cX_L\| \|\cT_e\| \|\cX_R\|\delta.
}
With \eqref{xcbnb978-1-extra} and \eqref{xcbnb978-2-extra}, inequality \eqref{yuwe9-extra} becomes 
\eq{\label{yuwe9-2-extra}
	\| \check{\sx}^e_i\|^2 
	\leq \left(\lambda + \frac{2(\sigma_{22}^e)^2 \mu^2}{1-\lambda}\right) \|\check{\sx}^e_{i-1}\|^2 + \frac{2(\sigma^e_{21})^2 \mu^2}{1-\lambda}  \|\bar{\sx}^e_{i-1}\|^2.
}
Combining \eqref{x_bar_square-2-extra} and \eqref{yuwe9-2-extra}, we arrive at the inequality recursion:
\eq{\label{xngyi-extra}
	\hspace{-3mm}
	\ba{c}
	\hspace{-2mm}\|\bar{\sx}^e_i\|^2\hspace{-2mm} \\
	\hspace{-2mm}\|\check{\sx}^e_i\|^2 \hspace{-2mm}
	\ea
	\preceq
	\underbrace{\ba{cc}
		\hspace{-2mm}1-\sigma^e_{11}\mu\hspace{-1mm} & \hspace{-1mm} \frac{(\sigma_{12}^e)^2}{\sigma^e_{11}}\mu\hspace{-2mm}\\
		\hspace{-2mm}\frac{2(\sigma_{21}^e)^2\mu^2}{1-\lambda} \hspace{-1mm}&\hspace{-1mm} \lambda + \frac{2(\sigma_{22}^e)^2 \mu^2}{1-\lambda}\hspace{-2mm}
		\ea}_{\define G_e}
	\ba{c}
	\hspace{-2mm}\|\bar{\sx}^e_{i-1}\|^2\hspace{-2mm} \\
	\hspace{-2mm}\|\check{\sx}^e_{i-1}\|^2 \hspace{-2mm}
	\ea.
}
From this point onwards, we follow exactly the same argument as in \eqref{upper-bound-1}--\eqref{237sdnkk} to arrive at the conclusion in Theorem \ref{lm-convergence-extra}.											

\vspace{-2mm}
\section{Proof of Lemma \ref{lm-sta-ed}}\label{app-sta-ed}
It is observed from expression \eqref{cn9999} for $E_d$ that one of the eigenvalues is $1-\mu\sigma^2$. It is easy to verify that when $\mu$ satisfies \eqref{mu-range}, it holds that
\(
-1< 1 - \mu \sigma^2 < 1.
\)
Next, we check the other two eigenvalues. Let $\theta$ denote a generic eigenvalue of $E_d$. From the right-bottom $2\times 2$ block of $E_d$ in \eqref{cn9999}, we know that $\theta$ will satisfy the following characteristic polynomial
\eq{\label{ed-cp}
\theta^2 - (2-\mu \sigma^2)a\, \theta + (1-\mu \sigma^2) a =0,
}
where $a\in(0,1)$ is a combination weight (see the expression for $A$ in \eqref{2-MSE-A}). Solving \eqref{ed-cp}, the two roots are
\eq{\label{roots}
\theta_{1,2} = \frac{(2\hspace{-0.8mm}-\hspace{-0.8mm}\mu \sigma^2)a \pm \hspace{-0.8mm} \sqrt{(2\hspace{-0.8mm}-\hspace{-0.8mm}\mu\sigma^2)^2 a^2 \hspace{-0.8mm}-\hspace{-0.8mm} 4(1\hspace{-0.8mm}-\hspace{-0.8mm}\mu\sigma^2)a }}{2}. 
}
Let 
\eq{\label{xcnweh8}
\Delta = (2\hspace{-0.8mm}-\hspace{-0.8mm}\mu\sigma^2)^2 a^2 \hspace{-0.8mm}-\hspace{-0.8mm} 4(1\hspace{-0.8mm}-\hspace{-0.8mm}\mu\sigma^2)a.
}
Based on the value of $\mu\sigma^2$ and $a$, $\Delta$ can be negative, zero, or positive. Recall from \eqref{mu-range} that $0<\mu\sigma^2<2$. In that case, over the smaller interval $1\leq \mu\sigma^2< 2$, it holds that $(1-\mu\sigma^2)\geq 0$ and, from \eqref{xcnweh8}, $\Delta >0$. For this reason, as indicated in cases 1 and 2 below, the scenarios corresponding to $\Delta<0$ or $\Delta=0$ can only occur over $0< \mu\sigma^2 < 1$:

\noindent \textbf{Case 1: $\Delta <0$}. It can be verified that when 
\eq{\label{j98}
1-\mu\sigma^2 > 0,\quad \mbox{and}\quad a < \frac{4(1-\mu\sigma^2)}{(2-\mu\sigma^2)^2},
}
it holds that $\Delta < 0$. In this situation, both $\theta_1$ and $\theta_2$ are imaginary numbers with magnitude 
\eq{
|\theta_1| \hspace{-0.8mm}=\hspace{-0.8mm} |\theta_2| \hspace{-0.8mm}=\hspace{-0.8mm} \frac{1}{4}\left( (2\hspace{-0.8mm}-\hspace{-0.8mm}\mu \sigma^2)^2a^2 \hspace{-0.8mm}+\hspace{-0.8mm} (-\hspace{-0.4mm}\Delta) \right) \hspace{-0.8mm}=\hspace{-0.8mm} (1\hspace{-0.8mm}-\hspace{-0.8mm}\mu\sigma^2 ) a < 1,
}
where the last inequality holds because $0< \mu\sigma^2 <1$ (see \eqref{mu-range} and \eqref{j98}) and $a\in (0,1)$.

\noindent \textbf{Case 2: $\Delta =0$}. It can be verified that when 
\eq{\label{bv6}
	1-\mu\sigma^2 > 0,\quad \mbox{and}\quad a = \frac{4(1-\mu\sigma^2)}{(2-\mu\sigma^2)^2},
}
it holds that $\Delta = 0$. In this situation, from \eqref{roots} we have
\eq{
	\theta_1 = \theta_2 = \frac{(2-\mu \sigma^2)a}{2} < 1,
}
where the last inequality holds because $0< \mu\sigma^2 <1$ (see \eqref{mu-range} and \eqref{j98}) and $a\in (0,1)$. Observe further that the upper bound on $a$ in \eqref{j98} is positive and smaller than one when $0<\mu\sigma^2 <1$.

\noindent \textbf{Case 3: $\Delta >0$}. It can be verified that when 
\eq{\label{n2g8}
	1-\mu\sigma^2 > 0,\quad \mbox{and}\quad a > \frac{4(1-\mu\sigma^2)}{(2-\mu\sigma^2)^2},
}
or when $1\le \mu \sigma^2 < 2$,
it holds that $\Delta > 0$. In this situation, $\theta$ is real and 
\eq{
\theta_{1} &= \frac{(2\hspace{-0.8mm}-\hspace{-0.8mm}\mu \sigma^2)a + \hspace{-0.8mm} \sqrt{(2\hspace{-0.8mm}-\hspace{-0.8mm}\mu\sigma^2)^2 a^2 \hspace{-0.8mm}-\hspace{-0.8mm} 4(1\hspace{-0.8mm}-\hspace{-0.8mm}\mu\sigma^2)a }}{2},\\
\theta_{2} &= \frac{(2\hspace{-0.8mm}-\hspace{-0.8mm}\mu \sigma^2)a - \hspace{-0.8mm} \sqrt{(2\hspace{-0.8mm}-\hspace{-0.8mm}\mu\sigma^2)^2 a^2 \hspace{-0.8mm}-\hspace{-0.8mm} 4(1\hspace{-0.8mm}-\hspace{-0.8mm}\mu\sigma^2)a }}{2}. 
}
Moreover, since $(2\hspace{-0.8mm}-\hspace{-0.8mm}\mu \sigma^2)a>0$, we have
\eq{
|\theta_2| < |\theta_1|=\theta_1.
}
We regard $\theta_1$ as a function of $a$, i.e., $\theta_1=f(a)$. It holds that $f(a)$ is monotone increasing with $a$. To prove it, note that
\begin{equation}
f^\prime(a)=\frac{2-\mu\sigma^2}{2} + \frac{2(2-\mu\sigma^2)a - 4(1-\mu\sigma^2)}{4\sqrt{\Delta}}.
\end{equation}
Now since
\eq{\label{xcnwj88}
&\ \Delta = (2\hspace{-0.8mm}-\hspace{-0.8mm}\mu\sigma^2)^2 a^2 \hspace{-0.8mm}-\hspace{-0.8mm} 4(1\hspace{-0.8mm}-\hspace{-0.8mm}\mu\sigma^2)a > 0 \nnb
\Longleftrightarrow &\ (2\hspace{-0.8mm}-\hspace{-0.8mm}\mu\sigma^2)^2 a > 4(1\hspace{-0.8mm}-\hspace{-0.8mm}\mu\sigma^2) \;\;\; \mbox{(because $a>0$)} \nnb
\Longrightarrow &\ 2(2-\mu\sigma^2)a > 4(1\hspace{-0.8mm}-\hspace{-0.8mm}\mu\sigma^2),
}
we conclude that $f^\prime(a)>0$. Since $a<1$, it follows that
\eq{
\theta_1 = f(a) < f(1) = 1.
}
In summary, when $\mu$ satisfies \eqref{mu-range}, for any $a\in (0,1)$ it holds that all three eigenvalues of $E_d$ stay within the unit-circle, which implies that $\rho(E_d)<1$, and also $\rho(\cE_d)<1$. As a result, $\check{\sz}_i$ in \eqref{cngw7} will converge to $0$. Since $\widehat{\sz}_i=0$ for any $i$, we  conclude that $\tzd_i$ converges to $0$.

\section{Proof of Lemma \ref{lm-sta-ex}}\label{app-sta-extra}
Similar to the arguments used to establish Lemma \ref{lm-Q_d-decom} and \eqref{ngy7}--\eqref{cngw7}, the EXTRA error recursion \eqref{special-extra} can also be divided into two separate recursions
\eq{
	\widehat{\sz}_i^e = \widehat{\sz}_{i-1}^e, \quad 
	\mbox{and} \quad
	\check{\sz}_i^e = \cE_e \check{\sz}_{i-1}^e,
}
where $\cE_e = E_e \otimes I_M$, and
\eq{\label{cn9999-extra}
	E_e = 
	\ba{ccc}
	\hspace{-2mm}1-\mu^e\sigma^2 & 0 & 0\hspace{-2mm}\\
	\hspace{-2mm}0 & a-\mu^e \sigma^2 & -\sqrt{2-2a}\hspace{-2mm}\\
	\hspace{-2mm}0 & (a-\mu^e\sigma^2)\sqrt{\frac{1-a}{2}} & a\hspace{-2mm}
	\ea.
}
Also, since both $\sy_0^e$ and $\sy_o^\star$ lie in the $\mathrm{range}(\cV)$, it can be verified that $\widehat{\sz}_0^e=0$. Therefore, we only focus on the convergence of $\check{\sz}_i^e$. Let $\theta^e$ denote a generic eigenvalue of $E_e$. From the right-bottom $2\times 2$ block of $E_e$ in \eqref{cn9999-extra}, we know that $\theta^e$ will satisfy the following characteristic polynomial
\eq{\label{ed-cp-ex}
	(\theta^e)^2 - (2a-\mu^e \sigma^2)\, (\theta^e) + (a-\mu^e \sigma^2) =0.
}
Solving it, we have
\eq{
\theta^e_{1,2}&=\frac{2a-\mu^e\sigma^2 \pm \sqrt{(2a-\mu^e\sigma^2)^2 - 4(a-\mu^e \sigma^2)}}{2}.
}
Now we suppose $\mu^e\sigma^2 \ge a+1$ as noted in \eqref{mu-range-extra}, it then follows that 
\eq{\label{xcn2hg}
	a-\mu^e\sigma^2\le -1
	} 
and hence both $\theta_1^e$ and $\theta_2^e$ are real numbers with
\eq{
\theta^e_{1}=\frac{2a\hspace{-1mm}-\hspace{-1mm}\mu^e\sigma^2 \hspace{-1mm}+\hspace{-1mm} \sqrt{(2a-\mu^e\sigma^2)^2 \hspace{-1mm}+\hspace{-1mm} 4(\mu^e \sigma^2-a)}}{2},
}
\eq{
\theta^e_{2}=\frac{2a\hspace{-1mm}-\hspace{-1mm}\mu^e\sigma^2 \hspace{-1mm}-\hspace{-1mm} \sqrt{(2a-\mu^e\sigma^2)^2 \hspace{-1mm}+\hspace{-1mm} 4(\mu^e \sigma^2-a)}}{2}.
}
Moreover, with $\mu^e\sigma^2 \ge a+1$ we further have  \eq{\label{b88}2a-\mu^e\sigma^2\le a-1<0,} which implies that
\eq{
|\theta_2^e|= \frac{\mu^e\sigma^2\hspace{-1mm}-\hspace{-1mm}2a \hspace{-1mm}+\hspace{-1mm} \sqrt{(2a-\mu^e\sigma^2)^2 \hspace{-1mm}+\hspace{-1mm} 4(\mu^e \sigma^2-a)}}{2} > 1,
}
where the last inequality holds because of \eqref{xcn2hg} and \eqref{b88}. Therefore, when $\mu^e$ is chosen such that
$\mu^e\sigma^2 \ge 1+a$, there always exists one eigenvalue $\theta^e$ such that $|\theta^e|>1$ which implies that $\check{\sz}_i^e$ diverges, and so does $\tzd_i^e$.

\bibliographystyle{IEEEbib}
\bibliography{reference}
\end{document}